\documentclass[oneside,a4paper,10pt,reqno]{amsart}
\usepackage{amsfonts}
\usepackage{amssymb}
\usepackage{amsmath}
\usepackage[height=9in,a4paper,hmargin={2.5cm,2.5cm}]{geometry}
\usepackage{epsfig}
\usepackage{graphicx}
\usepackage{enumitem}
\usepackage{verbatim}
\usepackage[T1]{fontenc}
\usepackage[utf8]{inputenc}
\usepackage{mathtools}
\usepackage{esint}
\usepackage{graphicx}
\usepackage{mathrsfs}
\usepackage{multirow}
\usepackage{amsthm}
\usepackage[english]{babel}
\usepackage{amsfonts}
\usepackage{listings}
\usepackage{caption}
\usepackage{csquotes}
\usepackage{amsmath}
\usepackage{amssymb}
\usepackage{mathtools}
\usepackage{tabularx}
\usepackage{enumitem}
\usepackage{bm}
\usepackage{tikz}
\tikzset{
	mynode/.style={fill,circle,inner sep=1.5pt,outer sep=0pt}
}
\usepackage{pgfplots} 
\numberwithin{equation}{section}
\pgfplotsset{/pgf/number format/use comma,compat=newest}
\theoremstyle{plain}
\newtheorem{thm}{Theorem}[section]

\newtheorem{lem}[thm]{Lemma}
\newtheorem{prop}[thm]{Proposition}
\theoremstyle{definition}
\newtheorem{defn}[thm]{Definition}
\newtheorem{rem}[thm]{Remark}

\newcommand{\R}{\mathbb{R}}

\newcommand{\N}{\mathbb{N}}

\newcommand{\Sf}{\mathbb{S}}

\newcommand{\Int}{\textnormal{Int}}
\newcommand{\dist}{\textnormal{dist}}

\newcommand{\divr}{\textnormal{div}}
\newcommand{\curl}{\textnormal{curl}}
\newcommand{\diam}{\textnormal{diam}}
\newcommand{\lip}{\textnormal{Lip}}
\newcommand{\supp}{\textnormal{Supp} }
\newcommand{\cof}{\textnormal{cof}}

\newcommand{\Kinfty}{|K|_\infty}

\usepackage{tikz}
\usetikzlibrary{decorations.pathreplacing}
\usetikzlibrary{calc}

\definecolor{blue_links}{RGB}{13,0,180} 

\usepackage{hyperref}
\hypersetup{
	colorlinks=true, 
	linktoc=all,     
	linkcolor=blue_links,  
	citecolor=blue_links,
	urlcolor=blue_links,
}

\usepackage{mathtools}

\newcommand{\Om}{\Omega}

\usepackage{color}

\title[Geometric rigidity for incompatible fields in the multi-well case and an application to strain-gradient plasticity]{Geometric rigidity for incompatible fields in the multi-well case and an application to strain-gradient plasticity}

\author[S. Almi]{Stefano Almi}
\address[Stefano Almi]{Dipartimento di Matematica e Applicazioni ``R.~Caccioppoli'',
	Universit\`a di Napoli Federico II, via Cintia, 80126 Napoli, Italy.}
\email{stefano.almi@unina.it}

\author[D. Reggiani]{Dario Reggiani}
\address[Dario Reggiani]{Scuola Superiore Meridionale, Universit\'a di Napoli Federico II
	Largo San Marcellino 10, 80138 Napoli, Italy.}
\email{dario.reggiani@unina.it}

\author[F. Solombrino]{Francesco Solombrino}
\address[Francesco Solombrino]{Dipartimento di Matematica e Applicazioni ``R.~Caccioppoli'',
	Universit\`a di Napoli Federico II, via Cintia, 80126 Napoli, Italy.}
\email{francesco.solombrino@unina.it}

\begin{document}
	
	\maketitle
			\begin{abstract}
	We derive a quantitative rigidity estimate for a multiwell problem in nonlinear elasticity with dislocations. Precisely, we show that the $L^{1^{*}}$-distance of a possibly incompatible strain field $\beta$ from a single well is controlled in terms of the $L^{1^{*}}$-distance from a finite set of wells, of ${\rm curl}\beta$, and of ${\rm div}\beta$. As a consequence, we derive a strain-gradient plasticity model as $\Gamma$-limit of a nonlinear finite dislocation model, containing a singular perturbation term accounting for the divergence of the strain field. This can also be seen as a generalization of the result of \cite{ADMLP} to the case of incompatible vector fields. 
	\end{abstract}

\section{Introduction}

Quantitative geometric rigidity estimates in nonlinear elasticity attracted significant interest after the seminal paper \cite{Friesecke2002ATO}, stating that for a sufficiently smooth bounded domain $\Omega \subset \R^n$, the $L^{2}$-distance of the gradient of a map $u \in H^{1}(\Om; \R^{n})$ from a suitable constant rotation is always controlled, up to a multiplicative factor, by the $L^{2}$-distance of $\nabla u$ from the set of rotations~$SO(n)$. Such result has been later extended to general exponents $p \in (1,\infty)$, see \cite{Conti-SchwCPAM}, to settings of mixed growth \cite{CDM}, as well as to variable exponents $p(x)$ finite, bounded away from~$1$, and satisfying suitable $\log$-H\"older continuity conditions~\cite{Caponietal}.

Remarkable applications of the above quantitative estimates are the rigorous derivation of reduced model of nonlinear elasticity for plates, shells, and rods with different scaling regimes \cite{Bresciani, Bresciani2, Friesecke2002ATO, Friesecke2006, Lewicka, Mora4, Mora2003, Mora2004} and the passage from nonlinear to linearized elasticity, which has been provided in \cite{DalMasoNegriPercivale:02} for quadratic growth, and later extended to subquadratic growth \cite{AgDMDS,Caponietal, Mora-Riva} and to collapsing multiwell energies~\cite{Agostiniani-Blass-Koumatos, Jesenko-Schmidt2, Schmidt_2008}. Further examples of linearization of nonlinearly elastic energies include incompressible materials \cite{Jesenko-Schmidt:20, edo}, pure traction problems \cite{ Maddalena-Percivale-Tomarelli, Mainini-Percivale, Mora-Maor} the passage from atomistic-to-continuum models \cite{Braides-Solci-Vitali:07, KFZ:2022, Schmidt:2009}, plasticity \cite{Ulisse}, thermoviscoelasticity \cite{Rufat, MFMK}, and elastic thin films \cite{KFZ:2021, KrePio19}.

Rigidity estimates have then been studied also beyond elasticity. In \cite{Chambolle-Giacomini-Ponsiglione:2007, KFZ:2021, Friedrich2015AQG} rigidity results have been obtained in the realm of free-discontinuity problems, where deformation may be discontinuous, hence allowing to justify linearization also in the setting of fracture mechanics~\cite{Friedrich2015ADO} (see also~\cite{Almi-Davoli-Friedrich, higherordergriffith}). In \cite{ContiGarrRigidity, lauteri.luckhaus, MuScaZepIncomp, MuScaZep} the authors provided rigidity results for incompatible fields, i.e., for fields that are not curl-free, and therefore can not be interpreted as gradient of a deformation field. In this case, the original estimate~\cite{Friesecke2002ATO} is generalized in the sense that the $L^{1^{*}}$-distance of a matrix field $\beta$ from a suitable rotation is controlled by the $L^{1^{*}}$-distance of~$\beta$ from $SO(n)$ plus the total variation of $\curl \beta$. A paramount application of rigidity estimates for incompatible fields is in the derivation of gradient plasticity models from discrete dislocations models. First examples of linearization processes in such framework appeared in~\cite{MuScaZep, MuScaZepIncomp, ScaZep} 2D for pure edge dislocations. Similar to elasticity, rigidity estimates are crucial for compactness of sequences of incompatible fields~$\beta_{\varepsilon}$ with bounded energy and guarantee convergence of a suitably rescaled and rotated version of~$\beta_{\varepsilon}$. A generalization to the 3D case has been obtained in~\cite{Conti-Garroni-Marziani}, while a discrete-to-continuum analysis has been recently discussed in~\cite{Alicandro-DeLuca-Ponsiglione-Palombaro}.

A common feature to all the above results in elasticity and dislocations is the presence of a single-well elastic energy density, which generally implies that the elastic strain gets closer and closer to the identity or to a fixed rotation in the linearized limit. Multiwell energies are relevant from a physical standpoint, as they may appear in solid-solid phase transition models and in shape-memory alloys ones, where the elastic energy is minimized in several copies of $SO(n)$. Rigidity estimates in this scenario have been obtained, for gradient fields, under suitable incompatibility conditions among the wells~\cite{Chaudhuri, DeLellis}, meaning that the wells do not have mutual rank-one connections. Remarkably, rigidity results in presence of general wells have been recovered first in 2D in~\cite{Lorent} and later on in~\cite{Chermisi-Conti, Jerrard-Lorent} in higher dimension, always under an $L^1$-bounds on second derivatives, an approach which may be related to our results. A further extension dealing with compact connected submanifolds of $\R^{2\times 2}$ can be found in~\cite{Lamy-Lorent}. When dealing with $\Gamma$-convergence, a two-wells problem in solid-solid phase transition was considered in~\cite{Conti-SchwCPAM} in 2D, and later on extended in~\cite{davoli.friedrich, davoli2} to general dimension. In the characterization of linear elastic energies as limit of multi-well nonlinear ones, we remark that the incompatibility assumption of the wells was first removed in~\cite{ADMLP}. There, as it also happened in~\cite{davoli2,  davoli.friedrich} some perturbations terms have to be considered as we discuss below.

The point of view of \cite{ADMLP} is of particular interest for our analysis. The presence of possibly not incompatible wells prevents from the direct use of the rigidity estimates~\cite{Chaudhuri, DeLellis}. Hence, the derivation of a linearly elastic model is obtained by adding a singular perturbation depending on the second derivatives of the deformation field. A generalization of such result to non-conservative vector fields is the one we discuss in this paper in the planar setting, where we show that a model of strain-gradient plasticity can be obtained as $\Gamma$-limit of functionals of the form 
 \begin{equation}
 \label{e:intro1}
 \frac{1}{\varepsilon^{2} |\log\varepsilon|^{2}} \left[\int_{\Omega_{\varepsilon}}W(\beta) \, d x + \eta_{\varepsilon}^{2} \int_{\Omega_{\varepsilon}} |{\rm div}\beta|^{2}\, dx, \right]
 \end{equation}
 accounting for a finite (but diverging) number of dislocation in~$\Om \subseteq \R^{2}$. In~\eqref{e:intro1}, the elastic density $W$ is assumed to attained its minimum in the set $K = SO(n) U_{1} \cup \ldots \cup SO(n) U_{\ell}$ for suitable invertible matrices $U_{i} \in \mathbb{R}^{2\times 2}$, the factor $\varepsilon^{2}|\log \varepsilon|^{2}$ accounts for the energy scaling of $N_{\varepsilon} \sim |\log \varepsilon|$ edge dislocations~\cite{GaLeoPonDislocations, Ginster, MuScaZepIncomp}, $\Omega_{\varepsilon} := \Omega \setminus \bigcup_{i} B_{\varepsilon} (x_{i})$ is the set outside the core-region, where the energy of dislocations is concentrated, and $\eta_{\varepsilon}>0$ is a suitably scaled parameter which tends to $0$ as $\varepsilon \to 0$. As it is customary in the framework of discrete dislocations, to such material defects is associated the measure $\mu= \sum_{i} \varepsilon \xi_{i} \delta_{x_{i}}$, where $\xi_{i}$ denotes the Burgers' vector of the dislocation sitting in~$x_{i}$. Correspondingly, the strain field~$\beta \in L^{2}(\Om; \R^{2\times 2})$ is assumed to be curl-free in~$\Om_{\varepsilon}$ and to satisfy a suitable circulation condition on~$\partial B_{\varepsilon} (x_{i})$, which expresses its compatibility with the system of dislocations~$\mu$. We refer to Section~\ref{sez 4} for the details of the model. 
 
 On the one hand, our asymptotic analysis of~\eqref{e:intro1} can be seen as a generalization of the result of \cite{ADMLP} for $\beta=\nabla u$ (notice that the additional factor $| \log \varepsilon|^{2}$ is not present in the elastic setting). On the other hand, it also extends the theory of \cite{Ginster2, MuScaZep, MuScaZepIncomp, ScaZep}, where the singular perturbation does not appear but the elastic energy has a single well in $SO(n)$.

A fundamental ingredient for compactness issues related to~\eqref{e:intro1} is a multiwell rigidity estimate for incompatible fields in every dimension $n \geq 2$, that we establish in Theorem~\ref{thm: rig est for incompat}. Precisely, we show that there exists $C = C(\Om, K) >0$ such that for every $\beta \in L^{1^*}(\Omega;\R^{n \times n})$ there exists $F \in K$ such that 
\begin{equation}
\label{e:intro2}
		\Vert \beta - F \Vert_{L^{1^*}(\Omega;\R^{n \times n})} \leq C \left( \Vert \dist(\beta,K) \Vert_{L^{1^*}(\Omega)}+|\divr \beta|(\Omega)+|\curl \beta|(\Omega) \right).
\end{equation}
We recall that in the single-well case the additional divergence term does not appear and~\eqref{e:intro2} was proven in \cite{ContiGarrRigidity, MuScaZep}. Our proof follows from the fact that given a cube $Q$ and a field $\beta \in L^{1^*}(Q,\R^{n \times n})$ with $\curl \beta \in \mathcal{M}_b(Q,\R^{n \times n \times n})$ and $\divr \beta \in \mathcal{M}_b(Q,\R^n)$, we can find a decomposition of $\beta$ of the form
\begin{equation*}
	\beta=
	\underbrace{\parbox{50pt}{\centering Y}}_{\parbox{50pt}{\scriptsize\centering $\curl Y=\curl \beta$ \\ $\divr Y=0$}}
	+
	\underbrace{\parbox{50pt}{\centering $\nabla v$}}_{\parbox{50pt}{\scriptsize\centering $\curl \nabla v=0$ \\ $\divr \nabla v=\divr \beta$}}
	+
	\underbrace{\parbox{50pt}{\centering $\nabla w$}}_{\parbox{50pt}{\scriptsize \centering Harmonic}}.
\end{equation*}
The key point is a quantitative rigidity estimate for harmonic functions. The qualitative version of such an estimate is a simple multiwell generalization of a result by Reshetnyak, namely, if a sequence $(\nabla w_j)_j$ with $\nabla w_j$ harmonic functions weakly* converges to a function $\nabla w$ and $\dist(\nabla w_j,K) \to 0$ in measure, then, $\nabla w=F$ for some $F \in K$. Indeed, $\nabla w$ must be harmonic, so that the div-curl Lemma entails strong convergence of $\nabla w_j \to \nabla w$ in $L^2$. Then, $\nabla w \in K$ a.e. which reduces to a single well inclusion, $\nabla w \in K_i$, by continuity and disjointedness of the wells. Hence, the qualitative result follows by Liouville's Theorem. The quantitative estimate for harmonic functions makes use of Caccioppoli inequality, isoperimetric inequality, coarea formula and some arguments from \cite{Chaudhuri} to be first derived in~$L^{1*}$ and in~$L^{1^*,\infty}$. In particular, it is crucial to have this estimate also in the weak $L^{1^*}$ which is the natural space for elliptic regularity estimates with measure data. Indeed, to derive the result for general fields we must combine the estimates for harmonic gradient fields with a regularity result by Bourgain-Brezis-Van Shaftingen (recalled in Theorem \ref{BBT}). The derivation of \eqref{e:intro2} is contained in Theorem \ref{thm: rig est for incompat}.

Exploiting the inequality~\eqref{e:intro2} and by suitable construction, we derive a linearized model of strain plasticity in dimension 2 as limit of~\eqref{e:intro1}. The limit energy writes as
\begin{equation}
\label{e:intro3}
\frac{1}{2} \int_\Omega \mathbb{C}_U \beta : \beta \, dx + \int_\Omega \varphi\left( RU,\frac{d\mu}{d|\mu|} \right) \, d|\mu| 
\end{equation}
and expresses the competition between the usual elastic stored energy due to the strain $\beta \in L^{2}(\Om; \R^{2\times 2})$ and a self-energy due to the dislocations distribution~$\mu \in \mathcal{M} (\Om; \R^{2})$, both obtained as limit of suitably rescaled and rotated versions of $\beta_{\varepsilon}$ and $\mu_{\varepsilon}$ with bounded energy~\eqref{e:intro1} (see Proposition~\ref{prop: compactness} and Definition~\ref{def: convergence} for full details). Furthermore, the matrix $RU$ belongs to $K$, with $R \in SO(n)$ and $U$ in the finite set $\{U_{1}, \ldots, U_{\ell}\}$. Finally, $\mathbb{C}_{U}: = \partial^2 W/\partial F^2(U)$. We refer to Theorem~\ref{thm: gamma limite} for a complete statement of the $\Gamma$-convergence result: observe that the matrix $U$ appearing in the limit is completely characterized by the convergence of rescaled and rotated versions of a sequence $(\beta_\varepsilon)_\varepsilon$ with bounded energy in the sense of Definition \ref{def: convergence}.

A crucial point in our proof strategy is to show that, up to subsequences, a sequence $(\beta_\varepsilon)_\varepsilon$ for which~\eqref{e:intro1} is bounded must be asymptotically close in a suitable sense to one single potential well of $W$. Precisely, we are lead to show that, for suitable rotations $R_\varepsilon$, 
$$
\sup_{\varepsilon >0} \frac{\Vert R_\varepsilon^T \beta_\varepsilon-U \Vert_{L^2}}{\varepsilon |\log \varepsilon|} < +\infty.
$$
This is done in two steps: first, we apply the rigidity estimate \eqref{e:intro2} to find a matrix $V \in K$ such that 
\begin{equation}\label{e:intro4}
\frac{\Vert R_\varepsilon^T \beta_\varepsilon-V\Vert_{L^2}}{\varepsilon |\log \varepsilon|} \leq C/\eta_\varepsilon. 
\end{equation}
In the second step, we decompose $\beta_\varepsilon=\nabla w_\varepsilon + \hat{\beta}_\varepsilon$ so that $\curl \hat{\beta}_\varepsilon= \curl \beta_\varepsilon$ and $\divr \hat{\beta}_\varepsilon=0$. By elliptic estimates we have $\Vert \hat{\beta}_\varepsilon \Vert_{L^{2, \infty}} \leq |\curl \beta_\varepsilon|=O(\varepsilon|\log \varepsilon|)$. As for $w_\varepsilon$, first notice that by the previous bound and \eqref{e:intro4}, we have \footnote{For the reader acquainted with the topic, we stress that we did not investigate the issue of extending the result of \cite{CDM} (recalled in Theorem \ref{thm: CDM rigidity}) in the weak $L^p$ setting to our multiwell case. This would have shortened a bit some proofs in the application we present, but does not follow from a local to global construction. We circumvented the issue by using our extension to incompatible fields, Theorem \ref{thm: rig est for incompat}, instead.} 
\begin{equation*}
	\Vert  \nabla w_\varepsilon-R_\varepsilon^T V \Vert_{L^{2,\infty}} \leq \frac{C \varepsilon |\log \varepsilon|}{\eta_\varepsilon}.
\end{equation*}
With this, since it holds $\Delta w_\varepsilon=\divr \beta_\varepsilon$, a local estimate on $\nabla^2 w_\varepsilon$ can be inferred:
$$
\Vert \nabla^2 w_\varepsilon \Vert_{L^{2 ,\infty}} \leq \frac{C \varepsilon|\log \varepsilon|}{\eta_\varepsilon}.
$$
If $\varepsilon|\log \varepsilon|/\eta_\varepsilon$ is bounded, the above smoothness estimate can be used to show that, up to rotations, there exists $U \in \{U_1,\dots,U_\ell\}$ such that 
\begin{equation*}
	\sup_{\varepsilon>0} \frac{\Vert \nabla w_\varepsilon-U \Vert_{L^{2,\infty}}}{\varepsilon|\log \varepsilon|} < +\infty.
\end{equation*}
To improve this estimate to an $L^2$ one, notice that now we are in a position to use the Brezis-Bourgain-Van Shaftingen inequality (see Theorem \ref{thm: n=2 BB}) in a similar spirit to the approach in \cite[Theorem 3.3]{MuScaZep}.

This is the core to our compactness argument, Proposition \ref{prop: compactness}, and essentially of the $\Gamma-\liminf$ inequality, Proposition \ref{prop: liminf}. We remark that we have to require something more compared to \cite{ADMLP} on the penalization factor $\eta_\varepsilon$, in order to absorb constants due to our localization arguments and to cope with possible infinite number of dislocations. Additional care is also needed in the $\Gamma-\limsup$ construction, indeed, there we essentially follow the footsteps of \cite[Theorem 3]{MuScaZepIncomp}, which is exactly what we need in the special case $\mathbb{C}_U=I$. Otherwise, the recovery sequence has to be conveniently modified in order to let the perturbation term
$
\frac{\eta_\varepsilon^2}{\varepsilon^2 |\log \varepsilon|^2} \int_{\Omega_\varepsilon} |\divr \beta_\varepsilon|^2 \, dx
$
vanish.

\subsection*{Outlook} While the asymptotic analysis of the nonlinear energy~\eqref{e:intro1} has been tackled in the idealized setting, where dislocation lines are assumed to be straight and parallel, essentially reducing to a planar problem, we remark that the multi-well rigidity estimate in Theorem~\ref{thm: rig est for incompat} is stated in general dimension $n \geq 1$, where the natural exponent $p=2$ is replaced by $p=1^{*}$ (see also \cite{ContiGarrRigidity} for the single-well case). In this respect, the analysis of a complete 3-dimensional line-tension model for dislocations under {\em dilute assumptions} in the spirit of~\cite{Conti-Garroni-Marziani, Garroni-Marziani-Scala} will be the subject of future investigation.

\section{Notations and Preliminaries}

\subsection{Notations}

The space of $m \times n$ matrices with real entries is denoted by $\R^{m \times n}$ Given two matrices $A_1$ and $A_2 \in \R^{m \times n}$ their product is denoted by $A_1 \colon A_2$ and the induced norm by $|A|$. With $A^T$ we denote the transpose of $A$. In the case $m=n$, the subspace of symmetric matrices is denoted by $\R^{n \times n}_{\textnormal{sym}}$ while the subspace of skew-symmetric matrices by $\R^{n \times n}_{\textnormal{skew}}$. The identity matrix is denoted by $I \in \R^{n \times n}$.

For a measurable set $E \subset \R^n$ we use $\mathcal{L}^n(E)$ to denote the $n$-dimensional Lebesgue measure of $E$. If $E$ is a set of finite perimeter, we use $P(E)$ to indicate its perimeter (see \cite{Ambrosio2000FunctionsOB}). By $\diam(E)$ we indicate the diameter of the set $E$. The function $\chi_E \colon \R^n \to [0,1]$ is the characteristic function associated to $E$. An open and connected set $\Omega \subset \R^n$ is called domain. Given $\Omega$ Lipschitz domain, we call "isoperimetric constant of $\Omega$" the constant $C_\Omega>0$ such that the relative isoperimetric inequality in $\Omega$ holds for every set $E$ of finite perimeter:
$$
\min \left\{ \mathcal{L}^n(E),\mathcal{L}^n(\Omega \setminus E) \right\} \leq C_\Omega P(E,\Omega).
$$
The ball of radius $r>0$ and centered in $x \in \R^n$ is indicated as $B_r(x)$ or $B(x,r)$ according to the context. If $x=0$ we simply write $B_r$.
By $\mathcal{M}(\Omega;\R^m)$ we indicate the Radon measures on $\Omega$ and $\R^m$-valued. 
We adopt the standard notation for Lebesgue spaces on measurable subsets $E \subseteq \R^n$
and Sobolev spaces on open subsets $\Omega \subseteq \R^n$. According to the context, we use $\Vert \cdot \Vert_{L^p(E)}$ to denote the norm in $L^p(E;\R^m)$ for every $1 \leq p \leq \infty$. A similar convention is also used to denote the norms in Sobolev spaces.

Given $\Omega \subset \R^n$ a measurable set and $p \in [1,\infty)$, we denote by $L^{p,\infty}(\Omega;\R^m)$ the space of weak-$L^p$ functions, that is, $f \in L^{p,\infty}(\Omega;\R^m)$ if and only if $f$ is measurable and there exists a constant $C>0$ such that
$$
\mathcal{L}^n(\{ x \in \Omega \colon |f(x)|>\lambda  \}) \leq \frac{C^p}{\lambda^p}, \qquad \mbox{for every $ \lambda>0$}.
$$
The semi-norm on $L^{p,\infty}$ is defined as
$$
\Vert f \Vert_{L^{p,\infty}(\Omega;\R^m)}:= \inf \{C>0 \colon \lambda^p \mathcal{L}(\{ |f|>\lambda \}) \leq C^p, \ \mbox{for every $\lambda>0$}\}.
$$
It is a semi-norm since the Minkowski Inequality holds in a weaker form. Finally, notice that if $f \in L^p(\Omega;\R^m)$ then, $f \in L^{p,\infty}(\Omega;\R^m)$ and $\Vert f \Vert_{L^{p,\infty}(\Omega;\R^m)} \leq \Vert f \Vert_{L^p(\Omega;\R^m)}$; but $L^{p,\infty}(\Omega;\R^m) \subsetneq L^p(\Omega;\R^m)$.

The partial derivatives with respect to the variable $x_i$ are denoted by $\partial_i$. Given an open subset $\Omega \subset \R^n$ and a function $u \colon \Omega \to \R^m$, we denote its Jacobian matrix by $\nabla u$, whose components are $(\nabla u)_{i,j}=\partial_j u_i$ for $i=1,\dots,m$ and $j=1,\dots,n$. We set $\nabla^2 u:=\nabla(\nabla u)$ and we use $\Delta u$ to denote the Laplacian of $u$, which is defined as $(\Delta u)_i:=\sum_{j=1}^n \delta_{jj} u_i$ for $i=1,\dots,m$. For a function $u \colon \Omega \to \R^n$ we use $Eu$ to denote the symmetric part of the gradient, i.e., $Eu:=\frac{1}{2}(\nabla u+\nabla u^T)$.

Given a matrix field $F \in L^1( \Omega; \R^{n \times n})$ by $\divr F$ we mean its divergence with respect to lines, namely $(\divr F)_i:=\sum_{j=1}^n \partial_j F_{i,j}$ for $i=1,\dots,n$. Finally, with $\curl F$ we denote the tensor valued distribution whose rows are the curl (in a distributional sense) of the rows of $F$; in the sense that $\curl F_{i,j,k}:=\partial_k F_{i,j}-\partial_{i,j} F_k$ for $i,j,k=1,\dots,n$.

We use the convention that constants may change from line to line. We will frequently emphasize
the explicit dependence of the constants on the parameters for the sake of clarity.

Let $n \geq 2$. Let $\ell \in \N$, from now on 
\begin{equation}\label{0}
	K:=\bigcup_{i=1}^\ell K_i,
\end{equation}
where $K_i:=SO(n)U_i$ and $U_i \in \R^{n \times n}$ are invertible matrices for $i=1,\dots,\ell$ such that $U_i U^{-1}_j \notin SO(n)$ for each $i \neq j$. We introduce the notations $\Kinfty:=\max_{F \in K} |F|$, $M:=\max\{ \diam(K),\Kinfty \}$. Finally, for $P \subset \R^{n \times n}$, we denote $d_P(\cdot):=\dist(\cdot,P)$.

\subsection{Preliminaries}

We recall two rigidity estimates for gradient strain fields that we are going to use in the sequel. The first one is proved in \cite[Theorem 3.1]{FJMCPAM} (see also \cite{Conti-SchwCPAM}).

\begin{thm}\label{thm: FJM rigidity}
	Let $\Omega$ be a bounded Lipschitz domain in $\R^n$, $n \geq 2$ and $p \in (1,\infty)$. There exists a constant $C(\Omega,p)>0$ with the following property: for every $u \in W^{1,p}(\Omega;\R^n)$ there exists an associated rotation $R \in SO(n)$ such that
	\begin{equation*}
		\Vert \nabla u -R \Vert_{L^p(\Omega;\R^{n \times n})} \leq C(\Omega,p) \Vert \dist(\nabla u,SO(n)) \Vert_{L^p(\Omega)}.
	\end{equation*}
\end{thm}

The second one is a generalization to weak $L^p$ spaces proven in \cite[Corollary 4.1]{CDM}.

\begin{thm}[$L^{p,\infty}$-rigidity]\label{thm: CDM rigidity}
	Let $\Omega$ be a bounded Lipschitz domain in $\R^n$, $n \geq 2$ and $p \in (1,\infty)$. There exists a constant $C(\Omega,p)>0$ with the following property: for every $u \in W^{1,1}(\Omega;\R^n)$ with $\nabla u \in L^{p,\infty}(\Omega;\R^{n \times n})$, there exists an associated rotation $R \in SO(n)$ such that
	\begin{equation*}
		\Vert \nabla u -R \Vert_{L^{p,\infty}(\Omega;\R^{n \times n})} \leq C(\Omega,p) \Vert \dist(\nabla u,SO(n)) \Vert_{L^{p,\infty}(\Omega)}.
	\end{equation*}
\end{thm}

\section{Multiwell Rigidity Estimates for Gradient Fields}

\subsection{Multiwell Rigidity Estimates for Harmonic Functions}\label{sez 2}

Our first aim is to prove a rigidity result for harmonic functions on a cube. To this aim we introduce a preliminary Lemma which will prove to be useful for several purposes in what follows.

\begin{lem}\label{lem: beginning lemma}

	Let $K_1$ and $K_2$ be two disjoint compact sets in $\R^{n \times n}$ and set $\rho \leq \frac{1}{8} \dist(K_1,K_2)$. Let $\Omega$ be a Lipschitz set and $w \in W^{2,1}(\Omega;\R^n)$. Then, there exists $C>0$ depending only on the isoperimetric constant of $\Omega$ such that, at least for one index $i \in \{1,2\}$, the following inequality holds
	\begin{equation}\label{isopcoarea}
		\mathcal{L}^n(\{d_{K_i}(\nabla w) \leq \rho\})^{\frac{n-1}{n}} \leq \frac{C}{\rho} \int_{\{ \rho<d_{K_i}(\nabla w)<2\rho \}} |\nabla^2 w(x)| \, dx.
	\end{equation}
\end{lem}

\begin{proof}
	By definition of $\rho$ we have that
	\begin{equation}\label{isop cond}
		\min_{i=1,2} \mathcal{L}^n (\{ d_{K_i}(\nabla w) \leq 2\rho \}) \leq \frac{\mathcal{L}^n(\Omega)}{2}.
	\end{equation}
	Assume without loss of generality that $i=2$ is the minimum in \eqref{isop cond}. By isoperimetric inequality, coarea formula and using the fact that $d_{K_2}$ is 1-Lipschitz, we have that there exists $C>0$ depending only on the isoperimetric constant of $\Omega$ such that
	\begin{align*}
		\mathcal{L}^n(\{d_{K_2}(\nabla w) \leq \rho\})^{\frac{n-1}{n}} & \leq \frac{1}{\rho} \int_{\rho}^{2\rho} \mathcal{L}^n(\{d_{K_2}(\nabla w) \leq t\})^{\frac{n-1}{n}} \, dt \\
		& \leq \frac{C}{\rho} \int_{\rho}^{2\rho} P(\{ d_{K_2}(\nabla w) \leq t \},\Omega) \, dt \\
		& \leq \frac{C}{\rho} \int_{\{ \rho<d_{K_2}(\nabla w)<2\rho \}} |\nabla (d_{K_2}(\nabla w))| \, dx \\
		& \leq \frac{C}{\rho} \int_{\{ \rho<d_{K_2}(\nabla w)<2\rho \}} |\nabla^2 w(x)| \, dx.
	\end{align*}
	Hence, \eqref{isopcoarea} is proven.
\end{proof}

\begin{prop}\label{prop: harmonic rigidity cube}
	Consider $K$ as in \eqref{0}. Let $Q$ be the unit cube and $Q' \Subset Q$ be a smaller cube concentric to $Q$ with side length bigger or equal than $1/2$. Set $\delta:=\frac{1}{4}\dist(Q',\partial Q)$. Then, for every $\frac{n}{n-1} \leq p <\infty$ there exists a constant $C(n,\delta,K,p)>0$ with the property that for every $w \in W^{1,p}(Q;\R^n)$ harmonic on $Q$ there exists an associated $F \in K$ such that
	\begin{equation}\label{rigidity of harmo on cube}
		\Vert \nabla w -F \Vert_{L^p(Q';\R^{n \times n})} \leq C(n,\delta,K,p) \, \Vert \dist(\nabla w ,K) \Vert_{L^p(Q)}.
	\end{equation}
\end{prop}

\begin{proof}
	We first consider the case $K=K_1 \cup K_2$ where $K_1$ and $K_2$ are disjoint compact sets of $\R^{n \times n}$. We set $\rho:=\frac{1}{8} \dist(K_1,K_2)$ and assume without loss of generality that 
	$$
	\mathcal{L}^n (\{ d_{K_2}(\nabla w) \leq 2\rho \} \cap Q') \leq \frac{\mathcal{L}^n(Q')}{2}.
	$$
	We claim that there exists $C=C(n,\delta,K,p)>0$ such that
	\begin{equation}\label{1}
		\int_{Q'} d^p_{K_1}(\nabla w(x)) \, dx \leq C \int_Q d^p_K(\nabla w(x)) \, dx.
	\end{equation}
	Set $\varepsilon:= \Vert \dist(\nabla w ,K) \Vert_{L^p(Q)}$. From now on, $C$ will denote a positive constant depending only on $n$, $\delta$, $K$ and $p$. 
	
	First of all, notice that on the set $\{ d_{K_2}(\nabla w) \leq 3M \}$ we have $d_{K_1}(\nabla w) \leq 4M$ by definition of $M$. Moreover, if $|\nabla w(x)| \geq 2M$, $d_{K_1}(\nabla w(x)) \leq 2 d_K(\nabla w(x))$. By definition of $M$ it clearly holds that either $|\nabla w(x)| \geq 2M$ or $d_{K_2}(\nabla w(x)) \leq 3M$. Therefore, we estimate
	\begin{align*}
		\int_{Q'} d^p_{K_1} (\nabla w) \, dx & \leq \int_{\{ |\nabla w| \geq 2M \} \cap Q'} d^p_{K_1} (\nabla w) \, dx + \int_{\{ d_{K_2}(\nabla w) \leq 3M \} \cap Q'} d^p_{K_1} (\nabla w) \, dx \\
		& \leq 2^p \int_{Q'} d^p_{K} (\nabla w) \, dx+(4M)^p \mathcal{L}^n( \{ d_{K_2}(\nabla w) \leq 3M \} \cap Q') \\
		& \leq 2^p \int_{Q'} d^p_{K} (\nabla w) \, dx + \frac{(4M)^p}{\rho^d} \int_{Q'} d^p_{K} (\nabla w) \, dx+ (4M)^p \mathcal{L}^n(\{ d_{K_2}(\nabla w) \leq \rho \} \cap Q').
	\end{align*}
	Hence, it is enough to prove a bound of the type $\mathcal{L}^n(\{ d_{K_2}(\nabla w) \leq \rho \} \cap Q') \leq C\varepsilon^p$.
	
	Since $w$ is harmonic on $Q$, we have that there exists $C(n,\delta,p)>0$ such that
	\begin{equation}\label{cacciop est}
		\int_{Q'} |\nabla^2 w|^p \, dx \leq C(n,\delta,p) \int_Q |\nabla w|^p \, dx \leq C(n,\delta,p) \int_Q (\Kinfty^p+d^p_K(\nabla w)) \, dx.
	\end{equation}
	Thus, using Lemma \ref{lem: beginning lemma} and \eqref{cacciop est}, we estimate
	\begin{align*}
		\mathcal{L}^n(\{ d_{K_2}(\nabla w) \leq \rho \} \cap Q') & \leq C \left( \int_{\{ \rho<d_{K_2}(\nabla w)<2\rho \} \cap Q'} |\nabla^2 w| \, dx \right)^{\frac{n}{n-1}} \\
		& \leq C \left( \mathcal{L}^n\left(\{ \rho<d_{K_2}(\nabla w)<2\rho \} \cap Q'\right)^{\frac{p-1}{p}} \left( \int_{Q'} |\nabla^2 w|^p \right)^{\frac{1}{p}} \, dx \right)^{\frac{n}{n-1}} \\
		& \leq C \left( \mathcal{L}^n\left(\{ \rho<d_{K_2}(\nabla w)<2\rho \} \cap Q'\right)^{\frac{p-1}{p}} \left( \int_Q |\nabla w|^p \right)^{\frac{1}{p}} \, dx \right)^{\frac{n}{n-1}} \\
		& \leq \vphantom{\int} C\mathcal{L}^n\left(\{ \rho<d_{K_2}(\nabla w)<2\rho \} \cap Q'\right)^{\frac{n(p-1)}{p(n-1)}} + C\varepsilon^{\frac{n}{n-1}}.
	\end{align*}
	When $\varepsilon \ll 1$, we can use Chebychev inequality to further estimate 
	$$
	C\mathcal{L}^n\left(\{ \rho<d_{K_2}(\nabla w)<2\rho \} \cap Q'\right)^{\frac{n(p-1)}{p(n-1)}} + C\varepsilon^{\frac{n}{n-1}} \leq C\varepsilon^{\frac{n(p-1)}{(n-1)}} +C \varepsilon^{\frac{n}{n-1}}.
	$$
	Hence, there exists $\varepsilon_0$ depending on $n$, $\delta$, $K$ and $p$ such that, 
	\begin{equation*}
		\mathcal{L}^n(\{ d_{K_2}(\nabla w) \leq \rho \} \cap Q') \leq
		\begin{cases*}
			 C\varepsilon^p & if $\varepsilon \geq \varepsilon_0$, \\
			 \displaystyle \frac{\omega_n \delta^n}{4^{n+1}} & if $\varepsilon \leq \varepsilon_0$.
		\end{cases*}
	\end{equation*}
	Therefore, if $\varepsilon \geq \varepsilon_0$ we conclude.
	Otherwise, by definition of $\delta$, we have that we can find $\frac{\delta}{2}  \leq r \leq \delta$ such that for every $x \in \{ d_{K_2}(\nabla w) \leq \rho \} \cap Q'$ it holds $B(x,2r) \subset Q$ and
	\begin{equation}\label{2}
	\frac{\mathcal{L}^n(\{ d_{K_2}(\nabla w) \leq \rho \} \cap Q' \cap B(x,r))}{\omega_n r^n} \leq \frac{1}{2^{n+1}}.
	\end{equation}
	
	Let
	\begin{equation}\label{A}
	A:= \left\{ x \in Q' \colon \fint_{B(x,2r)} | \nabla w(y)| \, dy>4M \right\}.
	\end{equation}
	We now show that 
	\begin{equation}\label{2.5}
	\mathcal{L}^n(A) \leq C \varepsilon^p.
	\end{equation}
	Indeed, observe that for every $y \in Q$ by definition of $M$ the following inequality holds 
	$
	|\nabla w(y)| \leq (|\nabla w(y)|-2M)^+ + 2M.
	$
	Hence, for every $x \in A$ we have
	$$
	\fint_{B(x,2r)} (|\nabla w(y)|-2M)^+ \, dy \geq 2M 
	$$
	and, using Jensen inequality
	$$
	\fint_{B(x,2r)} \left((|\nabla w(y)|-2M)^+\right)^p \, dy \geq 2^p M^p.
	$$
	Since for every $y \in Q$ we also have $\left((|\nabla w(y)|-2M)^+\right)^p \leq 2^p \ \dist^p(\nabla w(y),K)$, we infer
	$$
	A \subseteq \left\{ x \in Q' \colon \fint_{B(x,2r)} d^p_K(\nabla w(y)) \, dy>  M^p \right\}.
	$$
	Setting $g:=\frac{1}{\omega_n (2r)^n}\chi_{B(0,2r)}$, by Young's convolution inequality we have
	\begin{equation}\label{young}
	\bigg\Vert \fint_{B(\cdot,2r)} d^p_K(\nabla w(y)) \, dy \bigg\Vert_{L^1(Q')} \leq \Vert d_K^p(\nabla w) \Vert_{L^1(Q)} \Vert g \Vert_{L^1(\R^n)}=\varepsilon^p.
	\end{equation}
	Finally, using Chebychev inequality, we obtain \eqref{2.5}.
	
	Let now $x \in \left(\{ d_{K_2}(\nabla w) \leq \rho/2 \} \cap Q'\right) \setminus A $. Using the fact that $w$ is harmonic on $Q$ we have
	\begin{equation}\label{3}
		\Vert \nabla^2 w \Vert_{L^\infty(B(x,r))} \leq C(n) \frac{1}{r^n} \Vert \nabla w \Vert_{L^1(B(x,2r))}=C(n) \fint_{B(x,2r)} |\nabla w| \, dy.
	\end{equation} 
	Observe that $\mathcal{L}^n(Q' \cap B(x,r)) \geq \omega_n r^n /2^n$ since $x \in Q'$ and the side length of $Q'$ is bigger than $2r$. Hence, by virtue of \eqref{2} we can apply Lemma \ref{lem: beginning lemma} on $Q' \cap B(x,r)$ with $2\rho$ replaced by $\rho$, obtaining 
	\begin{align*}
		\mathcal{L}^n\left(\{ d_{K_2}(\nabla w) \leq \rho/2 \} \cap Q' \cap B(x,r)\right) & \leq \left(\frac{C_x}{\rho}  \int_{\{ \rho/2<d_{K_2}(\nabla w)<\rho \} \cap Q' \cap B(x,r)} |\nabla^2 w(x)| \, dx \right)^{\frac{n}{n-1}}\\
		& \leq  \left(\frac{C_x}{\rho} \mathcal{L}^n(\{ \rho/2<d_{K_2}(\nabla w)<\rho \} \cap Q' \cap B(x,r)) \Vert \nabla^2 w \Vert_{L^\infty(B(x,r))} \right)^{\frac{n}{n-1}},
	\end{align*}
	where $C_x>0$ is the isoperimetric constant of $Q' \cap B(x,r)$. Notice that since $x \in Q'$ and the side length of $Q'$ is bigger than $2r$, the constant $C_x$ is bounded uniformly for every $x \in Q'$. In turn, \eqref{3}, the fact that $x \notin A$ and $r \leq \delta \leq 1$ imply
	$$
	\mathcal{L}^n\left(\{ d_{K_2}(\nabla w) \leq \rho/2 \} \cap Q' \cap B(x,r)\right) \leq C  \mathcal{L}^n(\{ \rho/2<d_{K_2}(\nabla w)<\rho \} \cap Q' \cap B(x,r)).
	$$
	Using Besicovitch theorem, we can find a sub-covering of the set $(\{ d_{K_2}(\nabla w) \leq \rho/2 \} \cap Q') \setminus A$ from the covering $\{  B(x,r) \}_{x \in (\{ d_{K_2}(\nabla w) \leq \rho/2 \} \cap Q') \setminus A}$ with overlapping at most $\xi$. Therefore,
	\begin{align}
		\begin{split}\label{xx}
		\mathcal{L}^n\left((\{ d_{K_2}(\nabla w) \leq \rho/2 \} \cap Q') \setminus A \right) & \leq \xi C  \mathcal{L}^n(\{ \rho/2<d_{K_2}(\nabla w)<\rho \} \cap Q') \\
		& \leq \frac{C}{\rho^p} \int_{Q'} d^p_K(\nabla w) \, dx \leq  C \varepsilon^p.
		\end{split}
	\end{align}
	Combining with \eqref{2.5}, this proves the claim \eqref{1}.
	
	Now we consider $K$ as in the statement. Set $\rho:=\frac{1}{8} \min_{i,j} \dist(K_i,K_j)$. We have two possibilities: there exists a single well (without loss of generality) $K_1$ such that 
	$
	\mathcal{L}^n(\{ d_{K \setminus K_1}(\nabla w) \leq 2 \rho \} \cap Q') \leq \mathcal{L}^n(Q')/2;
	$
	or, for every $i=1,\dots,\ell$ we have $\mathcal{L}^n(\{ d_{K_i}(\nabla w) \leq 2 \rho \} \cap Q') \leq \mathcal{L}^n(Q')/2$. In the first case, we proceed exactly as above, with the two disjoints sets $K_1$ and $K \setminus K_1$. In the second case we can argue that the following estimate holds
	$$
	\int_{Q'} d_{K_1}^p(\nabla w) \, dx \leq C \int_{Q'} d^p_K(\nabla w) \, dx + C \sum_{i=2}^\ell \mathcal{L}^n(\{ d_{K_i}(\nabla w) \leq \rho/2 \} \cap Q').
	$$
	In turn, proceeding exactly as above, for every $i=2,\dots,\ell$ we infer 
	$$ 
	\mathcal{L}^n(\{ d_{K_i}(\nabla w) \leq \rho/2 \} \cap Q') \leq C \int_{Q} d^p_K(\nabla w) \, dx.
	$$
	Hence, we have proven \eqref{1} for the multiwell $K$ of the statement. Finally, since $K_1=SO(n)U_1$ where $U_1 \in \R^{n \times n}$ is an invertible matrix, the thesis follows by Theorem \ref{thm: FJM rigidity}.
\end{proof}

\begin{rem}\label{rem: constant scaling invariant}
	The constants $C>0$ appearing in Propositions \ref{prop: harmonic rigidity cube} and \ref{prop: weak harmonic rigidity} are invariant under uniform scaling and translation of the cubes $Q$ and $Q'$. The same holds also for the constant in Proposition \ref{prop: weak harmonic rigidity}.
\end{rem}

In order to get an estimate on the whole domain, we need to consider a Whitney covering for a Lipschitz domain.

\begin{prop}\label{prop: cubic decomp}
	Let $\Omega \subset \R^n$ be a bounded Lipschitz domain. There exists a constant $N$ depending only on the dimension $n$ and a collection $\mathcal{F}:=\{ \hat{Q}_1,\hat{Q}_2,\dots \}$ of closed cubes, whose sides are parallel to the axes and having disjoint interior so that
	\begin{enumerate} [label=\textnormal{(}\roman*\textnormal{)},ref=\roman*]
		\item \label{p1} $\Omega= \cup_k \hat{Q}_k$,
		\item \label{p2} $\textnormal{diam}\hat{Q}_k \leq \dist(\hat{Q}_k,\partial \Omega) \leq 4\textnormal{diam}\hat{Q}_k$,
		\item \label{p3} each point $x \in \Omega$ is contained in at most $N$ enlarged concentric cubes $Q_k$, where, given $x_k$ the centre of $\hat{Q}_k$, $Q_k:=x_k+2(\hat{Q}_k-x_k)$.
	\end{enumerate}
\end{prop}

We are now in a position to prove the rigidity result for harmonic functions on Lipschitz domains. 

\begin{thm}\label{thm: rigidity harmonic}
	Let $\Omega \subset \R^n$ be a bounded Lipschitz domain and consider $K$ as in \eqref{0}. For every $\frac{n}{n-1} \leq p <\infty$, there exists a constant $C(\Omega,K,p)>0$ with the property that for every function $w \in W^{1,p}(\Omega;\R^n)$ which is harmonic on $\Omega$, there exists an associated $F \in K$, such that
	\begin{equation}\label{harmonic rigidity}
		\Vert \nabla w -F \Vert_{L^p(\Omega;\R^{n \times n})} \leq C(\Omega,K,p) \, \Vert \dist(\nabla w ,K) \Vert_{L^p(\Omega)}.
	\end{equation}
\end{thm}

\begin{proof}
	Fix one of the cubes in the decomposition of $\Omega$ given by Proposition \ref{prop: cubic decomp}, $\hat{Q}:=\Int(\hat{Q}_k)=\overline{x}+\left( -\frac{r}{2},\frac{r}{2}\right)$, denote $Q:=\overline{x}+2 (Q-\overline{x})$ the enlarged cube. Using \eqref{p2} of Proposition \ref{prop: cubic decomp}, we have that the enlarged cube $Q$ is still contained in $\Omega$. Now we use the local estimate of Proposition \ref{prop: harmonic rigidity cube} with $Q'=\hat{Q}$ together with Remark \ref{rem: constant scaling invariant}, to obtain the existence of $F_Q \in K$ such that
	\begin{equation}\label{8}
		\int_{\hat{Q}} |\nabla w -F_Q|^p \, dx \leq C(K,n,p) \int_{Q} \dist^p(\nabla w,K) \, dx.
	\end{equation}
	Since $w$ is harmonic on $Q_k$, there exists $C(n,p)>0$ such that
	\begin{equation} \label{9}
		r^p \int_{\hat{Q}} |\nabla^2 w|^p \, dx \leq 2^p C(n,p) \min_{A \in \R^{n \times n}} \int_{Q} |\nabla w-A|^p \, dx.
	\end{equation}
	In turn, \eqref{8}, \eqref{9} and assertion \eqref{p2} of Proposition \ref{prop: cubic decomp} entail
	\begin{align*}
		\int_{\hat{Q}_k} |\nabla^2 w|^p \dist^p(x,\partial \Omega) \, dx \leq C(K,n,p) \int_{Q_k} \dist^p(\nabla w,K) \, dx.
	\end{align*}
	Summing over $k$ and using assertion \eqref{p3} of Proposition \ref{prop: cubic decomp}, gives
	\begin{equation}\label{10}
		\int_\Omega |\nabla^2 w|^p \dist^p(x,\partial \Omega) \, dx \leq C(K,n,p) \int_\Omega \dist^p(\nabla w,K) \, dx.
	\end{equation}
	Finally, we apply a weighted Poincaré inequality of the type
	\begin{equation}\label{11}
		\min_{A \in \R^{n \times n}} \int_\Omega |f(x)-A|^p \, dx \leq C(\Omega) \int_\Omega |\nabla f|^p \dist^p(x,\partial \Omega) \, dx,
	\end{equation}
	valid for $f \in W^{1,p}_{\textnormal{loc}}(\Omega,\R^{n \times n})$, which can be deduced from \cite{FJMCPAM} for the case $p=2$ or from \cite{BoasStraube} for $p \in [1,\infty)$.
	
	Applying inequality \eqref{11} to \eqref{10} we get that there exists $F \in \R^{n \times n}$ such that
	\begin{equation}\label{12}
		\int_\Omega |\nabla w-F|^p \, dx \leq C(\Omega,K,p) \int_\Omega \dist^p(\nabla w,K) \, dx.
	\end{equation} 
	If $F \in K$ we are done. Otherwise, Let $F' \in K$ such that $|F-F'|=\dist(F,K)$. Notice that by \eqref{12} follows that 
	$$
	\mathcal{L}^n(\Omega) |F-F'|^p \leq C(\Omega,k,p) \int_\Omega \dist^p(\nabla w,K) \, dx.
	$$
	Therefore, we conclude.
\end{proof}

We also need a local rigidity result for harmonic functions in weak Lebesgue spaces.
To this aim, we first recall a Caccioppoli-type estimate in $L^{p,\infty}$.

\begin{lem}\label{weak Caccio est}
	Let $\Omega \subset \R^n$ be a Lipschitz set and let $\varphi \in L^{p,\infty}(\Omega)$ be an harmonic function on $\Omega$ with $p \in (1,\infty)$. Let $\Omega' \Subset \Omega$ be an open set and define $\delta:=\dist(Q',\partial Q)$. Then, there exists $C=C(\Omega,p)>0$ such that
	$$
	\Vert \nabla \varphi \Vert_{L^{p,\infty}(\Omega')} \leq \frac{C}{\delta} \Vert \varphi \Vert_{L^{p,\infty}(\Omega)}.
	$$
\end{lem}

\begin{proof}
	The proof is straightforward and follows by using the mean value property for harmonic functions and a standard Caccioppoli-type estimate.
\end{proof}

\begin{prop}\label{prop: weak harmonic rigidity}
	Consider $K$ as in \eqref{0}. Let $Q$ be the unit cube and $Q' \Subset Q$ be a smaller cube concentric to $Q$ with side length bigger or equal than $1/2$. Set $\delta:=\frac{1}{4}\dist(Q',\partial Q)$. We denote $q:=\frac{n}{n-1}$. Then, there exists a constant $C(n,\delta,K)>0$ with the property that for every $w \in W^{1,1}(Q;\R^n)$ with $\nabla w \in L^{q,\infty}(Q;\R^{n \times n})$ and harmonic on $Q$, there exists an associated $F \in K$, such that
	\begin{equation}\label{weak rigidity of harmo on cube}
		\Vert \nabla w -F \Vert_{L^{q,\infty}(Q';\R^{n \times n})} \leq C(n,\delta,K) \, \Vert \dist(\nabla w ,K) \Vert_{L^{q,\infty}(Q)}.
	\end{equation}
\end{prop}

\begin{proof}
	We first consider $K=K_1 \cup K_2$ where $K_1$ and $K_2$ are two disjoint compact sets of $\R^{n \times n}$ and let $\rho:=\frac{1}{8}\dist(K_1,K_2)$. Assume that
	\begin{equation}\label{4}
		\mathcal{L}^n(\{ d_{K_2}(\nabla w) \leq 2\rho \} \cap Q') \leq \frac{\mathcal{L}^n(Q')}{2}.
	\end{equation}
	We claim that there exists $C=C(n,\delta,K)>0$ such that
	\begin{equation}\label{5}
		\lambda^q \mathcal{L}^n(\{ d_{K_1}(\nabla w) >\lambda \} \cap Q') \leq C \Vert d_K(\nabla w) \Vert_{L^{q,\infty}(Q)}^q, \qquad \mbox{for every $\lambda >0$}.
	\end{equation}
	Let us set $\varepsilon:=\Vert \dist(\nabla w,K) \Vert_{L^{q,\infty}(Q)}$.
	
	Observe that on the set $\{ d_{K_2}(\nabla w) \leq 2M \}$ we have $d_{K_1}(\nabla w) \leq 3M$ by definition of $M$. Moreover, if $d_{K_1}(\nabla w) > 2M$ then, $d_K(\nabla w) > M$ and $d_{K_1}(\nabla w) \leq 2d_K(\nabla w)$. Thanks to these considerations, for every $\lambda \geq 2M$ we have
	$$
	\lambda \mathcal{L}^n(\{ d_{K_1}(\nabla w)>\lambda \} \cap Q')^{1/q} \leq \lambda \mathcal{L}^n(\{ d_{K}(\nabla w)>\lambda/2 \})^{1/q} \leq 2\varepsilon.
	$$
	On the other hand, for every $\rho \leq \lambda \leq 2M$ we have
	\begin{align*}
		\{ d_{K_1}(\nabla w)> \lambda \} & \subseteq \{ d_K(\nabla w) > \lambda \} \cup \{ d_{K_2}(\nabla w) \leq \lambda \} \\
		& \subseteq \{ d_K(\nabla w) > \lambda \} \cup \{\rho< d_{K_2}(\nabla w) \leq \lambda \} \cup \{ d_{K_2}(\nabla w) \leq \rho \} \\
		& \subseteq \{ d_K(\nabla w) > \rho \} \cup \{ d_{K_2} (\nabla w) \leq \rho \},
	\end{align*}
	hence,
	\begin{align*}
		\lambda^q \mathcal{L}^n(\{ d_{K_1}(\nabla w) > \lambda \} \cap Q') & \leq (2M)^q \left( \mathcal{L}^n(\{ d_K(\nabla w) > \rho \} \cap Q') +(2M)^q \mathcal{L}^n(\{d_{K_2}(\nabla w) \leq \rho \} \cap Q') \right) \\
		& \leq \left(\frac{2M}{\rho}\right)^q \varepsilon^q+(2M)^q\mathcal{L}^n(\{d_{K_2}(\nabla w) \leq \rho \} \cap Q').
	\end{align*}
	Finally, if $0<\lambda < \rho$, observe that
	\begin{align*}
		\lambda^q \mathcal{L}^n(\{ d_{K_1}(\nabla w) > \lambda \} \cap Q') & \leq \lambda^q \left( \mathcal{L}^n(\{ d_K(\nabla w) > \lambda \} \cap Q') +\lambda^q \mathcal{L}^n(\{d_{K_2}(\nabla w) \leq \rho \} \cap Q') \right) \\
		& \leq \varepsilon^q+\rho^q\mathcal{L}^n(\{d_{K_2}(\nabla w) \leq \rho \} \cap Q').
	\end{align*} 
	Thus, as before it is enough to show a bound of the type $\mathcal{L}^n(\{d_{K_2}(\nabla w) \leq \rho \}) \leq C \varepsilon^q$.
	
	Since $w$ is harmonic on $Q$, by Lemma \ref{weak Caccio est} there exists $C(n,\delta)>0$ such that
	\begin{equation}\label{wcacciop est}
		\Vert \nabla^2 w \Vert_{L^{q,\infty}(Q')} \leq C(n,\delta) \Vert \nabla w \Vert_{L^{q,\infty}(Q)} \leq C(n,\delta)(\Kinfty+\varepsilon).
	\end{equation}
	Therefore, using Lemma \ref{lem: beginning lemma}, H\"{o}lder inequality and \eqref{wcacciop est}, we estimate
	\begin{align*}
		\mathcal{L}^n(\{d_{K_2}(\nabla w) \leq \rho\} \cap Q')^{1/q} & \leq C \int_{\{ \rho<d_{K_2}(\nabla w)<2\rho \} \cap Q'} |\nabla^2 w| \, dx \\
		& \leq C \mathcal{L}^n(\{ \rho<d_{K_2}(\nabla w) <2\rho \})^{1/n} \Vert \nabla^2 w \Vert_{L^{q,\infty}(Q')} \\
		& \leq C \mathcal{L}^n(\{ \rho<d_{K_2}(\nabla w) <2\rho \})^{1/n} (\Kinfty+\varepsilon).
	\end{align*}
	As the last term in the chain of inequalities can be estimated with $\varepsilon^{\frac{1}{n-1}}+\varepsilon$ when $\varepsilon \ll 1$, we infer the existence of $\varepsilon_0>0$ depending only on $n$, $\delta$ and $K$ such that
	\begin{equation*}
		\mathcal{L}^n(\{ d_{K_2}(\nabla w) \leq \rho \} \cap Q') \leq
		\begin{cases*}
			C\varepsilon^q & if $\varepsilon \geq \varepsilon_0$, \\
			\displaystyle \frac{\omega_n \delta^n}{4^{n+1}} & if $\varepsilon \leq \varepsilon_0$.
		\end{cases*}
	\end{equation*}
	Therefore, if $\varepsilon \geq \varepsilon_0$ we conclude. In the remaining case we can argue exactly as in the proof of Proposition \ref{prop: harmonic rigidity cube}, replacing \eqref{young} with
%
%
	\begin{equation}\label{7}
		\bigg\Vert \fint_{B(\cdot,2r)} d_K(\nabla w) \, dy \bigg\Vert_{L^{q,\infty}(Q')} \leq \Vert d_K(\nabla w) \Vert_{L^{q,\infty}(Q)} \Vert g \Vert_{L^1(\R^n)} = \varepsilon,
	\end{equation}
	which holds by \cite[Theorem 1.2.13]{Grafakos}. 
	Hence, we get by definition of $L^{q,\infty}$ norm that $M^q \mathcal{L}^n(A) \leq \varepsilon^q$, where $A$ is defined as in \eqref{A}.
	The estimate
	\begin{align*}
		\mathcal{L}^n\left((\{ d_{K_2}(\nabla w) \leq \rho/2 \} \cap Q') \setminus A \right) \leq  C \varepsilon^q,
	\end{align*}
	can be deduced arguing as in the proof of \eqref{xx}.
	Hence, the claim \eqref{5} is proven.
	
	Now we consider $K$ as in the statement. Arguing exactly as in the final step of Proposition \ref{prop: harmonic rigidity cube}, we can show that \eqref{5} still holds. Since $K_1=SO(n)U$ where $U \in \R^{n \times n}$ is an invertible matrix, the thesis follows by the weak rigidity estimate Theorem \ref{thm: CDM rigidity}. 
\end{proof}

\subsection[Multiwell Rigidity and Gradient Fields]{Multiwell Rigidity and Gradient Fields}

As we are going to prove the rigidity result for the critical exponent $\frac{n}{n-1}$ we need a few more additional results. 

\begin{lem}\label{lem: from weak to strong 1}
	Let $p \in [1,\infty)$ and $g: \R \to \R$ be a function such that $|g(t)| \leq \gamma t^\alpha$ for every $t \in \R$, for some $\gamma>0$ and $\alpha>1$. Let $E \subset \R^n$ be a measurable open set, given $u \in L^{p,\infty}(E) \cap L^\infty(E)$ we have that $g \circ u \in L^p(E)$ and
	$$
	\Vert g \circ u \Vert_{L^p(E)} \leq C \Vert u \Vert_{L^{p,\infty}(E)},
	$$
	where $C=C(\alpha,\gamma,p,\Vert u \Vert_{L^\infty})$.
\end{lem}

\begin{proof}
	The proof is identical to the one of \cite[Proposition 3.2]{MuScaZep}  and we omit it.
\end{proof}

\begin{lem}\label{lem: divdiv of sksym is zero}
	Let $\Omega \subset \R^n$ be an open set with Lipschitz boundary. Let $F \in L^{1^*} (\Omega, \R^{n \times n})$ be a matrix field such that $F(x)$ is skew symmetric for $\mathcal{L}^n$-a.e. $x \in \Omega$. Then, $\divr(\divr(F))=0$ as an element of $W^{-2,1^*}(\Omega)$.
\end{lem}

\begin{proof}
	We can assume first that $F$ is smooth on $\Omega$ and then conclude by approximation. Observe that
	$$
	\divr(\divr(F))=\sum_{i=1}^n \sum_{j=1}^n \partial_{i} \partial_{j} F_{ij}=0.
	$$ 
	Indeed, since $F$ is skew symmetric on $\Omega$, we have $\partial_{ii} F_{ii}=0$ and $\partial_i \partial_j F_{ij}=-\partial_j \partial_i F_{ji}$ for every $i,j = 1,\dots,n$.
\end{proof}

The next fundamental Theorem is contained in \cite[Theorem 3.1]{BrezisVanSc} (see also \cite{BourgainBrezis}).

\begin{thm}[Bourgain, Brezis, Van Schaftingen]\label{BBT}
	Let $Q \subset \R^n$ be the unit cube and let $f \in L^1(Q,\R^n)$. If
	$$
	[f]:=\sup \left\{ \int f \cdot \nabla \zeta \, dx \bigg| \ \zeta \in C^2(Q), \ \zeta=0 \ \mbox{and} \ \nabla \zeta=0 \ \mbox{on} \ \partial Q \ \mbox{and} \ \Vert \nabla^2 \zeta \Vert_{L^n(Q)} \leq 1 \right\} < +\infty,
	$$
	then the system
	\begin{equation*}
		\begin{cases*}
			-\Delta u = f & on $Q$, \\
			u=0 & on $\partial Q$,
		\end{cases*}
	\end{equation*}
	has a unique solution $u \in W^{1,1^*}(Q,\R^n)$ satisfying 
	$$
	\Vert u \Vert_{W^{1,1^*}(Q,\R^n)} \leq C( \Vert f \Vert_{L^1(Q,\R^n)}+[f]).
	$$
	where $C>0$ depends solely on the dimension $n$.
\end{thm}

In order to prove the rigidity in the critical exponent, we first need to generalize an argument contained in the proof of \cite[Theorem 3.2]{MuScaZep} to arbitrary dimension in order to apply Theorem \ref{BBT}. Finally, we employ an argument used in \cite{ContiGarrRigidity} to obtain an estimate up to the boundary.

\begin{thm}\label{thm: multiwell rigidity crit}
	Let $\Omega \subset \R^n$ be a bounded Lipschitz domain and consider $K$ as in \eqref{0}. There exists a constant $C(\Omega,K)>0$ with the property that for every $u \in W^{1,1^*}(\Omega;\R^n)$ with $\Delta u \in \mathcal{M}_b(\Omega;\R^2)$, there exists an associated $F \in K$ such that
	\begin{equation}\label{crit grad rigidity est}
		\Vert \nabla u-F \Vert_{L^{1^*}(\Omega;\R^{n \times n})} \leq C(\Omega,K)\left(\Vert \dist(\nabla u,K) \Vert_{L^{1^*}(\Omega)}+|\Delta u|(\Omega)\right).
	\end{equation}
\end{thm}

\begin{proof}
	Let us set $q:=1^*$ for brevity. By approximation we can assume $\Delta u \in L^1(\Omega;\R^2)$. We start with the weak local estimate on cubes.
	
	\noindent \textbf{Step 1:} Let $Q$ be a cube and $u \in W^{1,q}(Q;\R^n)$ with $\Delta u \in L^1(Q;\R^n)$. Let $v \in W^{1,1}(\Omega;\R^n)$ be the solution of 
	\begin{equation}\label{15-1}
		\begin{cases*}
			\Delta v = \Delta u & on $Q$, \\
			v=0 & on $\partial Q$.
		\end{cases*}
	\end{equation} 
	By regularity theory for elliptic equations with measure data (see \cite{BoccardoMeasureData} for $n \geq 3$, \cite{DoHunMul-MeasData} for $n=2$, \cite{MingioneMeasureData} and references therein), we have that $\nabla v \in L^{q,\infty}(Q;\R^{n \times n})$ and there exists a constant $C(n)>0$ such that 
	\begin{equation}\label{15}
		\Vert 	\nabla v \Vert_{L^{q,\infty}(Q)} \leq C(n) \Vert \Delta u \Vert_{L^{1}(Q)}.
	\end{equation}
	Hence, if we set $w:=u-v$, we have that $w$ is harmonic on $Q$ and, using triangular inequality and \eqref{15}, there exists $C(n)>0$ such that
	\begin{equation}\label{16}
		\Vert \dist(\nabla w,K) \Vert_{L^{q,\infty}(Q)} \leq C(n)(\Vert \dist(\nabla u,K) \Vert_{L^{q,\infty}(Q)} + \Vert \Delta u \Vert_{L^{1}(Q)}).
	\end{equation}
	Let now $Q' \Subset Q$ be a concentric cube to $Q$ and $\delta:=\dist(Q',\partial Q)$. Since $w$ is harmonic on $Q$, by Proposition \ref{prop: weak harmonic rigidity} we infer that there exist $C(K,n,\delta)>0$, $U \in \{ U_1,\dots U_\ell \}$ and $R \in SO(n)$ such that
	\begin{equation}\label{17bis}
		\Vert \nabla w -RU \Vert_{L^{q,\infty}(Q')} \leq C(K,n,\delta) \Vert \dist(\nabla w,K) \Vert_{L^{q,\infty}(Q)}.
	\end{equation}
	Therefore, 
	\begin{equation}\label{17}
		\Vert \nabla u -RU \Vert_{L^{q,\infty}(Q')} \leq C(K,n,\delta) \Vert \dist(\nabla u,K) \Vert_{L^{q,\infty}(Q)}+C(K,n,\delta) \Vert \Delta u \Vert_{L^{1}(Q)}. \footnote{We point out that this estimate is only local and cannot be used to straightforwardly derive a global estimate in $L^{q,\infty}$ as no weighted Poincaré inequality is available for weak $L^p$ spaces. By comparison, in the single well case recalled in Theorem~\ref{thm: CDM rigidity}, the authors of \cite{CDM} must take a different path going through a suitable Korn inequality and a rigidity estimate with mixed growth on the whole domain. The extension of this approach to our setting is not investigated in the present paper.}
	\end{equation}
	 Up to substituting $u$ with $R^T u$, we can assume $R=I$ in \eqref{17}.
	 Let $\mathcal{R}:Q \to SO(n)$ be the measurable matrix field with the property that
	 $|\nabla u(x)-\mathcal{R}(x)U|=\dist(\nabla u(x),SO(n)U)$ for $\mathcal{L}^n$-a.e. $x \in Q$. Using the canonical representation for orthogonal matrices and assuming $n$ even, we can find measurable functions $B \colon Q \to O(n)$ and $\theta:=(\theta_1,\dots,\theta_{n/2}) \colon Q \to [-\pi,\pi)^{n/2}$ such that for $\mathcal{L}^n$-a.e. $x \in Q$
	 \begin{equation}\label{18}
	 	\mathcal{R}(x)=B^T(x) 
	 	\left(
	 	\begin{array}{c c c c c}
	 		\cos(\theta_1(x)) & -\sin(\theta_1(x)) & & & \\
	 		\sin(\theta_1(x)) & \cos(\theta_1(x)) & & & \\
	 		 & & \ddots & & \\
	 		 & & & \cos(\theta_{n/2}(x)) & -\sin(\theta_{n/2}(x)) \\
	 		 & & & \sin(\theta_{n/2}(x)) & \cos(\theta_{n/2}(x))
	 	\end{array}
 		\right)
	 	B(x).
	 \end{equation}
 	If instead the dimension $n$ is odd, the matrix above has also a $1$ in position $(n,n)$. However, this is not relevant in the proof, hence we will consider only the case $n$ even.
 	By \eqref{17bis} and \eqref{17} we have
 	\begin{align}
 		\begin{split}\label{19}
 		\Vert \mathcal{R}U-U \Vert_{L^{q,\infty}(Q')} & \leq 2\Vert \mathcal{R}U-\nabla u\Vert_{L^{q,\infty}(Q')}+2\Vert \nabla u - U \Vert_{L^{q,\infty}(Q')} \\
 		& \leq C(K,n,\delta) \left(\Vert \dist(\nabla u,SO(n)U) \Vert_{L^{q,\infty}(Q')}+\Vert \dist(\nabla u,K) \Vert_{L^{q,\infty}(Q)} +\Vert \Delta u \Vert_{L^1(Q)}\right) \\
 		& \leq C(K,n,\delta)\left(\Vert \nabla u - U \Vert_{L^{q,\infty}(Q')}+\Vert \dist(\nabla u,K) \Vert_{L^{q,\infty}(Q)}+\Vert \Delta u \Vert_{L^{1}(Q)}\right) \\
 		& \leq C(K,n,\delta) \left(\Vert \dist(\nabla u,K) \Vert_{L^{q,\infty}(Q)}+\Vert \Delta u \Vert_{L^{1}(Q)}\right).
 		\end{split}
 	\end{align}
 	Since $U$ is invertible and using the expression in \eqref{18}, one can see that there exists a constant $c>0$ depending on $U$ such that for $\mathcal{L}^n$-a.e. $x \in Q'$ and every $i=1,\dots,n/2$
 	$$
 	|\mathcal{R}U-U| \geq c |\mathcal{R}-I| \geq c \frac{|\theta_i|}{2}.
 	$$
	Using \eqref{19}, this gives for every $i=1,\dots,n/2$,
	\begin{equation}\label{20}
	\Vert \theta_i \Vert_{L^{q,\infty}(Q')} \leq C(K,n,\delta) (\Vert \dist(\nabla u,K) \Vert_{L^{q,\infty}(Q)}+\Vert \Delta u \Vert_{L^{1}(Q)}).
	\end{equation}
	Let us define $g(t):=\cos t-1$ and $h(t):=\sin t-t$. By Lemma \ref{lem: from weak to strong 1}, we have for every $i=1,\dots,n/2$
	\begin{equation}\label{21}
	\Vert g(\theta_i) \Vert_{L^{q}(Q')}+\Vert h(\theta_i) \Vert_{L^{q}(Q')} \leq C \Vert \theta_i \Vert_{L^{q,\infty}(Q')}.
	\end{equation}
	
	Let $F \colon Q \to K$ be the measurable matrix field such that $|F(x)-\nabla u(x)|=\dist(\nabla u,K)$ for $\mathcal{L}^n$-a.e. $x \in Q$ and set $\rho:=\frac{1}{4} \min_{i \neq j} \dist(K_i,K_j)$. By definition of $F$, $\mathcal{R}$ and $\rho$ we have $F=\mathcal{R}U$ on the set $\{ \dist(\nabla u,SO(n)U) \leq \rho \}$.
	Using \eqref{17}, we infer
	$$
	\Vert \dist(\nabla u, SO(n)U) \Vert_{L^{q,\infty}(Q')} \leq C(K,n,\delta)(\Vert \dist(\nabla u,K) \Vert_{L^{q,\infty}(Q)}+\Vert \Delta u \Vert_{L^{1}(Q)}),
	$$
	which in turn, by definition of norm on weak Lebesgue spaces, implies 
	\begin{equation}\label{22}
		\mathcal{L}^n(\{ \dist(\nabla u,SO(n)U) > \rho \} \cap Q') \leq C(K,n,\delta) (\Vert \dist(\nabla u,K) \Vert_{L^{q,\infty}(Q)}+\Vert \Delta u \Vert_{L^{1}(Q)})^q.
	\end{equation}
	We now consider the infinitesimal rotation $R_{\textnormal{inf}} \colon Q \to \R^{n \times n}$ as 
	\begin{equation*}
		R_{\textnormal{inf}}(x):=\left(
		\begin{array}{c c c c c}
			1 & -\theta_1(x) & & & \\
			\theta_1(x) & 1 & & & \\
			 & & \ddots & & \\
			 & & & 1 & -\theta_{n/2}(x) \\
			 & & & \theta_{n/2}(x) & 1 
		\end{array}
		\right).
	\end{equation*}
	 Using \eqref{20}--\eqref{22} we estimate
	 \begin{align*}
	 	\Vert F-B^T R_{\textnormal{inf}} B U \Vert_{L^{q}(Q')} & \leq C(K)\mathcal{L}^n(\{ \dist(\nabla u,SO(n)U) > \rho \} \cap Q')^{\frac{1}{q}}+\Vert \mathcal{R}U-B^T R_{\textnormal{inf}} B U \Vert_{L^{q}(Q')} \\
	 	& \leq C(K,n,\delta) \left( \Vert \dist(\nabla u,K) \Vert_{L^{q,\infty}(Q)}+\Vert \Delta u \Vert_{L^{1}(Q)}+\sum_{i=1}^{n/2} \Vert \theta_i \Vert_{L^{q,\infty}(Q')} \right) \\
	 	& \leq C(K,n,\delta)(\Vert \dist(\nabla u,K) \Vert_{L^{q,\infty}(Q)}+\Vert \Delta u \Vert_{L^{1}(Q)}).
	 \end{align*}
 	 Hence, by definition of $F$ we have
 	 \begin{equation}\label{23}
 	 	\Vert \nabla u-B^T R_{\textnormal{inf}} B U \Vert_{L^{q}(Q')} \leq C(K,n,\delta)(\Vert \dist(\nabla u,K) \Vert_{L^{q}(Q)}+\Vert \Delta u \Vert_{L^{1}(Q)}).
 	 \end{equation}
  	From now on, we set for brevity $\varepsilon:= \Vert \dist(\nabla u,K) \Vert_{L^{q}(Q)}+\Vert \Delta u \Vert_{L^{1}(Q)}$.
  	Define the matrix field $h:=\nabla u-B^T R_{\textnormal{inf}} B U$. By \eqref{23} we have $h \in L^q(Q',\R^{n\times n})$ and $\Vert h \Vert_{L^{q}(Q')} \leq C(K,n,\delta) \varepsilon$.
  	Since $R_{\textnormal{inf}}$ is the sum of the identity matrix and a skew symmetric matrix, we infer that $U^T B^T R_{\textnormal{inf}} B U$ is the sum of a constant matrix and a skew symmetric matrix on $Q'$. Therefore, using Lemma \ref{lem: divdiv of sksym is zero} we estimate
  	\begin{equation}\label{24}
  		\Vert \divr(\divr(U^T \nabla u)) \Vert_{W^{-2,q}(Q')} = \Vert \divr(\divr(U^T h)) \Vert_{W^{-2,q}(Q')} \leq C(K,n,\delta) \varepsilon.
  	\end{equation}
  	\textbf{Step 2: }Notice that $\divr(\divr(U^T \nabla u))=\divr(U^T \Delta u)=\divr(\Delta (U^T u))$. We consider the following system
  	\begin{equation*}
  		\begin{cases*}
  			\Delta \tilde{v} = \Delta( U^T u) = U^T \Delta u & on $Q'$, \\
  			\tilde{v}=0 & on $\partial Q'$.
  		\end{cases*}
  	\end{equation*}
  	We have 
  	$
  	\Vert U^T \Delta u \Vert_{L^1(Q')} \leq C(K) \Vert \Delta u \Vert_{L^1(Q)}.
  	$
  	We now apply Theorem \ref{BBT} on $Q'$ with $f=U^T \Delta u$. 
  	By \eqref{24} it holds that $\Vert \divr(U^T \Delta u) \Vert_{W^{-2,q}(Q')} \leq C(K,n,\delta) \varepsilon$. Thus, observing that in Theorem \ref{BBT} $[f]=\Vert \divr f \Vert_{W^{-2,q}(Q')}$, we conclude
  	\begin{equation}\label{25}
  		\Vert \nabla \tilde{v} \Vert_{L^q(Q')} \leq C(K,n,\delta) \varepsilon.
  	\end{equation}
  	Let $\tilde{w}:=u-U^{-T} \tilde{v} \in W^{1,q}(Q';\R^n)$. By definition we get $\Delta \tilde{w}=0$ on $Q'$. Hence, we can apply Theorem \ref{thm: rigidity harmonic} to $\tilde{w}$ on $Q'$ and find $V \in K$ such that
  	\begin{equation}\label{26}
  		\Vert \nabla \tilde{w} -V \Vert_{L^{q}(Q')} \leq C(K,n) \Vert \dist(\nabla \tilde{w},K) \Vert_{L^{q}(Q')}.
  	\end{equation}
  	Combining \eqref{25} and \eqref{26} and using the triangle inequality, we estimate
  	\begin{align*}
  		\Vert \nabla u - V \Vert_{L^{q}(Q')} & \leq \Vert \nabla \tilde{w} - V \Vert_{L^{q}(Q')} + \Vert \nabla U^{-T} \tilde{v} \Vert_{L^{q}(Q')} \\
  		&  \leq C(K,n,\delta) (\Vert \dist(\nabla \tilde{w},K) \Vert_{L^{q}(Q')}+ \varepsilon) \\
  		& \leq C(K,n,\delta) \varepsilon.
  	\end{align*}
  	Hence, for the unit cube $Q$ and $Q'\Subset Q$ a smaller cube concentric to $Q$ with $\delta:=\dist(Q',\partial Q)$, we have proven that there exists a constant $C=C(K,n,\delta)>0$ such that for every $u \in W^{1,q}(Q;\R^n)$ there exists an associated matrix $F \in K$ with the property that
  	\begin{equation}\label{27}
  	\Vert \nabla u - F \Vert_{L^q(Q';\R^{n \times n})} \leq C(K,n,\delta)( \Vert \dist(\nabla u,K) \Vert_{L^q(Q)}+\Vert \Delta u \Vert_{L^1(Q)}).
  	\end{equation}
  	Moreover, the constant $C(K,n,\delta)$ is invariant after translation and rescaling of the cube.
  	
  	\noindent \textbf{Step 3: }Finally, we are ready to prove the estimate up to the boundary. We use again the Whitney covering of $\Omega$ with cubes, i.e., Proposition \ref{prop: cubic decomp}. Let $\{ \hat{Q}_j \}_j$ the countable family of cubes covering $\Omega$, for each $j \in \N$ we have $\hat{Q}_j:=x_j+(r_j/2,r_j/2)^n$ and $Q_j:=x_j+(r_j,r_j)^n$. Moreover, by \eqref{p3} of Proposition \ref{prop: cubic decomp}, it holds
  	\begin{equation}\label{28}
  		\chi_{\Omega} \leq \sum_j \chi_{\hat{Q}_j} \leq \sum_j \chi_{Q_j} \leq N \chi_{\Omega}.
  	\end{equation}
  	Consider $\{\varphi_j\}_j$ with $\varphi_j \in C^\infty_c(\hat{Q}_j)$ to be a partition of unity subordinated to the cubes $\hat{Q}_j$ and such that $|\nabla \varphi_j| \leq c/r_j$, where $c>0$ is a dimensional constant. We use the estimate \eqref{27} with $Q'=\hat{Q}_j$ and $Q=Q_j$ for each $j \in \N$. Therefore, there exists a constant $C(K,n)>0$ such that for every $j \in \N$ we can find matrices $F_j \in K$ with the property that
  	\begin{equation}\label{29}
  		\Vert \nabla u - F_j \Vert_{L^q(\hat{Q}_j)} \leq C(K,n)( \Vert \dist(\nabla u,K) \Vert_{L^q(Q_j)}+\Vert \Delta u \Vert_{L^1(Q_j)}).
  	\end{equation}
  	Define the matrix field $F \colon \Omega \to \R^{n \times n}$ as $F:= \sum_j \varphi_j F_j$, by triangular inequality we have that there exists $C=C(N)>0$ such that
  	\begin{equation}\label{30}
  		\Vert \nabla u - F \Vert_{L^q(\Omega)}^q \leq C \sum_j \Vert \nabla u -F_j \Vert_{L^q(\hat{Q}_j)}^q.
  	\end{equation}
  	Since $\sum_j \varphi_j=1$ on $\Omega$, we deduce that
  	\begin{equation*}
  		\nabla F = \sum_j \nabla \varphi_j F_j = \sum_j \nabla \varphi_j(F_j - \nabla u).
  	\end{equation*}
	 Hence, by \eqref{p2} of Proposition \ref{prop: cubic decomp}, \eqref{28} and \eqref{29}
	 \begin{align}
	 	\begin{split}\label{31}
	 		\int_\Omega \dist^q(x,\partial \Omega) |\nabla F|^q \, dx & \leq C \sum_j \int_{\hat{Q}_j} r_j^q |\nabla \varphi_j|^q |\nabla u-F_j|^q \, dx \\
	 		& \leq C \sum_j \int_{\hat{Q}_j} |\nabla u-F_j|^q \, dx \\
	 		& \leq C(K,n) \sum_j \left( \Vert \dist(\nabla u,K) \Vert_{L^q(Q_j)}^q+\Vert \Delta u \Vert_{L^1(Q_j)}^q \right) \\
	 		& \leq C(K,n) \left( \Vert \dist(\nabla u,K) \Vert_{L^q(\Omega)}^q+\Vert \Delta u \Vert_{L^1(\Omega)}^q \right).
	 	\end{split}
	 \end{align}
 	Using the weighted Poincaré inequality in $L^q$, we infer that there exists a matrix $A \in \R^{n \times n}$ such that
 	\begin{equation}\label{32}
 		\Vert F-A \Vert_{L^q(\Omega)}^q \leq C(\Omega) \int_\Omega \dist^q(x,\partial \Omega) |\nabla F|^q \, dx.
 	\end{equation}
 	Combining \eqref{29}--\eqref{32} we estimate
 	\begin{equation}\label{33}
 		\Vert \nabla u - A \Vert_{L^q(\Omega)}^q \leq C(\Omega,K,n) \left( \Vert \dist(\nabla u,K) \Vert_{L^q(\Omega)}^q+\Vert \Delta u \Vert_{L^1(\Omega)}^q \right).
 	\end{equation}
 	Finally, let $F_* \in K$ be the matrix such that $|A-F_*|=\dist(A,K)$, using \eqref{33} we estimate
 	\begin{equation}\label{34}
 		\mathcal{L}^n(\Omega)^{\frac{1}{q}} |F_*-A|\leq C(\Omega,K,n) \left( \Vert \dist(\nabla u,K) \Vert_{L^q(\Omega)}+\Vert \Delta u \Vert_{L^1(\Omega)} \right).
 	\end{equation}
 	Thus, combining \eqref{33} and \eqref{34} we infer \eqref{crit grad rigidity est} and the proof is concluded.
 	
\end{proof}


\subsection{Auxiliary Estimates for $W^{2,p}$ Functions} 
	 
	We point out that if $u \in W^{1,p}(\Omega,\R^n)$ is a function with $L^p$ (or $L^{1^*,\infty}$) bounded Hessian on $\Omega$, then a similar argument as in Section \ref{sez 2} entails a global rigidity result in a multiplicative form that we are going to use in the application in Section \ref{sez 4}. We therefore give statements and proofs.
	 
	We start proving the multiplicative rigidity result for "strong" $L^p$ norm.
	
	\begin{prop}\label{prop: rigidity multiplic form}
		Let $\Omega \subset \R^n$ be a bounded Lipschitz domain and consider $K$ as in \eqref{0}. There exists a constant $C>0$ depending only on $K$, $p$, $n$ and on the isoperimetric constant of $\Omega$ with the property that for every $u \in W^{2,p}(\Omega;\R^n)$ there is an associated matrix $U \in \{ U_1,\dots,U_\ell \}$ such that 
		\begin{equation}\label{65}
			\int_\Omega \dist^p(\nabla u,SO(n)U) \, dx \leq C \int_\Omega \dist^p(\nabla u,K) \, dx+C\left[ \left( \int_\Omega \dist^p(\nabla u,K) \,dx \right)^{\frac{p-1}{p}} \left( \int_\Omega |\nabla^2 u|^p \, dx \right)^{\frac{1}{p}} \right]^{\frac{n}{n-1}}.
		\end{equation}
	\end{prop}

	\begin{proof}
		Let $\rho:=\frac{1}{8} \min_{i \neq j} \dist(K_i,K_j)>0$. Notice that by definition of $\rho$ we have two possibilities: 
		\begin{itemize}
			\item[(a)] there exists a single well (without loss of generality) $K_1$ such that 
			$$
			\mathcal{L}^n(\{ d_{K \setminus K_1}(\nabla u) \leq 2 \rho \} ) \leq \frac{\mathcal{L}^n(\Omega)}{2};
			$$
			\item[(b)] for every $i=1,\dots,\ell$ we have $\mathcal{L}^n(\{ d_{K_i}(\nabla u) \leq 2 \rho \}) \leq \mathcal{L}^n(\Omega)/2$.
		\end{itemize}
		Assume that (a) holds. Arguing as in Proposition \ref{prop: harmonic rigidity cube}, there exists a constant $C=C(K,p,n)>0$ such that
		$$
		\int_\Omega d^p_{K_1}(\nabla u) \, dx \leq C \int_\Omega d^p_K(\nabla u) \, dx + C \mathcal{L}^n(\{ d_{K \setminus K_1}(\nabla u) \leq \rho \}).
		$$
		In turn, by Lemma \ref{lem: beginning lemma} we have that there exists $C>0$ depending on $K$, $p$, $n$ and the isoperimetric constant of $\Omega$ such that
		\begin{align*}
		\mathcal{L}^n(\{ d_{K \setminus K_1}(\nabla u) \leq \rho \}) & \leq C \left( \int_{\{ \rho<d_{K \setminus K_1}(\nabla u)<2\rho \}} |\nabla^2 u| \, dx \right)^{\frac{n}{n-1}} \\
		& \leq C \left[\mathcal{L}^n( \{ \rho<d_{K \setminus K_1}(\nabla u)<2\rho \} )^{\frac{p-1}{p}}\left( \int_\Omega |\nabla^2 u|^p \, dx  \right)^{\frac{1}{p}} \right]^{\frac{n}{n-1}} \\
		& \leq C \left[ \left( \int_\Omega d_K^p(\nabla u) \, dx \right)^{\frac{p-1}{p}} \left( \int_\Omega |\nabla^2 u|^p \, dx \right)^{\frac{1}{p}} \right]^{\frac{n}{n-1}}.
		\end{align*}
		Hence we conclude \eqref{65}. 
		
		If (b) holds instead, we observe that
		$$
		\int_\Omega d^p_{K_1}(\nabla u) \, dx \leq C \int_\Omega d^p_K(\nabla u) \, dx + C \sum_{i=2}^\ell \mathcal{L}^n(\{ d_{K_i}(\nabla u) \leq \rho \}).
		$$
		Moreover, by Lemma \ref{lem: beginning lemma} for every $i=2,\dots,\ell$ we have 
		\begin{align*}
			\mathcal{L}^n(\{ d_{K_i}(\nabla u) \leq \rho \}) & \leq  C \left( \int_{\{ \rho<d_{K_i}(\nabla u)<2\rho \}} |\nabla^2 u| \, dx \right)^{\frac{n}{n-1}} \\
			& \leq C \left[ \left( \int_\Omega d_K^p(\nabla u) \, dx \right)^{\frac{p-1}{p}} \left( \int_\Omega |\nabla^2 u|^p \, dx \right)^{\frac{1}{p}} \right]^{\frac{n}{n-1}},
		\end{align*}
		therefore, we conclude again \eqref{65}.
	\end{proof}

		For the weak $L^p$ norm we have the following result in the case $p=1^*$.
		
		\begin{prop}\label{prop: weak rigidity multipl form}
			Let $\Omega \subset \R^n$ be a bounded Lipschitz domain and consider $K$ as in \eqref{0}. There exists a constant $C>0$ depending only on $K$, $n$ and on the isoperimetric constant of $\Omega$ with the property that for every $u \in W^{2,1}(\Omega;\R^n)$ with $\nabla^2 u \in L^{1^*,\infty}(\Omega;\R^{n \times n \times n})$ there is an associated matrix $U \in \{ U_1,\dots,U_\ell \}$ such that 
			\begin{equation}\label{66}
				\Vert \dist(\nabla u,SO(n)U) \Vert_{L^{1^*,\infty}(\Omega)} \leq C \Vert \dist(\nabla u,K) \Vert_{L^{1^*,\infty}(\Omega)}+C\Vert \dist(\nabla u,K) \Vert_{L^{1^*,\infty}(\Omega)}^{\frac{1}{n-1}}\Vert \nabla^2 u \Vert_{L^{1^*,\infty}(\Omega)}.
			\end{equation}
		\end{prop} 
	
		\begin{proof}
			Let us set $q:=1^*$ for brevity. Let $\rho:=\frac{1}{8} \min_{i \neq j} \dist(K_i,K_j)>0$. As in Proposition \ref{prop: rigidity multiplic form} by definition of $\rho$ we have two possible cases:
			\begin{itemize}
				\item[(a)] there exists a single well (without loss of generality) $K_1$ such that 
				$$
				\mathcal{L}^n(\{ d_{K \setminus K_1}(\nabla u) \leq 2 \rho \} ) \leq \frac{\mathcal{L}^n(\Omega)}{2};
				$$
				\item[(b)] for every $i=1,\dots,\ell$ we have $\mathcal{L}^n(\{ d_{K_i}(\nabla u) \leq 2 \rho \}) \leq \mathcal{L}^n(\Omega)/2$. 
			\end{itemize}
			We are going to prove only the first case, the second one being analogous. Arguing as in Proposition \ref{prop: weak harmonic rigidity}, there exists a constant $C=C(K,n)>0$ such that for every $\lambda>0$
			$$
			\lambda^q \mathcal{L}^n(\{ d_{K_1}(\nabla u) > \lambda \}) \leq C \Vert d_K(\nabla u) \Vert_{L^{q,\infty}(\Omega)}^q+C \mathcal{L}^n( \{ d_{K \setminus K_1}(\nabla u) \leq \rho \} ).
			$$
			Using Lemma \ref{lem: beginning lemma} we infer the existence of $C>0$ depending on $K$, $n$ and the isoperimetric constant of $\Omega$ such that
			\begin{align*}
			\mathcal{L}^n( \{ d_{K \setminus K_1}(\nabla u) \leq \rho \} )^{\frac{1}{q}} & \leq C \mathcal{L}^n( \{ \rho < d_{K \setminus K_1}(\nabla u) <2 \rho \} )^{\frac{1}{n}} \Vert \nabla^2 u \Vert_{L^{q,\infty}(\Omega)} \\
			& \leq C\Vert d_K(\nabla u) \Vert_{L^{q,\infty}(\Omega)}^{\frac{1}{n-1}}\Vert \nabla^2 u \Vert_{L^{q,\infty}(\Omega)}.
			\end{align*}
			Thus, for every $\lambda>0$ we have 
			$$
			\lambda \mathcal{L}^n(\{ d_{K_1}(\nabla u) > \lambda \})^{\frac{1}{q}} \leq C  \Vert d_K(\nabla u) \Vert_{L^{q,\infty}(\Omega)}+C\Vert d_K(\nabla u) \Vert_{L^{q,\infty}(\Omega)}^{\frac{1}{n-1}}\Vert \nabla^2 u \Vert_{L^{q,\infty}(\Omega)}.
			$$
			By definition of $L^{q,\infty}$ norm, \eqref{66} follows.
		\end{proof}

\section{Multiwell Rigidity Estimates for Incompatible Fields}

For now we have dealt only with gradient fields, i.e., fields with zero curl. If we want to prove an analogous of the estimate in Theorem \ref{thm: multiwell rigidity crit} for incompatible fields, we need to take into account also the curl. The proof of the multiwell rigidity estimate follows closely the one for incompatible fields in the one well case (for $n=2$ in \cite{MuScaZep} and for $n \geq 3$ in \cite{ContiGarrRigidity}). In particular, it is hinged on the fact that given a cube $Q$ and a field $\beta \in L^{1^*}(Q,\R^{n \times n})$ with $\curl \beta \in \mathcal{M}_b(Q,\R^{n \times n \times n})$ and $\divr \beta \in \mathcal{M}_b(Q,\R^n)$, we can find a decomposition of $\beta$ of the type
\begin{equation*}
	\beta=
	\underbrace{\parbox{50pt}{\centering Y}}_{\parbox{50pt}{\scriptsize\centering $\curl Y=\curl \beta$ \\ $\divr Y=0$}}
	+
	\underbrace{\parbox{50pt}{\centering $\nabla v$}}_{\parbox{50pt}{\scriptsize\centering $\curl \nabla v=0$ \\ $\divr \nabla v=\divr \beta$}}
	+
	\underbrace{\parbox{50pt}{\centering $\nabla w$}}_{\parbox{50pt}{\scriptsize \centering Harmonic}}.
\end{equation*}

We first need two preliminary results. The first one follows immediately from Theorem 3.1 and Lemma 3.4 of \cite{BrezisVanSc} (see also \cite{BourgainBrezis}) and it is used in the proof of the two dimensional case.

\begin{thm}\label{thm: n=2 BB}
	Let $\Omega \subset \R^2$ be a bounded Lipschitz domain. Let $f \in L^1(\Omega,\R^2)$ be a vector field such that $\divr f \in H^{-2}(\Omega)$, then $f \in H^{-1}(\Omega,\R^2)$ and the following estimate holds
	$$
	\Vert f \Vert_{H^{-1}(\Omega,\R^2)} \leq C(\Omega) \left( \Vert \divr f \Vert_{H^{-2}(\Omega)} +\Vert f \Vert_{L^1(\Omega,\R^2)} \right).
	$$
\end{thm}

The second one is a technical lemma contained in \cite[Lemma 2.1]{ContiGarrRigidity}.

\begin{lem}\label{lem: CGBB lemma n geq 3}
	Let $n \geq 3$ and consider $Q \subset \R^n$ the unit cube with normal $\nu$ to its boundary. Given a vector field $Y \in L^1(Q,\R^n)$ such that $\divr Y=0$ distributionally on $Q$ and $Y \cdot \nu =0$ on $\partial Q$. Then, 
	$$
	\Vert Y \Vert_{L^{1^*}(Q)} \leq C(n) \Vert \curl Y \Vert_{L^1(Q)}.
	$$
\end{lem}

We are now ready to prove the multiwell rigidity estimate for incompatible fields.

\begin{thm}\label{thm: rig est for incompat}
	Let $\Omega \subset \R^n$ bounded, open and connected set with Lipschitz boundary and consider $K$ as in \eqref{0}. There exists a constant $C(\Omega,K)>0$ with the property that for every field $\beta \in L^{1^*}(\Omega;\R^{n \times n})$ with $\curl \beta \in \mathcal{M}_b(\Omega,\R^{n \times n \times n})$ and $\divr \beta \in \mathcal{M}_b(\Omega,\R^n)$, we can find an associated rotation $F \in K$ such that
	\begin{equation}\label{crit rig est for incompat}
		\Vert \beta-F \Vert_{L^{1^*}(\Omega;\R^{n \times n})} \leq C(\Omega,K)\left( \Vert \dist(\beta,K) \Vert_{L^{1^*}(\Omega)}+|\divr \beta|(\Omega)+|\curl \beta|(\Omega) \right).
	\end{equation}
\end{thm}

\begin{proof}
	Set as before $q:=1^*$. First, we are going to prove the estimate (locally) on the unit cube $Q$. The proof differs greatly in the cases $n=2$ and $n \geq 3$, hence, we split it. Once we have a local estimate on cubes, we employ the argument in \cite{ContiGarrRigidity}, used also in the proof of Theorem \ref{thm: multiwell rigidity crit}, to derive the estimate up to the boundary of $\Omega$. Moreover, without loss of generality, we can assume $\beta$ to be smooth.
	
	\noindent \textit{Estimate on $Q$ for $n \geq 3$.} 
	
	\noindent Let $\nu$ be the normal to $\partial Q$. Define as in \cite[Theorem 1.1]{ContiGarrRigidity} $u \in W^{1,1}(Q,\R^n)$ as the solution of the Neumann problem
	\begin{equation*}
		\begin{cases*}
			\Delta u= \divr \beta & in $Q$ \\
			\partial_\nu u = \beta \nu & on $\partial Q$,
		\end{cases*}
	\end{equation*}
	and set $Y:=\beta - \nabla u$. Since by definition of $u$ we have $\divr Y=0$, $Y\nu =0$ and $\curl Y = \curl \beta$, by Lemma \ref{lem: CGBB lemma n geq 3} applied to every row of $Y$ we infer that 
	\begin{equation}\label{35}
		\Vert Y \Vert_{L^q(Q)} \leq C(n) \Vert \curl \beta \Vert_{L^1(Q)}.
	\end{equation}
	Using Theorem \ref{thm: multiwell rigidity crit} we find $F \in K$ such that 
	$$
	\Vert \nabla u -F \Vert_{L^q(Q)} \leq C(K,n) \left( \Vert \dist(\nabla u,K) \Vert_{L^q(Q)}+\Vert \Delta u \Vert_{L^1(Q)} \right).
	$$
	Hence, since $\nabla u = \beta -Y$ and $\Delta u =\divr \beta$, using \eqref{35} we estimate
	\begin{align*}
		\Vert \beta -F \Vert_{L^q(Q)} & \leq \Vert Y \Vert_{L^q(Q)} +\Vert \nabla u-F \Vert_{L^{q}(Q)} \\
		& \leq C(K,n) \left( \Vert \dist( \beta,K) \Vert_{L^q(Q)}+\Vert \Delta u \Vert_{L^1(Q)} + 2\Vert Y \Vert_{L^q(Q)} \right) \\
		& \leq C(K,n) \left( \Vert \dist( \beta,K) \Vert_{L^q(Q)}+\Vert \divr \beta \Vert_{L^1(Q)} + \Vert \curl \beta \Vert_{L^1(Q)} \right).
	\end{align*}

	\noindent \textit{Estimate on $Q$ for $n =2$.}
	
	\noindent We have that $\curl \beta \in L^1(Q,\R^2)$. Observe that for $i=1,2$  
	\begin{equation*}
		(\curl \beta)_i= \curl (\beta^T e_1)=-\divr(J(\beta^T e_i)), \qquad \mbox{where } 
		J:=
		\left(
		\begin{array}{c c}
			0 & -1 \\
			1 & 0
		\end{array}
		\right).
	\end{equation*}
	Hence, $\curl \beta \in H^{-1}(Q,\R^n)$ and there exists a unique solution $z \in H^1_0(Q,\R^2)$ to the Dirichlet problem:
	\begin{equation}\label{36}
		\begin{cases*}
			\Delta z = - \curl \beta & on $Q$, \\
			z = 0 & on $\partial Q$.
		\end{cases*}
	\end{equation}
	Using regularity theory for measure data (see \cite{DoHunMul-MeasData}) we have
	\begin{equation}\label{37}
		\Vert \nabla z \Vert_{L^{2,\infty}(Q)} \leq C \Vert \curl \beta \Vert_{L^{1}(Q)}.
	\end{equation}
	Set $Y:=\nabla z J$, by \eqref{36} we have that $\curl(\beta-Y)=0$. Thus, since $Q$ is convex, we have that there exists $u \in H^1(Q,\R^2)$ such that $\beta-Y=\nabla u$ $\mathcal{L}^n$-a.e. in $Q$. Moreover, by definition observe that $\divr Y=-\curl \nabla z=0$, therefore, $\divr \beta = \Delta u$. Hence we estimate
	\begin{align}\label{38}
		\begin{split}
			\dist(\nabla u,K) \leq \dist(\beta,K)+|\nabla z|.
		\end{split}
	\end{align}
	Let $Q' \Subset Q$ be a smaller cube concentric to $Q$ and set $\delta:=\dist(Q',\partial Q)$, \eqref{38} and Proposition \ref{prop: weak harmonic rigidity}, arguing as in \eqref{15-1}--\eqref{17}, give the existence of a rotation $R \in SO(2)$ and $U \in \{ U_1,\dots,U_\ell \}$ such that
	\begin{align}\label{39}
		\begin{split}
		\Vert \nabla u - RU \Vert_{L^{2,\infty}(Q')} 
		& \leq C(K,\delta) \left( \Vert \dist(\nabla u,K) \Vert_{L^{2,\infty}(Q)}+\Vert \Delta u \Vert_{L^{1}(Q)} \right) \\
		& \leq C(K,\delta) \left( \Vert \dist(\beta,K) \Vert_{L^{2,\infty}(Q)}+\Vert \divr \beta \Vert_{L^{1}(Q)}+\Vert \curl \beta \Vert_{L^{1}(Q)} \right).
		\end{split}
	\end{align}
	Hence,
	\begin{equation}\label{40}
		\Vert \beta - RU \Vert_{L^{2,\infty}(Q')} \leq C(K,\delta) \left( \Vert \dist(\beta,K) \Vert_{L^{2,\infty}(Q)}+\Vert \divr \beta \Vert_{L^{1}(Q)}+\Vert \curl \beta \Vert_{L^{1}(Q)} \right).
	\end{equation}
	Up to substituting $\beta$ with $R^T \beta$, we can assume $R=I$ in \eqref{39} and \eqref{40}. Let $\theta \colon Q \to [-\pi,\pi)$ be the measurable function such that the associated two dimensional rotation
	\begin{equation*}
		R(\theta):=
		\left(
		\begin{array}{c c}
			\cos(\theta) & -\sin(\theta) \\
			\sin(\theta) & \cos(\theta)
		\end{array}
		\right)
	\end{equation*} 
	satisfies $|\beta - R(\theta(x))U|=\dist(\beta,SO(2)U)$ for $\mathcal{L}^2$-a.e. $x \in Q$. 
	We estimate
	\begin{align}\label{41}
		\begin{split}
			\Vert R(\theta)U-U \Vert_{L^{2,\infty}(Q')} & \leq 2 \Vert R(\theta)U-\beta \Vert_{L^{2,\infty}(Q')}+2\Vert \beta-U \Vert_{L^{2,\infty}(Q')} \\
			& = 2\Vert \dist(\beta,SO(2)U) \Vert_{L^{2,\infty}(Q')} +2\Vert \beta-U \Vert_{L^{2,\infty}(Q')} \\
			&  \leq 4\Vert \beta-U \Vert_{L^{2,\infty}(Q')}.
		\end{split}
	\end{align}
	Since the matrix $U$ is invertible, we can find a constant $c>0$ depending only on $U$ such that for $\mathcal{L}^2$-a.e. $x \in Q'$ we have
	$$
	|R(\theta)U-U| \geq c|R(\theta)-I|  \geq c \frac{|\theta|}{2}.
	$$
	By \eqref{41} this gives that
	\begin{equation}\label{42}
		\Vert \theta \Vert_{L^{2,\infty}(Q')} \leq C(K) \Vert \beta-U \Vert_{L^{2,\infty}(Q')}.
	\end{equation}
	If we define $g_1(t):=\cos t-1$ and $g_2:=\sin t-t$, by Lemma \ref{lem: from weak to strong 1} we deduce that
	\begin{equation}\label{43}
		\Vert g_1(\theta) \Vert_{L^2(Q')}+\Vert g_2(\theta) \Vert_{L^2(Q')} \leq C\Vert \theta \Vert_{L^{2,\infty}(Q')}.
	\end{equation}
	Let now $V \colon Q \to K$ be the measurable matrix field such that $|V(x)-\beta(x)|=\dist(\beta(x),K)$ for $\mathcal{L}^2$-a.e. $x \in Q$. By definition we have that $V=R(\theta)U$ on the set $\{ \dist(\beta,SO(2)U) \leq \rho \}$. By definition of $L^{2,\infty}$ norm we infer that
	\begin{equation}\label{44}
		\mathcal{L}^2(\{ \dist(\beta,SO(2)U) > \rho \} \cap Q') \leq C(K) \Vert \dist(\beta,SO(2)U) \Vert_{L^{2,\infty}(Q')}^2 \leq C(K)\Vert \beta-U \Vert_{L^{2,\infty}(Q')}^2.
	\end{equation}
	Consider now the infinitesimal rotation $R_{\textnormal{inf}}(\theta)$ defined as
	\begin{equation*}
		R_{\textnormal{inf}}(\theta):=
		\left(
		\begin{array}{c c}
			1 & -\theta \\
			\theta & 1
		\end{array}
		\right).
	\end{equation*}
	By virtue of \eqref{42}--\eqref{44} we estimate
	\begin{align*}
		\Vert V-R_{\textnormal{inf}}(\theta)U \Vert_{L^2(Q')} & \leq C(K)\mathcal{L}^2(\{ \dist(\beta,SO(2)U) > \rho \} \cap Q')^{\frac{1}{2}}+\Vert R(\theta)U-R_{\textnormal{inf}}(\theta)U \Vert_{L^2(Q')} \\
		& \leq C(K) \left( \Vert \beta-U \Vert_{L^{2,\infty}(Q')} +\Vert \theta \Vert_{L^{2,\infty}(Q')} \right) \\
		& \leq C(K) \Vert \beta-U \Vert_{L^{2,\infty}(Q')}.
	\end{align*}
	Therefore, by definition of $V$ we infer
	\begin{equation}\label{45}
		\Vert \beta - R_{\textnormal{inf}}(\theta)U \Vert_{L^2(Q')} \leq  \Vert \dist(\beta,K) \Vert_{L^{2}(Q)}+C(K) \Vert \beta-U \Vert_{L^{2,\infty}(Q')}.
	\end{equation}

	From now on, we set for brevity $\varepsilon:=\Vert \dist(\beta,K) \Vert_{L^{2}(Q)}+\Vert \divr \beta \Vert_{L^{1}(Q)}+\Vert \curl \beta \Vert_{L^{1}(Q)}$. Define the matrix field $h:=\beta-R_{\textnormal{inf}}(\theta)U$, thanks to \eqref{40} and \eqref{45} we have $\Vert h \Vert_{L^2(Q')} \leq C(K,\delta) \varepsilon$. Since $R_{\textnormal{inf}}(\theta)$ is the sum of the identity and a skew symmetric matrix, we can write
	$$
	R_{\textnormal{inf}}(\theta)U=U+\cof(U)(R_{\textnormal{inf}}(\theta)-I).
	$$
	Hence, $\cof(U)^{-1}\beta=\cof(U)^{-1}U+R_{\textnormal{inf}}(\theta)-I+\cof(U)^{-1}h$. Taking the curl gives
	$$
	\curl \left(\cof(U)^{-1}\beta \right) = -\nabla \theta+\curl \left( \cof(U)^{-1}h \right) \qquad \mbox{on $Q'$}.
	$$
	If we first multiply on the left the expression above with $J$ and then take the divergence we get
	$$
	\divr \left( J \left( \curl \left(\cof(U)^{-1}\beta \right) \right) \right)=\divr \left( J \left( \curl \left(\cof(U)^{-1} h \right) \right) \right) \qquad \mbox{on $Q'$},
	$$
	which in turn implies
	\begin{equation}\label{46}
		\Vert \divr \left( J \left( \curl \left(\cof(U)^{-1}\beta \right) \right) \right) \Vert_{H^{-2}(Q')} \leq C \Vert h \Vert_{L^2(Q')} \leq C(K,\delta) \varepsilon. 
	\end{equation}
	Hence, combining Theorem \ref{thm: n=2 BB} with $f=J\left( \curl \left(\cof(U)^{-1}\beta \right) \right)$ and \eqref{46} gives
	\begin{equation}\label{47}
		\Vert J \left( \curl \left(\cof(U)^{-1}\beta \right) \right)  \Vert_{H^{-1}(Q')} \leq C(K,\delta) \varepsilon.
	\end{equation}

	We now define a new Dirichlet problem on the smaller cube $Q'$. Let $\tilde{v} \in H^1_0(Q',\R^2)$ be the unique solution to
	\begin{equation*}
		\begin{cases*}
			\Delta \tilde{v} = -\curl \left( \cof(U)^{-1} \beta \right) & on $Q'$, \\
			\tilde{v}=0 & on $\partial Q'$.
		\end{cases*}
	\end{equation*}
	As $\curl \left( \cof(U)^{-1} \beta \right) \in H^{-1}(\Omega;\R^2)$ with the estimate in \eqref{47}, we have that $\Vert \nabla \tilde{v} \Vert_{L^2(Q')} \leq C(K,\delta) \varepsilon$.
	Next, a simple computation shows that $\curl \left( \cof(U)^{-1} \beta \right) = \cof(U)^{-1} \curl \beta$, hence on $Q'$ we have
	\begin{equation*}
	\divr \left( \cof(U) \nabla \tilde{v} \right)= \cof(U) \divr(\nabla \tilde{v})=\cof(U) \Delta \tilde{v}=-\cof(U)\curl \left( \cof(U)^{-1} \beta \right) = -\curl \beta.
	\end{equation*}
	Therefore, if we define the matrix field $\tilde{Y}:=\cof(U) \nabla \tilde{v} J$, we have that $\curl(\beta-\tilde{Y})=0$ on $Q'$. Since $Q'$ is simply connected, we can find $\tilde{u} \in H^1(Q';\R^2)$ such that $\beta-\tilde{Y}=\nabla \tilde{u}$ $\mathcal{L}^2$-a.e. on $Q'$. Moreover, it holds that
	\begin{equation}\label{48}
		\divr \tilde{Y} = \divr \left( \cof(U) \nabla \tilde{v} J \right)=\divr \left( \nabla (\cof(U) \tilde{v}) J \right)=0.
	\end{equation}
	Thus, by \eqref{48}
	$$
	\Vert \dist(\nabla \tilde{u},K) \Vert_{L^2(Q')}+ \Vert \Delta \tilde{u} \Vert_{L^1(Q')} \leq C(K) \Vert \nabla \tilde{v} \Vert_{L^2(Q')}+ \Vert \dist(\beta,K) \Vert_{L^2(Q')}+\Vert \divr \beta \Vert_{L^1(Q')} \leq C(K,\delta) \varepsilon.
	$$
	By Theorem \ref{thm: multiwell rigidity crit} applied to $\nabla \tilde{u}$ on $Q'$ there exists $F \in K$ such that
	$$
	\Vert \nabla \tilde{u}-F \Vert_{L^2(Q')} \leq C(K,\delta) \varepsilon,
	$$
	which in turn implies that
	$$
	\Vert \beta-F \Vert_{L^2(Q')} \leq C(K,\delta) \varepsilon.
	$$
	
	\noindent \textit{Estimate up to the boundary.}
	
	\noindent In the previous steps, we have proven in general for $n \geq 2$ that for every unit cube $Q$ and $Q'\Subset Q$ a smaller cube concentric to $Q$ with $\delta:=\dist(Q',\partial Q)$, there exists a constant $C=C(K,n,\delta)>0$ such that for every $\beta \in L^q(Q;\R^{n \times n})$ there exists an associated matrix $F \in K$ with the property that
	\begin{equation*}
		\Vert \beta - F \Vert_{L^q(Q';\R^2{n \times n})} \leq C(K,n,\delta)( \Vert \dist(\beta,K) \Vert_{L^q(Q)}+\Vert \divr \beta \Vert_{L^1(Q)}+\Vert \curl \beta \Vert_{L^1(Q)}).
	\end{equation*}
	Moreover, the constant $C(K,n,\delta)$ is invariant after translation and rescaling of the cube. Notice that in Step 3 in the proof of Theorem \ref{thm: multiwell rigidity crit} we have not used at all the fact that $\nabla u$ is a gradient field, hence, we can conclude reasoning exactly in the same way, up to replacing $\nabla u$ with $\beta$ and $\Vert \Delta u \Vert_{L^1}$ with $\Vert \divr \beta \Vert_{L^1}+\Vert \curl \beta \Vert_{L^1}$.

\end{proof}

\section{$\Gamma$-limit of an infinite dislocations model}\label{sez 4}

As an application of the previous result, we perform the $\Gamma$-convergence limit for an infinite dislocations model in the presence of a multi-well elastic energy. We start by introducing some notations.

Let $\Omega \subset \R^2$ be a simply connected and bounded Lipschitz domain representing a horizontal section of an infinite cylindrical crystal. Let $S:=\{ \hat{b}_1,\hat{b}_2 \}$ be a set of admissible normalized Burgers vectors for the crystal: $\hat{b}_1,\hat{b}_2 \in \R^2$, are linearly independent and are dependent on the crystalline structure. We also set $\mathbb{S}:=\textnormal{Span}_{\mathbb{Z}}S$, the span of $S$ with integer coefficients, i.e. the set of Burgers vectors for multiple dislocations. Every dislocation is characterized by a point in the plane $x_i \in \Omega$ and a vector $\xi_i \in \mathbb{S}$.

For a given crystal, we denote the interatomic distance with $\varepsilon>0$. We also assume that the distance between two dislocations is bounded from below in terms of an intermediate scale $\rho_\varepsilon \gg \varepsilon$ such that $\rho_\varepsilon \to 0$ as $\varepsilon \to 0$. This implies that dislocations are well separated (with respect to the atomic spacing $\varepsilon$), that is, there is a scale separation between the scale of the atomic lattice $\varepsilon$ and the scale of the dislocations distributions $\rho_\varepsilon$. Analogously to \cite{GaLeoPonDislocations, MuScaZep}, we assume the following:
\begin{enumerate}[label=($\rho$.\arabic*),ref=$\rho$.\arabic*]
	\item \label{rho1} $\lim_{\varepsilon\to 0} \rho_\varepsilon/\varepsilon^s=+\infty$ for every $s \in (0,1)$; 
	\item \label{rho2} $\lim_{\varepsilon \to 0} |\log \varepsilon| \rho_\varepsilon^2 =0$.
\end{enumerate} 

Under this assumptions on $\rho_\varepsilon$, we will show, as in the one well case (see \cite{MuScaZep}), that in the limit the energy can
be decomposed into two contributions: a self energy concentrated in the hard-core regions $B_{\rho_\varepsilon}(x_i)$ and an interaction energy stored outside the union of the hard-core regions.

We define the class $X_\varepsilon$ of the admissible dislocation densities as
\begin{equation}\label{X epsilon}
	X_\varepsilon:= \left\{  \mu \in \mathcal{M}(\Omega;\R^2) \colon \mu= \sum_{i=1}^M \varepsilon \xi_i \delta_{x_1}, \ M \in \N, \ B_{2\rho_\varepsilon}(x_i) \subset \Omega, \ |x_j - x_k| \geq 2 \rho_\varepsilon \ \mbox{for every $j \neq k$}, \ \xi_i \in \mathbb{S} \right\}.
\end{equation}
For a given measure $\mu \in X_\varepsilon$ and $r>0$ we also define 
\begin{equation}\label{set minus balls}
	\Omega_r(\mu):= \Omega \setminus \bigcup_{x_i \in \supp(\mu)} \overline{B_r(x_i)}.
\end{equation}
The class of admissible strains associated with $\mu \in X_\varepsilon$ is given by those $\beta \in L^2(\Omega_\varepsilon(\mu);\R^{2 \times 2})$ satisfying
$$
\curl \beta =0 \ \mbox{in} \ \Omega_\varepsilon(\mu) \qquad \mbox{and} \qquad \int_{\partial B_\varepsilon(x_i)} \beta t\, ds=\varepsilon \xi_i, \ \mbox{for $i=1,\dots,M,$}
$$
where $\curl \beta=0$ is intended in a distributional sense. The vector $t$ above denotes the oriented tangent vector to $\partial B_\varepsilon(x_i)$ and the integrand $\beta t$ is intended in the sense of traces (which is well defined thanks to \cite[Theorem 2, pag. 204]{DautrayLions90}).

As customary, it is useful to extend $\beta$ to the whole $\Omega$ setting $\beta \equiv I$ in the balls $B_\varepsilon(x_i)$. Hence, the class of admissible strains associated with a measure $\mu \in X_\varepsilon$ for the perturbed energy is given by
\begin{align}\label{Admissible strains}
	\begin{split}
		\mathcal{AS}_\varepsilon(\mu):= \left\{ \vphantom{\int} \right. \vphantom{\int} \beta \in L^2(\Omega;\R^{2 \times 2}) \colon & \beta \equiv I \ \mbox{in} \ \cup_{i=1}^M B_\varepsilon(x_i), \  \curl \beta=0 \ \mbox{on} \ \Omega_\varepsilon(\mu), \\
		&\left. \divr \beta \in L^2(\Omega_\varepsilon(\mu);\R^2),  \ \int_{\partial B_\varepsilon(x_i)}  \beta t \, d\mathcal{H}^1 = \varepsilon \xi, \ \mbox{for} \ i=1,\dots,M \right\}.
	\end{split}
\end{align} 

We recall that $K=\cup_{i=1}^\ell K_i$ with $K_i=SO(2)U_i$ where $U_i \in \R^{2 \times 2}$ are invertible matrices for every $i=1,\dots,\ell$. Moreover, we can assume that $U_1=I$. The elastic energy $W \colon \R^{2 \times 2} \to [0,+\infty)$ satisfies the following properties:
\begin{enumerate}[label=($W$\arabic*),ref=$W$\arabic*]
	\item \label{W1} $W \in C^0(\R^{2 \times 2})$ and $W \in C^2$ in a neighborhood of $K$;
	\item \label{W2} $W(U)=0$ for every $U \in \{U_1 \dots,U_\ell\}$;
	\item \label{W3} $W$ is frame indifferent, that is $W(RF)=W(F)$ for every $F \in \R^{2 \times 2}$ and $R \in SO(2)$;
	\item  \label{W4} there exists two constant $0<C_1 \leq C_2 <+\infty$ such that for every $F \in \R^{2 \times 2}$ it holds
	$$
	C_1 \ \dist^2(F,K) \leq W(F)\leq C_2 \ \dist^2(F,K).
	$$
\end{enumerate}

Under the assumptions \eqref{W1}--\eqref{W4} we define for $\mu \in X_\varepsilon$ and $\beta \in \mathcal{AS}_\varepsilon(\mu)$ the singularly perturbed energy  
$$
\int_{\Omega_\varepsilon(\mu)} W(\beta) \, dx + \eta_\varepsilon^2 \int_{\Omega_\varepsilon(\mu)} |\divr \beta|^2 \, dx,
$$
where $\eta_\varepsilon>0$ is a suitable parameter such that $\eta_\varepsilon \to 0$ as $\varepsilon \to 0$. We notice that the energy is finite for $\beta \in \mathcal{AS}_\varepsilon(\mu)$ thanks to \eqref{W4}. The higher order penalization term is crucial for compactness (see Proposition \ref{prop: compactness}).

Using \eqref{W2} we can rewrite the perturbed energy associated to an (extended) admissible strain $\beta \in \mathcal{AS}_\varepsilon(\mu)$ as
$$
E_\varepsilon(\mu,\beta):= \int_\Omega W(\beta) \, dx + \eta_\varepsilon^2 \int_{\Omega_\varepsilon(\mu)} |\divr \beta|^2 \, dx.
$$
As in \cite{MuScaZep} the relevant scaling for the energy is $\varepsilon^2 |\log \varepsilon|^2$, thus, we consider the scaled nonlinear perturbed dislocation energy given by
\begin{equation}\label{scaled energy}
	\mathcal{E}_\varepsilon(\mu,\beta) : = 
	\begin{cases*}
		\displaystyle \frac{1}{\varepsilon^2 |\log \varepsilon|^2} E_\varepsilon(\mu,\beta) & if $\mu \in X_\varepsilon$, $\beta \in \mathcal{AS}_\varepsilon(\mu)$, \\
		+\infty & otherwise in $\mathcal{M}(\Omega;\R^{2}) \times L^2(\Omega;\R^{2 \times 2})$.
	\end{cases*}
\end{equation}
Moreover, for every $\varepsilon>0$ and every $\Omega' \subseteq \Omega$ measurable, let us also denote 
$$
\mathcal{E}_\varepsilon(\mu,\beta;\Omega'):= \frac{1}{\varepsilon^2|\log \varepsilon|^2} \left( \int_{\Omega' \cap \Omega_\varepsilon(\mu)} W(\beta) \, dx+\eta_\varepsilon \int_{\Omega' \cap \Omega_\varepsilon(\mu)} |\divr \beta|^2 \, dx \right).
$$

Fix $\gamma \in (0,1)$, we take the penalization $\eta_\varepsilon$ such that
\begin{equation}\label{eta range}
	\frac{\varepsilon|\log \varepsilon|}{\rho_\varepsilon} \leq \eta_\varepsilon \leq \varepsilon^\gamma.
\end{equation}
Under this assumption on $\eta_\varepsilon$, we will show that, as in the one well case, sequences with bounded energy are compact in some topology.

\subsection{Compactness for Sequences with Bounded Energy}

In order to prove the compactness result, we need some technical lemmas.
First, we deal with the shrinking of Lipschitz domains. For a complete treatment of the topic see \cite[Sections 2.1 and 2.2]{ErnGuermond} and \cite{HofMitTay}. As we will use this technique repeatedly in the proof of compactness, we have to make sure that the constants involved remain bounded. This will be subject of investigation also in Appendix A.

\begin{lem}\label{lem: shrinking lip}
	Let $\Omega \subset \R^2$ be a bounded Lipschitz domain and simply connected. For every $\delta >0$ small enough, there exists a bounded Lipschitz domain $\Omega^\delta$ such that:
	\begin{itemize}
		\item[\textnormal{(i)}] the set $\Omega^\delta$ is simply connected and there exists $c>0$ not depending on $\delta$ such that 
		$$\{ x \in \Omega \colon \dist(x,\partial \Omega) > \delta \} \subset \Omega^\delta \Subset \Omega^\delta+B_{c\delta} \subset \Omega;$$
		\item[\textnormal{(ii)}] given a set $E \subset \Omega^\delta$ of finite perimeter, there exists $C(\Omega)>0$ such that 
		$$\min \left\{ \mathcal{L}^2(E) ,\mathcal{L}^2(\Omega \setminus E)\right\} \leq C(\Omega) P\left(E,\Omega^\delta\right)^2;$$
		\item[\textnormal{(iii)}] given $p \in (1,\infty)$, for every $f \in W^{1,p}_{\textnormal{loc}}(\Omega^\delta;\R^{2 \times 2})$ there exists a constant $C'(\Omega,p)>0$ such that
		$$
		\min_{F \in \R^{2 \times 2}} \int_{\Omega^\delta} |f-F|^p \, dx \leq C'(\Omega,p) \int_{\Omega^\delta} |\nabla f|^p \dist^p\left(x,\partial \Omega^\delta\right) \, dx.
		$$
	\end{itemize}
\end{lem}

\begin{proof}
	Using \cite[Lemma 2.1]{ErnGuermond} we get the existence of a diffeomorphism $\Phi_\delta \in C^\infty(\R^2,\R^2)$ such that $\Phi_\delta(\Omega) +B_{c \delta} \subset \Omega$ for some $c>0$ not depending on $\delta$, and 
	\begin{equation}\label{l0}
	\sup_{x \in \Omega} \left(|D^k \Phi_\delta-D^k \textnormal{Id}|+|D^k \Phi_\delta^{-1}-D^k \textnormal{Id}| \right) \leq C \delta,
	\end{equation}
	for every $k \in \N$ and for some $C>0$, where $D^k$ is the Fréchet derivative of order $k$. Set $\Omega^\delta:=\Phi_\delta(\Omega)$. Property (i) is immediate. To prove (ii) observe that given $E \subset \Omega^\delta$ with finite perimeter and such that $\mathcal{L}^2(E) \leq \mathcal{L}^2(\Omega^\delta)/2$, thanks to \eqref{l0}, for $\delta$ small we have
	$$
	P\left(\Phi_\delta^{-1}(E),\Omega \right) \leq \max_{x \in \Omega^\delta} |D \Phi_\delta^{-1}(x)| P(E,\Omega^\delta) \leq 3 P(E,\Omega^\delta).
	$$
	By area formula and \eqref{l0}, we estimate
	\begin{equation}\label{l3}
	\frac{\mathcal{L}^2(E)}{2} \leq \int_E |\det D\Phi^{-1}_\delta| \, dx= \mathcal{L}^2\left(\Phi_\delta^{-1}(E)\right)=\int_E |\det D\Phi^{-1}_\delta| \, dx \leq 2\mathcal{L}^2(E).
	\end{equation}
	From \eqref{l3} we infer that if $\mathcal{L}^2(E) \leq \mathcal{L}^2(\Omega^\delta)/8$, then $\mathcal{L}^2(\Phi_\delta^{-1}(E)) \leq \mathcal{L}^2(\Omega)/2$. Hence, using the relative isoperimetric inequality on $\Omega$, for every $\delta>0$ small enough we have
	$$
	\frac{\mathcal{L}^2(E)}{2}^2 \leq 2	\mathcal{L}^2\left(\Phi_\delta^{-1}(E)\right) \leq 2C(\Omega) P\left(\Phi_\delta^{-1}(E),\Omega \right)^2 \leq 18 C(\Omega) P\left(E,\Omega^\delta\right).
	$$
	If instead $\min \left\{ \mathcal{L}^2(E) ,\mathcal{L}^2(\Omega^\delta \setminus E)\right\} > \mathcal{L}^2(\Omega^\delta)/8$, we observe that $\mathcal{L}^2(E) \leq 8 \min \left\{ \mathcal{L}^2(E) ,\mathcal{L}^2(\Omega^\delta \setminus E)\right\}$ and conclude reasoning as above.
	
	For (iii) we begin observing that $\dist \left(x,\partial \Omega \right) \leq \max_{x \in \Omega^\delta}|D \Phi_\delta^{-1}(x)| \dist\left(\Phi_\delta(x),\partial \Omega^\delta\right) \leq 3\dist\left(\Phi_\delta(x),\partial \Omega^\delta\right)$ for every $x \in \Omega$. Let $f \in W^{1,p}_{\textnormal{loc}}(\Omega^\delta;\R^{2 \times 2})$. Since $\Omega$ is a Lipschitz domain, by the weighted Poincaré inequality there exists a constant $C'(\Omega,p)>0$ such that
	$$
	\min_{F \in \R^{2 \times 2}} \int_{\Omega} |f \circ \Phi_\delta-F|^p \, dx \leq C'(\Omega,p) \int_{\Omega} |\nabla (f \circ \Phi_\delta)|^p \dist^p\left(x,\partial \Omega\right) \, dx.
	$$
	A change of coordinates gives
	$$
	\min_{F \in \R^{2 \times 2}} \int_{\Omega^\delta} |f -F|^p |\det D\Phi_\delta^{-1}(y)| \, dy \leq 3^p C'(\Omega) \int_{\Omega^\delta} |\nabla f|^p |D\Phi_\delta \circ \Phi_\delta^{-1}(y)|^p  \ \dist^p \left(y,\partial \Omega^\delta\right) |\det D\Phi_\delta^{-1}(y)| \, dy.
	$$
	Recalling \eqref{l0}, we conclude (iii) for every $\delta$ small enough. 
\end{proof}

We also need an elliptic regularity result for strong and weak $L^2$ norms.

\begin{lem}\label{lem: elliptic regularity}
	Let $\Omega \subset \R^2$ be a bounded Lipschitz domain. For $\delta>0$ set $\Omega_\delta:= \{ x \in \Omega \colon \dist(x,\partial \Omega)>\delta \}$. There exists a constant $C>0$ such that for every $u \in H^1(\Omega;\R^2)$ with $\Delta u \in L^2(\Omega;\R^2)$ the following estimates hold
	\begin{align*}
		& \Vert \nabla^2 u \Vert_{L^2(\Omega_\delta)} \leq C \, \frac{1+\diam (\Omega)}{\delta} \left(\Vert \nabla u \Vert_{L^2(\Omega)}+ \Vert \Delta u \Vert_{L^2(\Omega)} \right), \\
		& \Vert \nabla^2 u \Vert_{L^{2,\infty}(\Omega_\delta)} \leq C \, \frac{1+\diam (\Omega)}{\delta} \left(\Vert \nabla u \Vert_{L^{2,\infty}(\Omega)}+ \Vert \Delta u \Vert_{L^2(\Omega)} \right).
	\end{align*}
\end{lem}

\begin{proof}
	We shall prove only the second inequality, as the first one is well known from elliptic regularity theory. Let $B_\Omega \subset \R^2$ be a ball such that $\Omega \subset B_\Omega$ and $\diam B_\Omega \leq 2 \diam (\Omega)$. Consider $v \in H^1_0(B_\Omega;\R^2)$ the unique solution to the system 
	\begin{equation*}
		\begin{cases*}
			\Delta v=\chi_\Omega \, \Delta u & on $B_\Omega$, \\
			v =0 & on $\partial B_\Omega$,
		\end{cases*}
	\end{equation*}
	and set $w=u-w$ on $\Omega$.
	Since $B_\Omega$ is a smooth domain, there exists an universal constant $C>0$ such that $\Vert \nabla^2 v \Vert_{L^2(B_\Omega)} \leq C \Vert \Delta u \Vert_{L^2(\Omega)}$ and $\Vert \nabla v \Vert_{L^2(B_\Omega)} \leq C \, \diam (\Omega) \Vert \Delta u \Vert_{L^2(\Omega)}$. As $w$ is harmonic on $\Omega$, we can use Caccioppoli inequality for $w$ with the $L^{2,\infty}$ norm (see Lemma \ref{weak Caccio est}), getting 
	$$
	\Vert \nabla^2 w \Vert_{L^{2,\infty}(\Omega_\delta)} \leq \frac{C}{\delta} \Vert \nabla w \Vert_{L^{2,\infty}(\Omega)}.
	$$
	Hence,
	\begin{align*}
		\Vert \nabla^2 u \Vert_{L^{2,\infty}(\Omega_\delta)} & \leq \Vert \nabla^2 w \Vert_{L^{2,\infty}(\Omega_\delta)}+ \Vert \nabla^2 v  \Vert_{L^{2,\infty}(\Omega_\delta)} \\ 
		& \leq \frac{C}{\delta} \Vert \nabla w \Vert_{L^{2,\infty}(\Omega)}+ C\Vert \Delta u \Vert_{L^2(\Omega)} \\
		& \leq \frac{C}{\delta} \Vert \nabla u \Vert_{L^{2,\infty}(\Omega)}+\frac{C}{\delta} \Vert \nabla v \Vert_{L^2(\Omega)}+ C\Vert \Delta u \Vert_{L^2(\Omega)} \\
		& \leq C \, \frac{1+ \diam (\Omega)}{\delta} \left(\Vert \nabla u \Vert_{L^{2,\infty}(\Omega)}+ \Vert \Delta u \Vert_{L^2(\Omega)} \right),
	\end{align*}
	which is the desired inequality.
\end{proof}

Finally, we need the following rigidity lemma involving a domain with a "hole and a cut" in the spirit of \cite[Proposition 3.3]{ScaZep}.

\begin{lem}\label{lem: rigidity non equi spaced annulus}
	There exists $C=C(K)>0$ with the following property: for every $0<4r_1<r_2<1$ and every $u \in H^1(B_{r_2} \setminus B_{r_1};\R^2)$ with $\Delta u \in L^2(B_{r_2} \setminus B_{r_1};\R^2)$ there exists an associated matrix $F \in K$ such that
	\begin{equation}\label{annulus rigidity}
		\Vert \nabla u - F \Vert_{L^2(B_{r_2} \setminus B_{r_1};\R^{2 \times 2})} \leq C \Vert \dist(\nabla u,K) \Vert_{L^2(B_{r_2} \setminus B_{r_1})}+r_2 C \Vert \Delta u \Vert_{L^2(B_{r_2} \setminus B_{r_1};\R^{2})}.
	\end{equation}
	The same result holds if instead of the full annulus we have the annulus with a cut, e.g. $(B_{r_2} \setminus B_{r_1}) \setminus L$, where $L:=\{ (z,0) \in \R^2 \colon r_1<z<r_2 \}$.
\end{lem}

\begin{proof}
	Using Theorem \ref{thm: multiwell rigidity crit} applied to $u$ restricted in $B_{4r_1} \setminus B_{r_1}$ we get the existence of $C=C(K)>0$ (by scaling invariance) and $V \in K$ such that
	\begin{equation}\label{l1}
		\int_{B_{4r_1} \setminus B_{r_1}} |\nabla u - V|^2 \, dx \leq C \int_{B_{4r_1} \setminus B_{r_1}} \dist^2(\nabla u,K) \, dx + C \left( \int_{B_{4r_1} \setminus B_{r_1}} |\Delta u| \, dx \right)^2.
	\end{equation}
	We now want to find a modification of $u$ on the whole $B_{r_2}$. To this aim, define $v \in H^1(B_4 \setminus B_1;\R^2)$ as $v(x):=u(r_1 x)-r_1Vx$. Notice that by interior elliptic regularity, $v \in H^2(B_3 \setminus B_2;\R^2)$ and there exists $C>0$ depending only on the set $B_3 \setminus B_2$ such that
	\begin{equation}\label{l1.5}
	\int_{B_3 \setminus B_2} |\nabla^2 v|^2 \, dx \leq C \int_{B_4 \setminus B_1} |\nabla v|^2 \, dx + C \int_{B_4 \setminus B_1} |\Delta v|^2 \, dx.
	\end{equation}
	Using Stein extension Theorem we find $C>0$ depending only on the set $B_3 \setminus B_2$ and a function $\tilde{v} \in H^2(B_3;\R^2)$ such that $\tilde{v}=v$ on $B_3 \setminus B_2$,
	\begin{align}\label{n1}
	\begin{split}
	\int_{B_3} |\nabla \tilde{v}|^2 \, dx & \leq C \int_{B_3 \setminus B_2} |\nabla v|^2 \, dx = C r_1^2 \int_{B_3 \setminus B_2} |\nabla u(r_1 x)-V|^2 \, dx=\int_{B_{3r_1} \setminus B_{2r_1}} |\nabla u - V|^2 \, dx \\
	& \leq C \int_{B_{4r_1} \setminus B_{r_1}} \dist^2(\nabla u,K) \, dx + C \left( \int_{B_{4r_1} \setminus B_{r_1}} |\Delta u| \, dx \right)^2
	\end{split}
	\end{align}
	where in the last inequality we used \eqref{l1}. By \eqref{l1.5},
	\begin{align}\label{n2}
	\begin{split}
	\int_{B_3} |\nabla^2 \tilde{v}|^2 \, dx & \leq C r_1^2 \int_{B_4 \setminus B_1} |\nabla u(r_1 x)-V|^2 \, dx + C r_1^4 \int_{B_4 \setminus B_1} |\Delta u(r_1 x)|^2 \, dx \\
	& \leq  C  \int_{B_{4 r_1} \setminus B_{r_1}} |\nabla u-V|^2 \, dx+C r_1^2 \int_{B_{4 r_1} \setminus B_{r_1}} |\Delta u|^2 \, dx \\
	& \leq C \int_{B_{4r_1} \setminus B_{r_1}} \dist^2(\nabla u,K) \, dx + C r_1^2 \int_{B_{4 r_1} \setminus B_{r_1}} |\Delta u|^2 \, dx.
	\end{split}
	\end{align}
	Hence, if we define the function 
	\begin{equation*}
		\tilde{u}(x):=
		\begin{cases*}
			\tilde{v}(x/r_1)+Vx & on $B_{2 r_1}$, \\
			u(x) & on $B_{r_2} \setminus B_{2 r_1}$,
		\end{cases*}
	\end{equation*}
	we estimate 
	\begin{align}\label{l1.6}
		\begin{split}
		\int_{B_{2 r_1}} \dist^2(\nabla \tilde{u},K) \, dx & \leq \int_{B_{2 r_1}} \frac{1}{r_1^2} \left|\nabla \tilde{v}\left(\frac{x}{r_1}\right)\right|^2 \, dx = \int_{B_{2}} \left|\nabla \tilde{v}\left(y\right)\right|^2 \, dy \\
		& \leq  C \int_{B_{4r_1} \setminus B_{r_1}} \dist^2(\nabla u,K) \, dx + C \left( \int_{B_{4r_1} \setminus B_{r_1}} |\Delta u| \, dx \right)^2,
		\end{split}
	\end{align}
	and 
	\begin{align}\label{l1.7}
		\begin{split}
		\left( \int_{B_{2 r_1}} |\Delta \tilde{u}| \, dx\right)^2 & \leq C r_1^2 \int_{B_{2 r_1}} |\Delta \tilde{u}|^2 \, dx \leq \frac{C}{r_1^2} \int_{B_{2 r_1}} \left| \nabla^2 \tilde{v} \left( \frac{x}{r_1} \right) \right|^2 \, dx = C \int_{B_{2}} \left|  \nabla^2 \tilde{v} \left( y \right) \right|^2 \, dy \\
		& \leq C \int_{B_{4r_1} \setminus B_{r_1}} \dist^2(\nabla u,K) \, dx+ Cr_1^2 \int_{B_{4 r_1} \setminus B_{r_1}} |\Delta u|^2 \, dx.
		\end{split}
	\end{align}
	We now apply Theorem \ref{thm: multiwell rigidity crit} to $\tilde{u}$ on $B_{r_2}$ finding a constant $C>0$ and a matrix $F \in K$ such that
	\begin{align*}
			\int_{B_{r_2}} |\nabla \tilde{u}-F|^2 \, dx & \leq C \int_{B_{r_2}} \dist^2(\nabla \tilde{u},K) \, dx + C \left( \int_{B_{r_2}} |\Delta \tilde{u}| \, dx \right)^2 \\
			& \leq C \int_{B_{2r_1}} \dist^2(\nabla \tilde{u},K) \, dx + C \int_{B_{r_2} \setminus B_{2r_1}} \dist^2(\nabla u,K) \, dx+C \left( \int_{B_{r_2}} |\Delta \tilde{u}| \, dx \right)^2 \\
			& \leq C \int_{B_{r_2} \setminus B_{r_1}} \dist^2(\nabla u,K)^2 \, dx+ C \left( \int_{B_{2r_1}} |\Delta \tilde{u}| \, dx+\int_{B_{r_2} \setminus B_{r_1}} |\Delta u| \, dx \right)^2 \\
			& \leq C \int_{B_{r_2} \setminus B_{r_1}} \dist^2(\nabla u,K)^2 \, dx+r_2^2 C \int_{B_{r_2} \setminus B_{r_1}} |\Delta u|^2 \, dx.
	\end{align*}
	where we have also used \eqref{l1.6}, \eqref{l1.7} and H\"older inequality.
	Thus, by definition of $\tilde{u}$ we have that 
	\begin{equation}\label{l2}
		\int_{B_{r_2}\setminus B_{2r_1}} |\nabla u-F|^2 \, dx \leq C \int_{B_{r_2} \setminus B_{r_1}} \dist^2(\nabla u,K) \, dx+r_2^2 C \int_{B_{r_2} \setminus B_{r_1}} |\Delta u|^2 \, dx.
	\end{equation}
	On one hand, using \eqref{l1}, \eqref{l2} and triangular inequality we infer 
	$$
	r_1^2 |F-V|^2 \leq  C \int_{B_{r_2} \setminus B_{r_1}} \dist^2(\nabla u,K) \, dx+r_2^2 C \int_{B_{r_2} \setminus B_{r_1}} |\Delta u|^2 \, dx;
	$$
	on the other hand, using again a triangular estimate we have
	$$
	\int_{B_{r_2} \setminus B_{r_1}} |\nabla u-F|^2 \, dx \leq \int_{B_{r_2} \setminus B_{2r_1}} |\nabla u-F|^2 \, dx+ 2\int_{B_{2r_1} \setminus B_{r_1}} |\nabla u-V|^2 \, dx+ C r_1^2 |F-V|^2.
	$$
	Therefore, \eqref{annulus rigidity} is proven. For the last part of the statement we can argue using a covering with overlapping of the domain in the spirit of \cite[Proposition 3.3]{ScaZep}.
\end{proof}

We now are ready to prove compactness for sequences $(\mu_\varepsilon,\beta_\varepsilon)$ with bounded energy $\mathcal{E}_\varepsilon$.

\begin{prop}[Compactness]\label{prop: compactness}
	Let $\varepsilon_j \to 0$ and let $(\mu_j,\beta_j) \in \mathcal{M}(\Omega;\R^2) \times L^2(\Omega;\R^2)$ be a sequence such that $\sup_j \mathcal{E}_{\varepsilon_j} (\mu_j,\beta_j) < +\infty$. Then, there exists $U \in \{ U_1 \dots U_\ell \}$, a sequence of constant rotations $(R_j)_j \in SO(2)$, $R \in SO(2)$, a measure $\mu \in \mathcal{M}(\Omega;\R^2) \cap H^{-1}(\Omega;\R^2)$ and a field $\beta \in L^2(\Omega;\R^{2 \times 2})$ such that, up to subsequences, $R:=\lim_{j \to +\infty} R_j$,
	\begin{equation}\label{49}
		\frac{\mu_j}{\varepsilon_j|\log \varepsilon_j|} \xrightharpoonup{*} \mu  \qquad \mbox{ in $\mathcal{M}(\Omega;\R^2)$}, 
	\end{equation}
	\begin{equation}\label{50}
		\chi_{\Omega_j}\frac{R^T_j \beta_j - U}{\varepsilon_j |\log \varepsilon_j|} \rightharpoonup \beta \qquad \mbox{ in $L^2(\Omega;\R^{2 \times 2})$};
	\end{equation}
	where $\Omega_j:=\{ x \in \Omega \colon \dist(x,\partial \Omega)>\rho_{\varepsilon_j}/2 \}$
	and $\curl \beta = R^T \mu$.
\end{prop}

\begin{proof}
	Let $(\mu_j,\beta_j) \in \mathcal{M}(\Omega;\R^2) \times L^2(\Omega;\R^{2 \times 2 })$ be a sequence such that $\mathcal{E}_{\varepsilon_j}(\mu_j,\beta_j) \leq C$ for some positive constant $C>0$ independent of $j$, where
	$$
	\mu_j=\sum_{i=1}^{M_j} \varepsilon_j \xi_{i,j} \delta_{x_{i,j}},
	$$
	with $\xi_{i,j} \in \mathbb{S}$, $x_{i,j} \in \Omega$ such that $B_{\rho_{\varepsilon_j}}(x_{i,j}) \subset \Omega$ and $|x_{i,j}-x_{k,j}|\geq 2 \rho_{\varepsilon_j}$ for every $i \neq k$.
	Using the growth conditions \eqref{W4} on $W$ we have 
	\begin{equation}\label{51}
		\int_{\Omega} \dist^2(\beta_j,K) \, dx + \eta^2_{\varepsilon_j} \int_{\Omega_{\varepsilon_j}(\mu_j)} |\divr \beta_j|^2 \, dx \leq C \varepsilon_j^2|\log \varepsilon_j|^2.
	\end{equation}
	We split the proof into three main steps.
	
	\noindent \textbf{Step 1.} \textit{Weak convergence of the scaled dislocation measures.} 
	We begin by showing that the sequence $\mu_j/(\varepsilon_j|\log \varepsilon_j|)$ is bounded in mass. In particular, our aim is to show that there exists $C>0$ such that for every $j \geq 1$ large 
	\begin{equation}\label{m1}
		\frac{1}{\varepsilon_j |\log \varepsilon_j|} |\mu_j|(\Omega) \leq C.
	\end{equation}
	We further subdivide the proof of \eqref{m1} into two sub-steps.
	
	\noindent \textbf{Step 1.1.} \textit{Modification of $\beta_j$ inside $B_{\rho_{\varepsilon_j}}(x_{i,j})$.}
	Fix $x_{i,j} \in \Omega$. We start by modifying $\beta_j$ inside $B_{\rho_{\varepsilon_j}}(x_{i,j})$ in order to apply a rigidity estimate. Set $C_{i,j}:=B_{4\eta_{\varepsilon_j}}(x_{i,j}) \setminus B_{\eta_{\varepsilon_j}}(x_{i,j})$ and consider the functions $\Gamma_{i,j} \colon C_{i,j} \to \R^{2 \times 2}$ defined as
	$$
	\Gamma_{i,j}(x):=\frac{\varepsilon_j}{2\pi} \xi_{i,j} \otimes J \frac{x-x_{i,j}}{|x-x_{i,j}|^2}
	$$
	where $J$ is as before the anticlockwise rotation of $\pi/2$. Notice that
	$$
	\int_{C_{i,j}} |\Gamma_{i,j}|^2 \, dx \leq C \varepsilon_j^2 |\xi_{i,j}|^2.
	$$
	Using Theorem \ref{thm: rig est for incompat} applied to $\beta_j$ restricted to $C_{i,j}$ and recalling that $\curl \beta_j =0$ in $C_{i,j}$, we have that by scaling invariance there exists $C=C(K)>0$ and $F \in K$ such that
		$$
		\int_{C_{i,j}} |\beta_j - F|^2 \, dx \leq C \int_{C_{i,j}} \dist^2(\beta_j,K) \, dx + C \left( \int_{C_{i,j}} |\divr \beta_j| \, dx \right)^2.
		$$
		In turn, since $\beta_j \in \mathcal{AS}_{\varepsilon_j}(\mu_j)$, we estimate
		\begin{align*}
			\int_{C_{i,j}} |\beta_j - F|^2 \, dx & \geq \int_{\eta_{\varepsilon_j}}^{4\eta_{\varepsilon_j}} \frac{1}{2\pi r} \left|  \int_{\partial B_r} (\beta_j-F) \cdot t \, ds  \right|^2 \, dr=\int_{\eta_{\varepsilon_j}}^{4\eta_{\varepsilon_j}} \frac{1}{2\pi r} \left| \int_{\partial B_r} \beta_j \cdot t \, ds \right|^2 \, dr \\
			& =\int_{\eta_{\varepsilon_j}}^{4\eta_{\varepsilon_j}} \frac{\varepsilon^2}{2\pi r} |\xi_{i,j}|^2 \, dr = |\xi_{i,j}|^2\frac{ \log 4}{2\pi} \varepsilon^2.
		\end{align*}
	Hence, we infer that
	\begin{equation}\label{m2}
		\int_{C_{i,j}} |\Gamma_{i,j}|^2 \, dx \leq C \int_{C_{i,j}} \dist^2(\beta_j,K) \, dx+ C \left(  \int_{C_{i,j}} |\divr \beta_j| \, dx \right)^2.
	\end{equation}
	By construction we have $\curl(\beta_j -\Gamma_{i,j})=0$ in $C_{i,j}$, $\int_{\partial B_{\eta_{\varepsilon_j}}(x_{i,j})} (\beta_j-\Gamma_{i,j}) \cdot t \, ds=0$ and $\divr \Gamma_{i,j}=0$. Hence, we can find $u_{i,j} \in H^1(C_{i,j};\R^2)$ such that $\nabla u_{i,j}=\beta_j-\Gamma_{i,j}$ and $\Delta u_{i,j}=\divr \beta_j$ on $C_{i,j}$. Using Theorem \ref{thm: multiwell rigidity crit} applied to $\nabla u_{i,j}$ on $C_{i,j}$ and \eqref{m2}, we find a matrix $F_{i,j} \in K$ such that
	\begin{align*}
		\int_{C_{i,j}} |\nabla u_{i,j}-F_{i,j}|^2 \, dx & \leq C \int_{C_{i,j}} \dist^2(\nabla u_{i,j},K) \, dx + C \left(  \int_{C_{i,j}} |\divr \beta_j| \, dx \right)^2 \\
		& \leq C \int_{C_{i,j}} \dist^2(\beta_j,K) \, dx+ C \left(  \int_{C_{i,j}} |\divr \beta_j| \, dx \right)^2.
	\end{align*}
	Using the fact that $\Delta u_{i,j}=\divr \beta_j \in L^2(C_{i,j};\R^2)$ and arguing as in the proof of Lemma \ref{lem: rigidity non equi spaced annulus} (see \eqref{n1} and \eqref{n2}), we find $v_{i,j} \in H^2(B_{3\eta_{\varepsilon_j}}(x_{i,j});\R^2)$ such that $\nabla v_{i,j}=\nabla u_{i,j}-F_{i,j}$ on $B_{3 \eta_{\varepsilon_j}}(x_{i,j}) \setminus B_{2 \eta_{\varepsilon_j}}(x_{i,j})$,
	\begin{equation}\label{m3}
		\int_{B_{3 \eta_{\varepsilon_j}}(x_{i,j})} |\nabla v_{i,j}|^2 \, dx \leq C \int_{C_{i,j}} |\nabla u_{i,j}-F_{i,j}|^2 \, dx \leq C \int_{C_{i,j}} \dist^2(\beta_j,K) \, dx+ C \left(  \int_{C_{i,j}} |\divr \beta_j| \, dx \right)^2
	\end{equation}
	and
	\begin{align}\label{m4}
		\begin{split}
			\int_{B_{3 \eta_{\varepsilon_j}}(x_{i,j})} |\nabla^2 v_{i,j}|^2 \, dx & \leq C \int_{C_{i,j}} |\nabla^2 u_{i,j}|^2 \, dx+C \int_{C_{i,j}} |\nabla u_{i,j}-F_{i,j}|^2 \, dx \\
			& \leq  \frac{C}{\eta_{\varepsilon_j}^2} \int_{C_{i,j}} |\nabla u_{i,j}-F_{i,j}|^2 \, dx + C \int_{C_{i,j}} |\Delta u_{i,j}|^2 \, dx \\
			& \leq \frac{C}{\eta_{\varepsilon_j}^2} \int_{C_{i,j}} \dist^2(\beta_j,K) \, dx +  C \int_{C_{i,j}} |\divr \beta_j|^2 \, dx.
		\end{split}
	\end{align}
	Now we are ready to define the modification $\tilde{\beta}_j \colon B_{\rho_{\varepsilon_j}}(x_{i,j}) \to \R^{2 \times 2}$ as
	\begin{equation*}
		\tilde{\beta}_j:=
		\begin{cases*}
			\beta_j & in $B_{\rho_{\varepsilon_j}}(x_{i,j}) \setminus B_{2\eta_{\varepsilon_j}}(x_{i,j})$,\\
			\nabla v_{i,j}+F_{i,j} & in $B_{2\eta_{\varepsilon_j}}(x_{i,j})$.
		\end{cases*}
	\end{equation*}
	Notice that by construction, H\"older inequality and \eqref{m3} we have
	\begin{align}\label{m5}
		\begin{split}
			\int_{B_{\rho_{\varepsilon_j}}(x_{i,j})} \dist^2(\tilde{\beta}_j,K) \, dx & \leq \int_{B_{\rho_{\varepsilon_j}}(x_{i,j})} \dist^2(\beta_j,K) \, dx+\int_{B_{2 \eta_{\varepsilon_j}}(x_{i,j})} |\nabla v_{i,j}|^2 \, dx \\
			& \leq C \int_{B_{\rho_{\varepsilon_j}}(x_{i,j})} \dist^2(\beta_j,K) \, dx + C \eta_{\varepsilon_j}^2 \int_{C_{i,j}} |\divr \beta_j|^2 \, dx.
		\end{split}
	\end{align}
	For the divergence of $\tilde{\beta}_j$ on $B_{\rho_{\varepsilon_j}}(x_{i,j})$, let us remark that since we obtained $v_{i,j}$ using an extension from $H^2(B_{3 \eta_{\varepsilon_j}}(x_{i,j}) \setminus B_{2 \eta_{\varepsilon_j}}(x_{i,j}))$ to $H^2(B_{3 \eta_{\varepsilon_j}}(x_{i,j}))$, we have that the trace of $\beta_j-\Gamma_{i,j}$ coincides with the trace of $\nabla v_{i,j}+F_{i,j}$ as elements of $H^{1/2}(\partial B_{2\eta_{\varepsilon_j}}(x_{i,j});\R^{2 \times 2})$. Hence, for every $\varphi \in C^1_c (B_{\rho_{\varepsilon_j}}(x_{i,j});\R^2)$, we have
	$$
	\int_{\partial B_{2\eta_{\varepsilon_j}}(x_{i,j})} \varphi \cdot [(\nabla v_{i,j}+F_{i,j}) \nu] \, d\mathcal{H}^1 =\int_{\partial B_{2\eta_{\varepsilon_j}}(x_{i,j})} \varphi \cdot [(\beta_j-\Gamma_{i,j}) \nu] \, d\mathcal{H}^1 =\int_{\partial B_{2\eta_{\varepsilon_j}}(x_{i,j})} \varphi \cdot (\beta_j  \nu) \, d\mathcal{H}^1,
	$$
	where $\nu$ is the normal vector to $\partial B_{2\eta_{\varepsilon_j}}(x_{i,j})$.
	This implies that $\divr \tilde{\beta}_j$ is absolutely continuous on $B_{\rho_{\varepsilon_j}}(x_{i,j})$ with respect to the Lebesgue measure. Therefore, using \eqref{m4} we estimate
	\begin{align}\label{m6}
		\begin{split}
			\int_{B_{\rho_{\varepsilon_j}}(x_{i,j})} |\divr \tilde{\beta}_j|^2 & \leq \int_{B_{\rho_{\varepsilon_j}}(x_{i,j})} |\divr \beta_j |^2 \, dx +\sum_{i=1}^{M_j} \int_{B_{2 \eta_{\varepsilon_j}}(x_{i,j})} |\nabla^2 v_{i,j}|^2 \, dx \\
			& \leq \frac{C}{\eta_{\varepsilon_j}^2} \int_{B_{\rho_{\varepsilon_j}}(x_{i,j})} \dist^2(\beta_j,K) \, dx +  C \int_{B_{\rho_{\varepsilon_j}}(x_{i,j})} |\divr \beta_j|^2 \, dx.
		\end{split}
	\end{align}
	Finally, by construction we have $|\curl \tilde{\beta}_j|(B_{\rho_{\varepsilon_j}}(x_{i,j}))=\varepsilon_j |\xi_{i,j}|$. 
	
	\noindent \textbf{Step 1.2.} \textit{Rigidity of $\tilde{\beta}_j$ on $B_{\rho_{\varepsilon_j}}(x_{i,j})$.} To shorten the notation from now on we assume that $x_{i,j}=0$ and we drop the dependence on $i$. Moreover, we assume also that $\rho_{\varepsilon_j} \leq 1$.
	Applying Theorem \ref{thm: rig est for incompat} to $\tilde{\beta}_j$ on $B_{\rho_{\varepsilon_j}}$, we find $F_j \in K$ such that
	\begin{equation}\label{m6.5}
		\Vert \tilde{\beta}_j-F_j \Vert_{L^2(B_{\rho_{\varepsilon_j}})} \leq C\left( \Vert \dist(\tilde{\beta}_j,K) \Vert_{L^2(B_{\rho_{\varepsilon_j}})} +\Vert \divr \tilde{\beta}_j \Vert_{L^1(B_{\rho_{\varepsilon_j}})} +\varepsilon_j |\xi_{i,j}| \right) \leq C \frac{1}{\eta_{\varepsilon_j}} \varepsilon_j |\log \varepsilon_j|.
	\end{equation}
	
	Define now the function $z_j \in H^1_0(\Omega;\R^2)$ as the unique solution to the Dirichlet problem
	\begin{equation*}
		\begin{cases*}
			\Delta z_j = -\curl \tilde{\beta}_j & on $B_{\rho_{\varepsilon_j}}$, \\
			z_j=0 & on $\partial B_{\rho_{\varepsilon_j}}$.
		\end{cases*}
	\end{equation*}
	By elliptic regularity with measure data we infer $\Vert \nabla z_j \Vert_{L^{2,\infty}(B_{\rho_{\varepsilon_j}})} \leq C \varepsilon_j |\xi_{i,j}|$. Define $w_j \in H^1(B_{\rho_{\varepsilon_j}};\R^2)$ as $\nabla w_j:=\tilde{\beta}_j-\nabla z_j J$. We have that $\divr \nabla w_j=\divr \tilde{\beta}_j$ on $B_{\rho_{\varepsilon_j}}$ and, by \eqref{m5}, 
	\begin{equation}\label{m7}
		\Vert \dist(\nabla w_j,K) \Vert_{L^{2,\infty}(B_{\rho_{\varepsilon_j}})}^2 \leq C\int_{B_{\rho_{\varepsilon_j}}} \dist^2(\beta_j,K) \, dx + C \eta_{\varepsilon_j}^2 \int_{C_{i,j}} |\divr \beta_j|^2 \, dx+C\varepsilon_j^2|\xi_{i,j}|^2.
	\end{equation}
	Using Lemma \ref{lem: elliptic regularity}, \eqref{m6}, \eqref{m6.5} and recalling \eqref{eta range}, we estimate
	\begin{align}\label{m8}
		\begin{split}
			\Vert \nabla^2 w_j \Vert_{L^{2,\infty}(B_{\rho_{\varepsilon_j/2}})} \, dx & \leq \frac{C}{\rho_{\varepsilon_j}} \Vert \nabla w_j-F_j \Vert_{L^{2,\infty}(B_{\rho_{\varepsilon_j}})}+\frac{C}{\rho_{\varepsilon_j}}\Vert \divr \tilde{\beta}_j \Vert_{L^2(B_{\rho_{\varepsilon_j}})} \\
			& \leq \frac{C}{\rho_{\varepsilon_j}} \Vert \tilde{\beta}_j-F_j \Vert_{L^2(B_{\rho_{\varepsilon_j}})}+\frac{C}{\rho_{\varepsilon_j}}\varepsilon_j |\xi_{i,j}|+\frac{C}{\rho_{\varepsilon_j}}\Vert \divr \tilde{\beta}_j \Vert_{L^2(B_{\rho_{\varepsilon_j}})} \\
			& \leq C \frac{1}{\rho_{\varepsilon_j} \eta_{\varepsilon_j}} \varepsilon_j|\log \varepsilon_j| \leq C.
		\end{split}
	\end{align}
	Combining \eqref{m7} and \eqref{m8} with Proposition \ref{prop: weak rigidity multipl form} and  Theorem \ref{thm: CDM rigidity}, we find a rotation $S_j \in SO(2)$ and $\tilde{U} \in \{ U_1,\dots,U_\ell \}$ such that the following bound holds
	\begin{align}\label{m9}
		\begin{split} 
			\Vert \nabla w_j - S_j \tilde{U} \Vert_{L^{2,\infty}(B_{\rho_{\varepsilon_j/2}})}^2 
			\leq C \int_{B_{\rho_{\varepsilon_j}}} \dist^2(\beta_j,K) \, dx + C \eta_{\varepsilon_j}^2 \int_{C_{i,j}} |\divr \beta_j|^2 \, dx+C\varepsilon_j^2|\xi_{i,j}|^2.
		\end{split}
	\end{align}
	
	Arguing as in the proof of Theorem \ref{thm: rig est for incompat}, in particular from \eqref{40} to \eqref{47}, we infer that
	\begin{align*}
		\Vert J(\curl (\cof(\tilde{U})^{-1}\tilde{\beta}_j)) \Vert_{H^{-1}(B_{\rho_{\varepsilon_j/2}})}^2 & \leq C \left( \Vert \dist(\tilde{\beta}_j,K) \Vert_{L^2(B_{\rho_{\varepsilon_j}})}^2 +\Vert \tilde{\beta}_j - S_j \tilde{U} \Vert_{L^{2,\infty}(B_{\rho_{\varepsilon_j/2}})}^2+\varepsilon_j^2|\xi_{i,j}|^2 \right) \\
		& \leq C \int_{B_{\rho_{\varepsilon_j}}} \dist^2(\beta_j,K) \, dx + C \eta_{\varepsilon_j}^2 \int_{C_{i,j}} |\divr \beta_j|^2 \, dx+C\varepsilon_j^2|\xi_{i,j}|^2.
	\end{align*}
	Hence, if we define another Dirichlet problem in $B_{\rho_{\varepsilon_j/2}}$ as
	\begin{equation*}
		\begin{cases*}
			\Delta \tilde{z}_j = -\curl \left( \cof(\tilde{U})^{-1} \tilde{\beta}_j \right) & on $B_{\rho_{\varepsilon_j/2}}$, \\
			\tilde{z}_j=0 & on $\partial B_{\rho_{\varepsilon_j/2}}$;
		\end{cases*}
	\end{equation*}
	we have that there exists a unique solution $\tilde{z}_j \in H^1_0(B_{\rho_{\varepsilon_j/2}};\R^2)$ such that
	\begin{equation}\label{m10}
		\int_{B_{\rho_{\varepsilon_j/2}}} |\nabla \tilde{z}_j|^2 \, dx \leq C \int_{B_{\rho_{\varepsilon_j}}} \dist^2(\beta_j,K) \, dx + C \eta_{\varepsilon_j}^2 \int_{C_{i,j}} |\divr \beta_j|^2 \, dx+C\varepsilon_j^2|\xi_{i,j}|^2.
	\end{equation}
	Define $\tilde{w}_j \in H^1(B_{\rho_{\varepsilon_j/2}};\R^2)$ as $\nabla \tilde{w}_j := \tilde{\beta}_j-\cof(\tilde{U}) \nabla\tilde{z}_j J$. Therefore, similarly as before, using Lemma \ref{lem: elliptic regularity}, \eqref{m6}, \eqref{m6.5} and \eqref{m10} we estimate
	\begin{align}\label{m11}
		\begin{split}
			\Vert \nabla^2 \tilde{w}_j \Vert_{L^2(B_{\rho_{\varepsilon_j/4}})} & \leq \frac{C}{\rho_{\varepsilon_j}} \Vert \tilde{\beta}_j-F_j \Vert_{L^2(B_{\rho_{\varepsilon_j}})}+\frac{C}{\rho_{\varepsilon_j}} \Vert \nabla  \tilde{z}_j \Vert_{L^2(B_{\rho_{\varepsilon_j/2}})}+\frac{C}{\rho_{\varepsilon_j}}\Vert \divr \tilde{\beta}_j \Vert_{L^2(B_{\rho_{\varepsilon_j}})} \\
			& \leq C \frac{1}{\rho_{\varepsilon_j} \eta_{\varepsilon_j}} \varepsilon_j|\log \varepsilon_j| \leq C.
		\end{split}
	\end{align}
	Finally, using Proposition \ref{prop: rigidity multiplic form}, \eqref{m10}, \eqref{m11} and Theorem \ref{thm: FJM rigidity}, we infer the existence of $R_j \in SO(2)$ and $U \in \{ U_1,\dots,U_\ell \}$ such that the following bound holds
	\begin{align*}
		\Vert \nabla \tilde{w}_j-R_j U \Vert_{L^2(B_{\rho_{\varepsilon_j/4}})}^2  & \leq C \int_{B_{\rho_{\varepsilon_j/4}}} \dist^2(\nabla \tilde{w}_j,K) \, dx \\
		& \leq C \int_{B_{\rho_{\varepsilon_j}}} \dist^2(\beta_j,K) \, dx + C \eta_{\varepsilon_j}^2 \int_{C_{i,j}} |\divr \beta_j|^2 \, dx+C\varepsilon_j^2|\xi_{i,j}|^2.
	\end{align*}
	This in turn gives,
	$$
	\Vert \tilde{\beta}_j-R_j U\Vert_{L^2(B_{\rho_{\varepsilon_j/4}})}^2 \leq C\int_{B_{\rho_{\varepsilon_j}}} \dist^2(\beta_j,K) \, dx + C \eta_{\varepsilon_j}^2 \int_{C_{i,j}} |\divr \beta_j|^2 \, dx+C\varepsilon_j^2|\xi_{i,j}|^2.
	$$
	By definition of $\tilde{\beta}_j$ we have
	\begin{equation*}
		\int_{B_{\rho_{\varepsilon_j/4}} \setminus B_{2 \eta_{\varepsilon_j}}} |\beta_j - R_j U|^2 \, dx \leq C\int_{B_{\rho_{\varepsilon_j}}} \dist^2(\beta_j,K) \, dx + C \eta_{\varepsilon_j}^2 \int_{C_{i,j}} |\divr \beta_j|^2 \, dx+C\varepsilon_j^2|\xi_{i,j}|^2.
	\end{equation*}
	
	Hence, for every $j$ large enough and $i=1,\dots,M_j$ we have proven that there exist $A_{i,j} \in K$, $C>0$ not dependent on $j$ and $i$ and $0<s_1<s_2<\gamma<1$ (where $\gamma$ is as in \eqref{eta range}) such that
	\begin{equation}\label{m12}
		\int_{B_{\varepsilon_j^{s_1}}(x_{i,j}) \setminus B_{\varepsilon_j^{s_2}}(x_{i,j})} |\beta_j-A_{i,j}|^2 \, dx \leq C\int_{B_{\rho_{\varepsilon_j}}(x_{i,j})} \dist^2(\beta_j,K) \, dx + C \eta_{\varepsilon_j}^2 \int_{C_{i,j}} |\divr \beta_j|^2 \, dx+C\varepsilon_j^2|\xi_{i,j}|^2.
	\end{equation}
	In turn, since $\beta_j \in \mathcal{AS}_{\varepsilon_j}(\mu_j)$, for every $j \geq 1 $ and $i=1,\dots,M_j$ we estimate
	\begin{align}\label{m13}
		\begin{split}
			\int_{B_{\varepsilon_j^{s_1}}(x_{i,j}) \setminus B_{\varepsilon_j^{s_2}}(x_{i,j})} |\beta_j - A_{i,j}|^2 \, dx & \geq \int_{\varepsilon_j^{s_2}}^{\varepsilon_j^{s_1}} \frac{1}{2\pi r} \left|  \int_{\partial B_r} (\beta_j-A_{i,j}) \cdot t \, d\mathcal{H}^1  \right|^2 \, dr=\int_{\varepsilon_j^{s_2}}^{\varepsilon_j^{s_1}} \frac{1}{2\pi r} \left|  \int_{\partial B_r} \beta_j \cdot t \, d\mathcal{H}^1  \right|^2 \, dr \\
			& =\int_{\varepsilon_j^{s_2}}^{\varepsilon_j^{s_1}} \frac{\varepsilon_j^2}{2\pi r} |\xi_{i,j}|^2 \, dr = \frac{s_2-s_1}{2\pi} \varepsilon_j^2|\log \varepsilon_j||\xi_{i,j}|^2.
		\end{split}
	\end{align}
	Hence, combining \eqref{51}, \eqref{m12}, \eqref{m13} we obtain
	\begin{align*}
		C \varepsilon_j^2 |\log \varepsilon_j|^2+C\varepsilon_j^2 \sum_{i=1}^{M_j} |\xi_{i,j}|^2 & \geq \sum_{i=1}^{M_j} \int_{B_{\rho_{\varepsilon_j}}(x_{i,j}) \setminus B_{\varepsilon_j}(x_{i,j})}\left( \dist^2(\beta_j,K)+\eta_{\varepsilon_j}^2|\divr \beta_j|^2 \right) \, dx+C\varepsilon_j^2 \sum_{i=1}^{M_j} |\xi_{i,j}|^2 \\
		& \geq \sum_{i=1}^{M_j} \frac{s_2-s_1 }{2\pi} \varepsilon_j^2|\log \varepsilon_j||\xi_{i,j}|^2.
	\end{align*}
	This in turn gives
	$$
	\left( \frac{s_2-s_1 }{2\pi} |\log \varepsilon_j|-C \right) \sum_{i=1}^{M_j}  |\xi_{i,j}|^2 \leq C |\log \varepsilon_j|^2.
	$$
	Finally, using the fact that non null elements of $\mathbb{S}$ have norm bounded away from zero, we conclude that for $j$ large enough $M_j \leq C|\log \varepsilon_j|$ and thus \eqref{m1}.

	\noindent \textbf{Step 2.} \textit{Weak convergence of the scaled strains.} 
	We further subdivide the proof into three sub-steps.
	
	\noindent \textbf{Step 2.1.} \textit{Modification of $\beta_j$.} We need to modify $\beta_j$ on $\Omega$ in order to obtain $\tilde{\beta}_j \in L^2(\Omega;\R^{2 \times 2})$ such that $|\curl \tilde{\beta}_j|(\Omega)=|\mu_j|(\Omega)$. 
    Take $x_{i,j}$ in the support of $\mu_j$. For every $i=1,\dots,M_j$, we set as before $C_{i,j}:=B_{4\eta_{\varepsilon_j}}(x_{i,j}) \setminus B_{\eta_{\varepsilon_j}}(x_{i,j})$. Arguing exactly as in Step 1.1. and keeping the exact notations, we define 
	the modification $\tilde{\beta}_j \colon \Omega \to \R^{2 \times 2}$ of $\beta_j$ on $\Omega$ as
	\begin{equation*}
		\tilde{\beta}_j:=
		\begin{cases*}
			\beta_j & in $\Omega_{2 \eta_{\varepsilon_j}}(\mu_j)$,\\
			\nabla v_{i,j}+F_{i,j} & in $B_{2\eta_{\varepsilon_j}}(x_{i,j})$ for every $i=1,\dots,M_j$;
		\end{cases*}
	\end{equation*}
	such that
	\begin{align}\label{55}
		\begin{split}
			\int_\Omega \dist^2(\tilde{\beta}_j,K) \, dx & \leq \int_\Omega \dist^2(\beta_j,K) \, dx+\sum_{i=1}^{M_j} \int_{B_{2 \eta_{\varepsilon_j}}(x_{i,j})} |\nabla v_{i,j}|^2 \, dx \\
			& \leq C \int_\Omega \dist^2(\beta_j,K) \, dx + C \sum_{i=1}^{M_j} \eta_{\varepsilon_j}^2 \int_{C_{i,j}} |\divr \beta_j|^2 \, dx \\
			& \leq \vphantom{\int} C \varepsilon_j^2|\log \varepsilon_j|^2,
		\end{split}
	\end{align}
	and
	\begin{align}\label{56}
		\begin{split}
			\int_\Omega |\divr \tilde{\beta}_j|^2 & \leq \int_\Omega |\divr \beta_j |^2 \, dx +\sum_{i=1}^{M_j} \int_{B_{2 \eta_{\varepsilon_j}}(x_{i,j})} |\nabla^2 v_{i,j}|^2 \, dx \\
			& \leq \frac{C}{\eta_{\varepsilon_j}^2} \int_{\Omega} \dist^2(\beta_j,K) \, dx +  C \int_{\Omega} |\divr \beta_j|^2 \, dx \\
			& \leq C \frac{1}{\eta_{\varepsilon_j}^2}\varepsilon_j^2 |\log \varepsilon_j|^2.
		\end{split}
	\end{align}
	Finally, by construction we have $|\curl \tilde{\beta}_j|(\Omega)=|\mu_j|(\Omega)$. 
	
	\noindent \textbf{Step 2.2.} \textit{Rigidity of $\tilde{\beta}_j$ on $\Omega$.} We essentially repeat the construction in Step 1.2 but on the whole set $\Omega$, keeping track of the additional dependence on $\rho_\varepsilon$. As we are going to use Lemma \ref{lem: shrinking lip}, in this step we explicit the dependence of the constants appearing. We now apply Theorem \ref{thm: rig est for incompat} to $\tilde{\beta}_j$ on $\Omega$ and find $V_j \in K$ such that
	\begin{equation}\label{56.5}
		\Vert \tilde{\beta}_j-V_j \Vert_{L^2(\Omega)} \leq C(\Omega) \left( \Vert \dist(\tilde{\beta}_j,K) \Vert_{L^2(\Omega)} +\Vert \divr \tilde{\beta}_j \Vert_{L^1(\Omega)} +|\mu_j|(\Omega) \right) \leq C(\Omega) \frac{1}{\eta_{\varepsilon_j}}\varepsilon_j|\log \varepsilon_j|.
	\end{equation}
	
	Define the function $z_j \in H^1_0(\Omega;\R^2)$ as the unique solution to the Dirichlet problem
	\begin{equation*}
		\begin{cases*}
			\Delta z_j = -\curl \tilde{\beta}_j & on $\Omega$, \\
			z_j=0 & on $\partial \Omega$.
		\end{cases*}
	\end{equation*}
	By elliptic regularity with measure data, we have $\Vert \nabla z_j \Vert_{L^{2,\infty}(\Omega)} \leq C(\Omega) |\mu_j|(\Omega)$. Let $w_j \in H^1(\Omega;\R^2)$ be defined as $\nabla w_j:=\tilde{\beta}_j-\nabla z_j J$. Observe that $\divr \nabla w_j=\divr \tilde{\beta}_j$ on $\Omega$ and, by \eqref{55}, 
	\begin{equation}\label{57}
		\Vert \dist(\nabla w_j,K) \Vert_{L^{2,\infty}(\Omega)} \leq C \varepsilon_j |\log \varepsilon_j|.
	\end{equation}
	Fix $0<\delta \ll 1$. We can find $\Omega^\delta \Subset \Omega$ satisfying (i)--(iii) of Lemma \ref{lem: shrinking lip}. Hence, using Lemma \ref{lem: elliptic regularity}, \eqref{56} and \eqref{56.5} we estimate
	\begin{align}\label{58}
		\begin{split}
		 \Vert \nabla^2 w_j \Vert_{L^{2,\infty}(\Omega^\delta)} \, dx & \leq \frac{C(\Omega)}{\delta} \Vert \nabla w_j-V_j \Vert_{L^{2,\infty}(\Omega)}+\frac{C(\Omega)}{\delta}\Vert \divr \tilde{\beta}_j \Vert_{L^2(\Omega)} \\
		 & \leq \frac{C(\Omega)}{\delta} \Vert \tilde{\beta}_j-V_j \Vert_{L^2(\Omega)}+\frac{C(\Omega)}{\delta}|\mu_j|(\Omega)+\frac{C(\Omega)}{\delta}\Vert \divr \tilde{\beta}_j \Vert_{L^2(\Omega)} \\
		 & \leq C(\Omega) \frac{1}{\delta \eta_{\varepsilon_j}} \varepsilon_j|\log \varepsilon_j|.
		\end{split}
	\end{align}
	Combining \eqref{57} and \eqref{58} with Proposition \ref{prop: weak rigidity multipl form} and Theorem \ref{thm: CDM rigidity}, we find a rotation $S_j \in SO(2)$ and a matrix $\tilde{U} \in \{ U_1,\dots,U_\ell \}$ such that the following estimate holds
	\begin{align}\label{59}
		\begin{split} 
		\Vert \nabla w_j - S_j \tilde{U} \Vert_{L^{2,\infty}(\Omega^\delta)}^2 
		 \leq C(\Omega) \varepsilon_j^2|\log \varepsilon_j|^2+ C(\Omega)\frac{1}{\delta^2 \eta_{\varepsilon_j}^2} \varepsilon_j^4 |\log \varepsilon_j|^4.
		\end{split}
	\end{align}
	Indeed, in view of property (ii) of $\Omega^\delta$, and Proposition \ref{App: CDM} we have that the constant $C$ in \eqref{59} depends only on $\Omega$ and not on $\delta$.
	 
	Arguing as in the proof of Theorem \ref{thm: rig est for incompat}, in particular from \eqref{40} to \eqref{47}, and using Lemma \ref{App: lem for BBV}, we infer that
	\begin{align*}
	\Vert J(\curl (\cof(\tilde{U})^{-1}\tilde{\beta}_j)) \Vert_{H^{-1}(\Omega^\delta)}^2 & \leq C(\Omega) \left( \Vert \dist(\tilde{\beta}_j,K) \Vert_{L^{2}(\Omega)}^2+\Vert \tilde{\beta}_j- S_j \tilde{U} \Vert_{L^{2,\infty}(\Omega^\delta)}^2+|\mu_j|(\Omega)^2 \right) \\
	& \leq C(\Omega) \varepsilon_j^2|\log \varepsilon_j|^2+ C(\Omega)\frac{1}{\delta^2 \eta_{\varepsilon_j}^2} \varepsilon_j^4 |\log \varepsilon_j|^4.
	\end{align*}
	Hence, if we define another Dirichlet problem in $\Omega^\delta$ as
	\begin{equation*}
		\begin{cases*}
			\Delta \tilde{z}_j = -\curl \left( \cof(\tilde{U})^{-1}\tilde{\beta}_j \right) & on $\Omega^\delta$, \\
			\tilde{z}_j=0 & on $\partial \Omega^\delta$,
		\end{cases*}
	\end{equation*}
	there exists a unique solution $\tilde{z}_j \in H^1_0(\Omega^\delta,\R^2)$ such that
	\begin{equation}\label{60}
		\int_{\Omega^\delta} |\nabla \tilde{z}_j|^2 \, dx \leq \Vert \curl (\cof(\tilde{U})^{-1}\tilde{\beta}_j) \Vert_{H^{-1}(\Omega^\delta)}^2 \leq  C(\Omega) \varepsilon_j^2|\log \varepsilon_j|^2+ C(\Omega) \frac{1}{\delta^2 \eta_{\varepsilon_j}^2} \varepsilon_j^4 |\log \varepsilon_j|^4.
	\end{equation}
	Since $\Omega^\delta$ is simply connected by (i), we can define $\tilde{w}_j \in H^1(\Omega^\delta;\R^2)$ as $\nabla \tilde{w}_j := \tilde{\beta}_j-\cof(\tilde{U}) \nabla \tilde{z}_j J$. By Lemma \ref{lem: shrinking lip} applied to $\Omega^\delta$ we can find another set $\Lambda^\delta:=(\Omega^\delta)^\delta \Subset \Omega^\delta \Subset \Omega$ such that (i)--(iii) hold. Similarly as before, using Lemma \ref{lem: elliptic regularity}, \eqref{56}, \eqref{56.5} and \eqref{60} we estimate
	\begin{align}\label{61}
		\begin{split}
			\Vert \nabla^2 \tilde{w}_j \Vert_{L^2(\Lambda^\delta)} & \leq \frac{C(\Omega)}{\delta} \Vert \tilde{\beta}_j-V_j \Vert_{L^2(\Omega)}+\frac{C(\Omega)}{\delta} \Vert \nabla  \tilde{z}_j \Vert_{L^2(\Lambda^\delta)}+\frac{C(\Omega)}{\delta}\Vert \divr \tilde{\beta}_j \Vert_{L^2(\Omega)} \\
			& \leq C(\Omega) \frac{1}{\delta \eta_{\varepsilon_j}} \varepsilon_j|\log \varepsilon_j|+C(\Omega)\frac{1}{\delta^2 \eta_{\varepsilon_j}}\varepsilon_j^2 |\log \varepsilon_j|^2.
		\end{split}
	\end{align}
	Finally, using Proposition \ref{prop: rigidity multiplic form}, \eqref{60}, \eqref{61} Theorem \ref{thm: FJM rigidity} and Lemma \ref{App: FJM}, we infer the existence of $R_j \in SO(2)$ and $U \in \{ U_1,\dots,U_\ell \}$ such that
	\begin{align*}
		\Vert \nabla \tilde{w}_j-R_j U \Vert_{L^2(\Lambda^\delta)}^2  & \leq C(\Omega) \int_{\Lambda^\delta} \! \dist^2(\nabla \tilde{w}_j,K) \, dx\left(1+\int_{\Lambda^\delta} |\nabla^2 \tilde{w}_j|^2 \, dx \right) \\
		& \leq C(\Omega) \left( \varepsilon_j^2|\log \varepsilon_j|^2+\frac{1}{\delta^2\eta_{\varepsilon_j}^2}\varepsilon_j^4 |\log \varepsilon_j|^4+ \frac{1}{\delta^4\eta_{\varepsilon_j}^4}\varepsilon_j^6 |\log \varepsilon_j|^6+\frac{1}{\delta^6\eta_{\varepsilon_j}^4}\varepsilon_j^8 |\log \varepsilon_j|^8 \right).
	\end{align*}
	This in turn gives, taking $\delta:=\rho_{\varepsilon_j}/8$, as $\Omega_j \subset \Lambda^\delta$,
	$$
	\Vert \tilde{\beta}_j-R_j U\Vert_{L^2(\Omega_j)}^2 \leq C(\Omega) \left( \varepsilon_j^2|\log \varepsilon_j|^2+\frac{1}{\rho_{\varepsilon_j}^2\eta_{\varepsilon_j}^2}\varepsilon_j^4 |\log \varepsilon_j|^4+ \frac{1}{\rho_{\varepsilon_j}^4\eta_{\varepsilon_j}^4}\varepsilon_j^6 |\log \varepsilon_j|^6+\frac{1}{\rho_{\varepsilon_j}^6\eta_{\varepsilon_j}^4}\varepsilon_j^8 |\log \varepsilon_j|^8 \right).
	$$
	Recalling the bounds on $\eta_{\varepsilon_j}$ in \eqref{eta range}, we deduce that $\Vert \tilde{\beta}_j-R_j U\Vert_{L^2(\Omega_j)} \leq C(\Omega) \varepsilon_j |\log \varepsilon_j|$. Finally, by definition of $\tilde{\beta}_j$ we have
	\begin{equation}\label{62}
		\int_{\Omega_j \cap \Omega_{2 \eta_{\varepsilon_j}}(\mu_j)} |\beta_j - R_j U|^2 \, dx \leq C\varepsilon_j^2 |\log \varepsilon_j|^2,
	\end{equation}
	where the constant $C>0$ does not depends on $j$.
	
	\noindent \textbf{Step 2.3.} \textit{Rigidity of $\beta_j$ near the dislocations.} Let now $i \in \{ 1, \dots M_j \}$ and define \footnote{Notice that here we do not have any information on $\curl \beta$. For this reason, we have to introduce a cut in the annulus to exploit Lemma \ref{lem: rigidity non equi spaced annulus}.} 
	$$
	L_{i,j}:=\{ (z,(x_{i,j})_2) \in \R^2 \colon (x_{i,j})_1+\varepsilon_j < z < (x_{i,j})_1+3\eta_{\varepsilon_j} \}
	$$
	Observe that $\beta_j$ restricted to $(B_{3 \eta_{\varepsilon_j}}(x_{i,j}) \setminus B_{\varepsilon_j}(x_{i,j})) \setminus L_{i,j}$, since it has zero curl, coincides with the gradient of a function in $H^1((B_{3 \eta_{\varepsilon_j}}(x_{i,j}) \setminus B_{\varepsilon_j})(x_{i,j}) \setminus L_{i,j};\R^2)$. Hence, we can use Lemma \ref{lem: rigidity non equi spaced annulus} on the annulus with the cut to find a constant $C>0$ not dependent on $i$ and $j$ and a matrix $\tilde{F}_{i,j} \in K$ such that
	\begin{equation}\label{63}
		\int_{B_{3\eta_{\varepsilon_j}(x_{i,j})} \setminus B_{\varepsilon_j}(x_{i,j})} |\beta_j-\tilde{F}_{i,j}|^2 \, dx \leq C \int_{B_{3\eta_{\varepsilon_j}(x_{i,j})} \setminus B_{\varepsilon_j}(x_{i,j})} \left( \dist^2(\beta_j,K) + \eta_{\varepsilon_j}^2  |\divr \beta_j|^2 \right) \, dx.
	\end{equation}
	Combining \eqref{62}, \eqref{63}, using  triangular inequality and recalling that $B_{3\eta_{\varepsilon_j}} \Subset \Omega_j$, we get
	\begin{equation}\label{64}
		\sum_{i=1}^{M_j} \eta_{\varepsilon_j}^2|R_j U - \tilde{F}_{i,j}|^2 \leq C\varepsilon_j^2|\log \varepsilon_j|^2+ C \sum_{i=1}^{M_j} \int_{B_{3\eta_{\varepsilon_j}(x_{i,j})} \setminus B_{\varepsilon_j}(x_{i,j})} \left( \dist^2(\beta_j,K) + \eta_{\varepsilon_j}^2  |\divr \beta_j|^2 \right) \, dx.
	\end{equation}
	Using triangular inequality and \eqref{62}--\eqref{64}, we estimate
	\begin{align*}
		\int_{\Omega_j \cap \Omega_{\varepsilon_j}(\mu_j)} \! |\beta_j-R_j U|^2 \, dx & \leq \int_{\Omega_j \cap \Omega_{2 \eta_{\varepsilon_j}}(\mu_j)} |\beta_j - R_j U|^2 \\
		& \ \ \ \ \ \ \ \ \ \ + \sum_{i=1}^{M_j} \left( \int_{B_{\eta_{\varepsilon_j}}(x_{i,j}) \setminus B_{\varepsilon_j}(x_{i,j})} \! |\beta_j-\tilde{F}_{i,j}|^2 \, dx + \eta_{\varepsilon_j}^2|R_j U-\tilde{F}_{i,j}|^2 \right) \\
		& \leq C \varepsilon_j^2|\log \varepsilon_j|^2.
	\end{align*}
	Finally, recalling that $\beta_j \equiv I$ on $\bigcup_{i=1}^{M_j} B_{\varepsilon_j}(x_{i,j})$ and that $M_j \leq C|\log \varepsilon_j|$ from Step 1, we conclude that, up to subsequences,
	\begin{equation*}
		\chi_{\Omega_j}\frac{R^T_j \beta_j - U}{\varepsilon_j |\log \varepsilon_j|} \rightharpoonup \beta \qquad \mbox{ in $L^2(\Omega;\R^{2 \times 2})$},
	\end{equation*}
	where $\beta \in L^2(\Omega,\R^{2 \times 2})$.
	
	\noindent \textbf{Step 3.} \textit{The limit measure $\mu$ belongs to $H^{-1}(\Omega;\R^2)$ and $\curl \beta=R^T \mu$.} Let $\phi \in C^1_c(\Omega)$ and let $(\phi_j) \subset H^1_0(\Omega)$ be a sequence converging to $\phi$ uniformly on $\Omega$ and strongly on $H_0^1(\Omega)$, such that $\phi_j(y)=\phi_j(x_{i,j})$ for every $y \in B_{\varepsilon_j}(x_{i,j})$ and every $i=1,\dots,M_j$ and $\supp \phi_j \subseteq \Omega_j$ for every $j$ large enough. We have
	\begin{align*}
		\int_\Omega \phi \, d\mu & = \lim_{j \to +\infty} \frac{1}{\varepsilon_j|\log \varepsilon_j|}\int_\Omega \phi_j \, d\mu_j = \lim_{j \to +\infty} \frac{1}{\varepsilon_j|\log \varepsilon_j|} \langle \curl \beta_j,\phi_j \rangle \\
		& = \lim_{j \to +\infty} \frac{1}{\varepsilon_j|\log \varepsilon_j|} \langle \curl (\beta_j-R_jU),\phi_j \rangle
		 = \lim_{j \to +\infty} \frac{1}{\varepsilon_j|\log \varepsilon_j|} \int_{\Omega_j} (\beta_j-R_jU) J \nabla \phi_j \, dx \\
		 & = \int_\Omega R\beta J \nabla \phi \, dx=\langle \curl(R\beta),\phi \rangle =\langle R \ \curl(\beta),\phi \rangle.
	\end{align*}
	Hence, $\curl \beta = R^T\mu$. Finally, since from Step 2 we have that $\beta \in L^2(\Omega;\R^{2 \times 2})$, we immediately deduce that $\mu \in H^{-1}(\Omega;\R^2)$.

	\end{proof}

	In view of Proposition \ref{prop: compactness} it is convenient to give the following notion of convergence for sequences of pairs $(\mu_\varepsilon,\beta_\varepsilon)$.
	
	\begin{defn}\label{def: convergence}
		A pair of sequences $(\mu_\varepsilon,\beta_\varepsilon) \subset \mathcal{M}(\Omega;\R^2) \times L^2(\Omega;\R^{2 \times 2})$ is
		said to converge to a triplet $(\mu,\beta,RU) \in \mathcal{M}(\Omega;\R^2) \times L^2(\Omega;\R^{2 \times 2}) \times K$ if there
		exists a sequence of rotations $(R_\varepsilon) \subset SO(2)$ such that $R_\varepsilon \to R$ and
		\begin{align}
			& \frac{1}{\varepsilon|\log \varepsilon|} \mu_\varepsilon \xrightharpoonup{*} \mu \qquad \mbox{in $\mathcal{M}(\Omega;\R^2)$}, \label{conv of meas} \\
			& \chi_{\{ \dist(x,\partial \Omega)>\rho_\varepsilon/2 \}}\frac{R_\varepsilon^T \beta_\varepsilon - U}{\varepsilon|\log \varepsilon|} \rightharpoonup \beta \qquad \mbox{in $L^2(\Omega;\R^{2 \times 2})$}.
		\end{align}
	\end{defn}
	
	\subsection{The Self-Energy}
	
	Before proving the $\Gamma$-limit we need to introduce some additional results (see also \cite[Section 6]{GaLeoPonDislocations}).
	For $0<r_1<r_2 \leq 1$ and $\xi \in \R^2$ we introduce the class
	\begin{align*}
	\mathcal{AS}_{r_1,r_2}(\xi):= \left\{ \vphantom{\int} \right. \beta \in L^2(B_{r_2} \setminus B_{r_1};\R^{2 \times 2}) \colon  \ & \curl \beta =0 \mbox{ in } B_{r_2} \setminus B_{r_1}, \\ & \divr \beta \in L^2(B_{r_2} \setminus B_{r_1};\R^2), \ \int_{\partial B_{r_1}} \beta \cdot t \, d\mathcal{H}^1=\xi \left. \right\}.
	\end{align*}
	In the special case $r_2=1$, we simply write $\mathcal{AS}_{r_1}(\xi)$ instead of $\mathcal{AS}_{r_1,1}(\xi)$. We also define for every $\delta \in (0,1)$
	\begin{equation}\label{psi}
		\psi(\xi,\delta,U):=\min \left\{ \int_{B_1 \setminus B_\delta} \mathbb{C}_U \beta : \beta \, dx \colon \ \beta \in \mathcal{AS}_\delta(\xi) \right\}
	\end{equation}
	where $U \in \{ U_1,\dots,U_\ell \}$ and $\mathbb{C}_U := \partial^2 W/\partial F^2 (U)$. Notice that by elliptic regularity we have that a solution to the minimum problem \eqref{psi} exists and the minimum coincides with
	$$
	\min \left\{ \int_{B_1 \setminus B_\delta} \mathbb{C}_U \beta : \beta \, dx \colon \ \beta \in L^2(B_1 \setminus B_\delta; \R^{2 \times 2}), \ \curl \beta =0 \mbox{ in } B_1 \setminus B_\delta, \ \int_{\partial B_\delta} \beta \cdot t \, d\mathcal{H}^1=\xi \right\}.
	$$ 
	Hence, we can apply the results from \cite{GaLeoPonDislocations}. Notice, moreover, that $\psi(\cdot,\delta)$ is $2$-homogeneous for every $\delta$.
	\begin{prop}[\cite{GaLeoPonDislocations} Corollary 6, Remark 7]\label{prop: self energy}
		Let $\xi \in \R^2$, $\delta \in (0,1)$, $U \in \{U_1,\dots,U_\ell\}$ and let $\psi(\xi,\delta,U)$ be as in \eqref{psi}. Then, 
		$$
		\lim_{\delta \to 0} \frac{\psi(\xi,\delta,U)}{|\log \delta|} = \hat{\psi}(\xi,U),
		$$
		where $\hat{\psi} \colon \R^2 \times \{U_1,\dots,U_\ell\} \to [0,+\infty)$ is defined by
		\begin{equation}\label{hatpsi}
			\hat{\psi}(\xi,U):= \lim_{\delta \to 0} \frac{1}{|\log \delta|} \frac{1}{2} \int_{B_1 \setminus B_\delta} \mathbb{C}_U \beta_{\R^2}(\xi) : \beta_{\R^2}(\xi) \, dx,
		\end{equation}
		where $\beta_{\R^2}(\xi)$ is a distributional solution to
		\begin{equation*}
			\begin{cases*}
				\curl \beta_{\R^2}(\xi)=\xi \delta_0 & in $\R^2$, \\
				\divr \mathbb{C}_U \beta_{\R^2}(\xi)=0 & in $\R^2$.
			\end{cases*}
		\end{equation*}
	\end{prop}

	\begin{rem}\label{rem: selfenergy}
		Assume that $\rho_\varepsilon$ satisfies \eqref{rho1} and \eqref{rho2}. Then, from \cite[Proposition 8]{GaLeoPonDislocations} follows that the function $\psi_\varepsilon\colon \R^2 \times \{U_1,\dots,U_\ell\} \to [0,+\infty)$ defined as
		\begin{equation}\label{psieps}
			\psi_\varepsilon (\xi,U) := \frac{1}{|\log \varepsilon|} \min \left\{ \int_{B_{\rho_\varepsilon} \setminus B_\varepsilon} \mathbb{C}_U \beta : \beta \, dx \colon \ \beta \in \mathcal{AS}_{\varepsilon,\rho_\varepsilon}(\xi) \right\}
		\end{equation}
		satisfies
		$$
		\psi_\varepsilon(\xi,U)=\frac{\psi(\xi,\varepsilon,U)}{|\log \varepsilon|} (1+o(1)),
		$$
		where $o(1) \to 0$ as $\varepsilon \to 0$ uniformly with respect to $\xi$. In particular, $\psi_\varepsilon$ converges pointwise to $\hat{\psi}$ given by \eqref{hatpsi} as $\varepsilon \to 0$.
	\end{rem}

	We can now define the density $\varphi \colon K \times \R^2 \to [0,+\infty)$ of the self-energy through the following relaxation procedure
	\begin{equation}\label{selfenergy}
		\varphi(RU,\xi):= \min \left\{ \sum_{k=1}^M \lambda_k \hat{\psi}(R^T \xi_k,U) \colon \ \sum_{k=1}^M \lambda_k \xi_k=\xi, \ M \in \N, \ \lambda_k \geq 0,\ \xi_k \in \mathbb{S} \right\}.
	\end{equation}
	It follows directly from the definition that the function $\varphi$ is positively $1$-homogeneous and convex (see also \cite[Remark 9]{GaLeoPonDislocations}).

	\subsection{$\Gamma$-Convergence}
	
	Finally, we prove the $\Gamma$-convergence result for the energy $\mathcal{E}_\varepsilon$ defined in \eqref{scaled energy}. In what follows, we assume that $\eta_\varepsilon=\varepsilon|\log \varepsilon|/\rho_\varepsilon$ (this will be crucial for the proof of the upper bound).
	
	\begin{thm}\label{thm: gamma limite}
		The energy functionals $\mathcal{E}_\varepsilon$ defined in \eqref{scaled energy} $\Gamma$-converge with respect to the convergence defined in Definition \ref{def: convergence} to the functional $\mathcal{E}$ defined on $\mathcal{M}(\Omega;\R^2) \times L^2(\Omega;\R^{2 \times 2}) \times K$ by
		\begin{equation}\label{final functional}
			\mathcal{E}(\mu,\beta,RU):=
			\begin{cases*}
				\displaystyle	\frac{1}{2} \int_\Omega \mathbb{C}_U \beta : \beta \, dx+\int_\Omega \varphi\left( RU,\frac{d\mu}{d|\mu|} \right) \, d|\mu| & if $\mu \in H^{-1}(\Omega;\R^2) \cap \mathcal{M}(\Omega;\R^2)$, $\curl \beta=R^T \mu$, \\
				+\infty \vphantom{\int} & otherwise;
			\end{cases*}
		\end{equation}
	where $\mathbb{C}_U:=\partial^2 W/\partial F^2(U)$ and $\varphi$ is as in \eqref{selfenergy}.
	\end{thm}

	The proof of Theorem \ref{thm: gamma limite} will be subdivided in Proposition \ref{prop: liminf} and Proposition \ref{prop: limsup}. 

	\begin{prop}[$\Gamma-\liminf$ inequality] \label{prop: liminf}
		For every $(\mu,\beta,RU) \in (H^{-1}(\Omega;\R^2) \cap \mathcal{M}(\Omega;\R^2)) \times L^2(\Omega;\R^{2 \times 2}) \times K$ with $\curl \beta=R^T \mu$, and every sequence $(\mu_\varepsilon,\beta_\varepsilon) \in \mathcal{M}(\Omega;\R^2) \times L^2(\Omega;\R^{2 \times 2})$ converging to $(\mu,\beta,RU)$ in the sense of Definition \ref{def: convergence} we have 
		$$
		\liminf_{\varepsilon \to 0} \mathcal{E}_\varepsilon (\mu_\varepsilon,\beta_\varepsilon) \geq \mathcal{E}(\mu,\beta,RU).
		$$
	\end{prop}

	Before proving Proposition \ref{prop: liminf} we need the following energy estimate near the dislocations. Let us also introduce the following notation: for every $\vartheta_1,\vartheta_2 \in [0,2\pi)$ we define (in polar coordinates) the open sector
	\begin{equation}\label{sector}
		S(\vartheta_1,\vartheta_2):=\{ (r,\theta) \in \R^2 \colon r>0, \ \vartheta_1 < \theta < \vartheta_2 \},
	\end{equation}
	Finally, for every $x \in \R^2$ we define the open sector centred in $x$ as $S(x,\vartheta_1,\vartheta_2):=x+S(\vartheta_1,\vartheta_2)$.
	
	\begin{lem}\label{lem: gammaliminf}
		Let $(\mu,\beta,RU) \in (H^{-1}(\Omega;\R^2) \cap \mathcal{M}(\Omega;\R^2)) \times L^2(\Omega;\R^{2 \times 2}) \times K$ with $\curl \beta=R^T \mu$. Let $\varepsilon_j \to 0$ as $j \to +\infty$ and consider $(\mu_{\varepsilon_j},\beta_{\varepsilon_j}) \in X_{\varepsilon_j} \times \mathcal{AS}_{\varepsilon_j}(\mu_{\varepsilon_j})$ \textnormal{(}see \eqref{X epsilon} and \eqref{Admissible strains} for the precise definitions\textnormal{)} a sequence converging to $(\mu,\beta,RU)$ in the sense of Definition \ref{def: convergence} such that $\sup_j \mathcal{E}_{\varepsilon_j}(\mu_{\varepsilon_j},\beta_{\varepsilon_j})<+\infty$. Assume that $\mu_{\varepsilon_j}=\sum_{i=1}^{M_j} \xi_{i,j} \delta_{x_{i,j}}$. Then, for every $s \in (0,1)$ and $j $ large enough we have
		\begin{equation}\label{67}
			\sum_{i=1}^{M_j} \mathcal{E}_{\varepsilon_j} \left( \mu_{\varepsilon_j},\beta_{\varepsilon_j};B_{\rho_{\varepsilon_j}}(x_{i,j}) \right) \geq \frac{1}{|\log \varepsilon_j|} \sum_{i=1}^{M_j} \left( s-\frac{|\log \rho_{\varepsilon_j}|}{|\log \varepsilon_j|} \right) \hat{\psi}(R^T \xi_{i,j},U),
		\end{equation}
		where $\hat{\psi}$ is defined in \eqref{hatpsi}.
	\end{lem}

	\begin{proof}
		Let us denote $\mu_j=\mu_{\varepsilon_j}$ and $\beta_j=\beta_{\varepsilon_j}$ for brevity. By assumption, there exists $C>0$ such that for every $j \geq 1$
		\begin{equation}\label{68}
			\int_{\Omega_{\varepsilon_j}(\mu_j)} W(\beta_j) \, dx+\eta_{\varepsilon_j} \int_{\Omega_{\varepsilon_j}(\mu_j)} |\divr \beta_j|^2 \, dx \leq C \varepsilon_j^2 |\log \varepsilon_j|^2
		\end{equation}
		and there exists a sequence $(R_j) \subset SO(2)$ such that $R_j \to R$ and
		\begin{equation}\label{69}
			\int_{\Omega_j} |\beta_j-R_jU|^2 \, dx \leq C \varepsilon_j^2|\log \varepsilon_j|^2,
		\end{equation}
		where $\Omega_j:=\{ x \in \Omega \colon \dist(x,\partial \Omega)>\rho_{\varepsilon_j}/2 \}$. Notice moreover that by definition of $X_{\varepsilon_j}(\Omega)$ we have that $B_{\rho_{\varepsilon_j}}(x_{i,j}) \subset \Omega_j$ for every $j \geq 1$ and $i=1,\dots,M_j$.
		Fix $\delta \in (0,1/2)$. For every $i=1,\dots,M_j$ we decompose the annulus $B_{\rho_{\varepsilon_j}}(x_{i,j}) \setminus B_{\delta \varepsilon_j^s}(x_{i,j})$ centred at $x_{i,j}$ into dyadic annuli with constant ratio $\delta$ between the inner and outer radii. The annuli are defined as
		$$
		C_j^{k,i}:=B_{\delta^{k-1}\rho_{\varepsilon_j}}(x_{i,j}) \setminus B_{\delta^{k}\rho_{\varepsilon_j}}(x_{i,j}),
		$$
		and we consider only those corresponding to $k=1,\dots,\tilde{k}_j$, where
		\begin{equation}\label{70}
			\tilde{k}_j:= \lfloor k_j \rfloor +1 \qquad \mbox{and} \qquad k_j:=s\frac{|\log \rho_{\varepsilon_j}|}{|\log \delta|}-\frac{|\log \varepsilon_j|}{|\log \delta|}.
		\end{equation}
		Notice that $\delta^{\tilde{k}_j}\rho_{\varepsilon_j} \geq \delta^{k_j+1}\rho_{\varepsilon_j}=\delta \varepsilon_j^s$. Therefore, for every $i=1,\dots,M_j$ we have
		\begin{equation}\label{71}
			\frac{1}{\varepsilon_j^2|\log \varepsilon_j|^2} \int_{B_{\rho_{\varepsilon_j}}(x_{i,j}) \setminus B_{\delta \varepsilon_j^s}(x_{i,j})} W(\beta_j) \, dx \geq \frac{1}{|\log \varepsilon_j|^2} \sum_{k=1}^{\tilde{k}_j} \int_{C_j^{k,i}} \frac{W(\beta_j)}{\varepsilon_j^2} \, dx.
		\end{equation}
		We now divide the proof into two steps.
		
		\noindent \textbf{Step 1.} We claim that for every $j$ large enough, every $i=1,\dots,M_j$ and every $k=1,\dots,\tilde{k}_j$
		\begin{equation}\label{72}
			\int_{C^{k,i}_j} \frac{W(\beta_j)}{\varepsilon_j^2} \, dx \geq \psi(R^T \xi_{i,j},\delta,U)-\sigma_j,
		\end{equation}
		where $\psi(\cdot,\delta,U)$ is defined as in \eqref{psi} and $\sigma_j$ is a nonnegative infinitesimal sequence as $j \to +\infty$.
		As in \cite[Proposition 3.11]{ScaZep} we prove it arguing by contradiction. Assume that \eqref{72} does not hold true, then there exists a subsequence (not relabelled) such that, for every infinitesimal positive sequence $(\varsigma_j)$ there exists an index $i \in \{1,\dots,M_j\}$ and $k \in \{1,\dots,\tilde{K}_j\}$ such that for every $j \geq 1$
		\begin{equation}\label{73}
			\int_{C^{k,i}_j} \frac{W(\beta_j)}{\varepsilon_j^2} \, dx < \psi(R^T \xi_{i,j},\delta,U)-\varsigma_j.
		\end{equation}
		By assumption \eqref{W4} on the growth of $W$, from \eqref{73} we infer
		\begin{equation}\label{74}
			\int_{C^{k,i}_j} \dist^2(\beta_j,K) \, dx \leq C \psi(R^T \xi_{i,j},\delta,U)\varepsilon_j^2.
		\end{equation}
		We define the set
		$$
		\mathscr{C}_j^{k,i}:=\left( C_j^{k-1,i} \cup C_j^{k,i} \cup C^{k+1,i}_j \right)\cap S\left(x_{i,j},\frac{\pi}{4},\frac{7\pi}{4}\right)
		$$
		where $C^{0,i}_j:=B_{(1+\delta)\rho_{\varepsilon_j}}(x_{i,j}) \setminus B_{\rho_{\varepsilon_j}}(x_{i,j})$ and $S(x_{i,j},\cdot,\cdot)$ is as in \eqref{sector}. Notice that $\mathscr{C}_j^{k,i} \subset B_{2 \rho_{\varepsilon_j}}(x_{i,j})$ and it is a simply connected set, moreover, since $\delta<1/2$, we have (see Figure~\ref{fig:C})
		\begin{equation}\label{75}
			C_j^{k,i} \cap \left\{ (x_1,x_2) \in \R^2 \colon x_1<(x_{i,j})_1 \right\} \subset \left\{ x \in \mathscr{C}_j^{k,i} \colon \dist\left(x, \partial \mathscr{C}_j^{k,i}\right)> \frac{\pi}{8} \delta^{k+1} \rho_{\varepsilon_j} \right\}.
		\end{equation}

\begin{center}
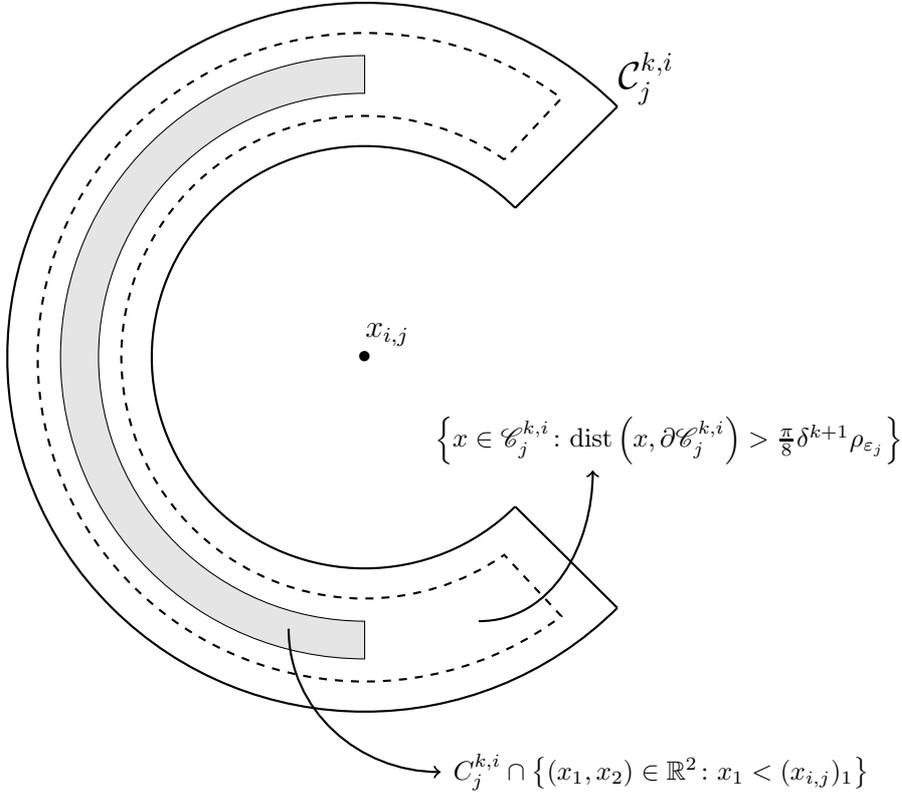
\begin{figure}[h!]
\begin{tikzpicture}
   \draw[thick] ([shift=(45:2.8cm)]0,0) arc (45:315:2.8cm);
     \draw[thick] ([shift=(45:4.7cm)]0,0) arc (45:315:4.7cm);
      \draw[thick] ([shift=(45:2.8cm)]0,0) -- ([shift=(45:4.7cm)]0,0);
      \draw[thick]  ([shift=(315:2.8cm)]0,0) --  ([shift=(315:4.7cm)]0,0);
     \draw[fill=gray, fill opacity=0.2] ([shift=(90:3.5cm)]0,0) arc (90:270:3.5cm)
     --
    ([shift=(270:4cm)]0,0) arc (270:90:4cm)
    --cycle;
    \draw [dashed, thick] ([shift=(55:3.2cm)]0,0) arc (55:305:3.2cm)
    --
        ([shift=(307:4.3cm)]0,0) arc (307:53:4.3cm)
        -- cycle;
    \node at (3.7,3.7) {{\LARGE{$\mathcal{C}^{k, i}_{j}$}}};
    \draw [->, thick] (1.5, -3.5) to [out=0,in=270] (3, -1.5);
    \node at (4,-1.1) {$ \left\{ x \in \mathscr{C}_j^{k,i} \colon \dist\left(x, \partial \mathscr{C}_j^{k,i}\right)> \frac{\pi}{8} \delta^{k+1} \rho_{\varepsilon_j} \right\}$};
     \draw [->, thick] (-1, -3.6) to [out=270,in=180] (1, -5.5);
    \node at (3.9,-5.5) {$C_j^{k,i} \cap \left\{ (x_1,x_2) \in \R^2 \colon x_1<(x_{i,j})_1 \right\}$};
    \node at (0, 0) {\textbullet};
    \node at (0.3, 0.3) {\Large{$x_{i, j}$}};
\end{tikzpicture}
\caption{Representation of the sets~$\mathcal{C}^{k, i}_{j}$ and~\eqref{75}.} \label{fig:C}
\end{figure}
\end{center}
		
		As $\beta_j \in \mathcal{AS}_{\varepsilon_j}(\mu_j)$, there exists $u_j \in H^1(\mathscr{C}_j^ {k,i};\R^2)$ such that $\beta_j = \nabla u_j$ on $\mathscr{C}_j^{k,i}$. Moreover, using \eqref{68} and the growth conditions on $W$ we estimate
		\begin{equation}\label{76}
			\int_{\mathscr{C}_j^{k,i}} \dist^2(\nabla u_j,K) \, dx \leq C\varepsilon_j^2|\log \varepsilon_j|^2,
		\end{equation}
		and, recalling that $\eta_{\varepsilon_j}=\varepsilon_j|\log \varepsilon_j|/\rho_{\varepsilon_j}$,
		\begin{equation}\label{77}
			\int_{\mathscr{C}_j^{k,i}} |\Delta u_j|^2 \, dx = \int_{\mathscr{C}_j^{k,i}} |\divr \beta_j|^2 \, dx \leq C \frac{\varepsilon_j^2|\log \varepsilon_j|^2}{\eta_{\varepsilon_j}} \leq C \rho_{\varepsilon_j}^2.
		\end{equation}
		Applying Theorem \ref{thm: multiwell rigidity crit} to $u_j$ on $\mathscr{C}_j^{k,i}$, we infer the existence of $C>0$ not depending on $j$, $i$ or $k$ (by scaling invariance) and a matrix $V_j \in K$ such that
		\begin{align}
			\begin{split}\label{78}
				\int_{\mathscr{C}_j^{k,i}} |\nabla u_j-V_j|^2 \, dx & \leq C \int_{\mathscr{C}_j^{k,i}} \dist^2(\nabla u_j,K) \, dx+ C \left( \int_{\mathscr{C}_j^{k,i}} |\Delta u_j| \right)^2 \\
				& \leq C\varepsilon_j^2 |\log \varepsilon_j|^2+C \mathcal{L}^2(\mathscr{C}_j^{k,i})^2 \int_{\mathscr{C}_j^{k,i}} |\Delta u_j|^2 \, dx \\
				& \leq \vphantom{\int} C(\delta^{k-2})^2 \rho_{\varepsilon_j}^4.
			\end{split}
		\end{align}
		where we have also used \eqref{76} and \eqref{77}.
		Let $v_j \in H^1\left( (B_1 \setminus B_{\delta^3}) \cap S(\pi/4,7\pi/4);\R^2 \right)$ be the function defined as $v_j(x):=u_j(\delta^{k-2}\rho_{\varepsilon_j}(x-x_{i,j}))-\delta^{k-2}\rho_{\varepsilon_j}V_j(x-x_{i,j})$ for every $x \in (B_1 \setminus B_{\delta^3}) \cap S(\pi/4,7\pi/4)=:\mathscr{C}$. Using \eqref{77}, \eqref{78} and a change of coordinates, we have
		\begin{align*}
		& \int_\mathscr{C} |\nabla v_j|^2 \, dx = \int_{\mathscr{C}_j^{k,i}} |\nabla u_j -V_j|^2 \, dx \leq C(\delta^{k-2})^2 \rho_{\varepsilon_j}^4, \\
		& \int_\mathscr{C} |\Delta v_j|^2 \, dx \leq C(\delta^{k-2})^2 \rho_{\varepsilon_j}^2 \int_{\mathscr{C}_j^{k,i}} |\Delta u_j|^2 \, dx \leq C(\delta^{k-2})^2 \rho_{\varepsilon_j}^4.
		\end{align*}
		Thus, recalling \eqref{75}, we estimate
		\begin{equation*}
			\int_{(B_\delta \setminus B_{\delta^2}) \cap \{ x_1 <0 \} } |\nabla^2 v_j|^2 \, dx \leq \frac{C}{\delta^3} (\delta^{k-2})^2 \rho_{\varepsilon_j}^4 = C(\delta) (\delta^{k-2})^2 \rho_{\varepsilon_j}^4,
		\end{equation*}
		from which we deduce using a change of coordinates
		\begin{equation*}
			\int_{C_j^{k,i} \cap \{ x_1 < (x_{i,j})_1 \}} |\nabla^2 u_j|^2 \, dx \leq \frac{1}{(\delta^{k-2})^2 \rho_{\varepsilon_j}^2} \int_{(B_\delta \setminus B_{\delta^2}) \cap \{ x_1 <0 \} } |\nabla^2 v_j|^2 \, dx \leq C(\delta) \rho_{\varepsilon_j}^2,
		\end{equation*}
		where $C(\delta)>0$ is a constant depending on $\delta$ but not on $j$, $i$ or $k$. Hence, we can use multiplicative rigidity, namely Proposition \ref{prop: rigidity multiplic form} to find $\tilde{U}^j \in \{ U_1,\dots,U_\ell \}$ such that
		\begin{equation}\label{79}
			\int_{C_j^{k,i} \cap \{ x_1 < (x_{i,j})_1 \}} \dist^2(\nabla u_j,SO(2)\tilde{U}^j) \, dx \leq C(1+C(\delta)\rho_{\varepsilon_j}^2)\int_{C_j^{k,i} \cap \{ x_1 < (x_{i,j})_1 \}} \dist^2(\nabla u_j,K) \, dx,
		\end{equation}
		where $C>0$ in \eqref{79} depends on the isoperimetric constant of the set $C_j^{k,i} \cap \{ x_1 < (x_{i,j})_1 \}$. However, being the constant in the relative isoperimetric inequality invariant after translation and rescaling of the set, we infer that $C$ does not depend on $j$, $k$ or $i$. Combining \eqref{74} with \eqref{79} and Theorem \ref{thm: FJM rigidity} we find $F_j^1 \in K$ such that
		\begin{equation}\label{80}
			\int_{C_j^{k,i} \cap \{ x_1 < (x_{i,j})_1 \}} |\beta_j -F_j^1|^2 \, dx = \int_{C_j^{k,i} \cap \{ x_1 < (x_{i,j})_1 \}} |\nabla u_j -F_j^1|^2 \, dx \leq C(\delta) \psi(R^T \xi_{i,j},\delta,U)\varepsilon_j^2.
		\end{equation}
		Reasoning analogously, we find $F_j^2,F_j^3 \in K$ such that \eqref{80} holds in the sets $C_j^{k,i} \cap \{ x_2 < (x_{i,j})_2 \}$ and $C_j^{k,i} \cap \{ x_2 > (x_{i,j})_2 \}$, respectively. Finally, using the usual technique of covering with overlapping for the set $C_j^{k,i}$, we obtain the global estimate with a single matrix $F_j \in K$
		\begin{equation}\label{81}
		\int_{C_j^{k,i}} |\beta_j -F_j|^2 \, dx  \leq C(\delta) \psi(R^T \xi_{i,j},\delta,U)\varepsilon_j^2,
		\end{equation}
		where $C(\delta)>0$ does not depend on $j$, $i$ or $k$.
		
		Now, since $F_j \in K$ we have that $F_j=\overline{R}_j U^j$ with $\overline{R}_j \in SO(2)$ and $U^j \in \{U_1,\dots,U_\ell\}$ for every $j \geq 1$. We want to prove that $U^j=U$ for every $j$ large enough and $\lim_j \overline{R}_j=R$. Observe that, since $C_j^{k,i} \subset \Omega_j$ for every $j$,$i$ and $k$, using \eqref{69}, \eqref{81} and the triangular inequality, we estimate
		\begin{align*}
			|\overline{R}_jU^j-R_jU|^2 & \leq \frac{C}{\rho_{\varepsilon_j}^2 \delta^{2k}} \left( C(\delta)\psi(R^T \xi_{i,j},\delta,U) \varepsilon^2_j+\varepsilon_j^2 |\log \varepsilon_j|^2 \right) \\
			&  \leq C \left( \frac{\varepsilon_j}{\rho_{\varepsilon_j} \delta^k} \right)^2 \left( \psi(R^T \xi_{i,j},\delta,U)+|\log \varepsilon_j|^2 \right).
		\end{align*}
		Recalling \eqref{70}, we conclude $\frac{\varepsilon_j}{\rho_{\varepsilon_j} \delta^k} \leq \frac{\varepsilon_j^{1-s}}{\delta}$ and hence, $|\overline{R}_jU^j-R_jU|^2 \leq C \varepsilon_j^{2-2s}|\log \varepsilon_j|^2$. Thus, there exists $J \in \N$ such that for every $j \geq J$ we have $U^j=U$ and $\lim_j \overline{R}_j=R$. Arguing as in \cite[Proposition 3.11]{ScaZep} and recalling the expression of $\psi(\cdot,\delta,U)$, we find an infinitesimal positive sequence $(\varsigma_j)$ such that for every $j \geq J$
		\begin{equation*}
			\int_{C^{k,i}_j} \frac{W(\beta_j)}{\varepsilon_j^2} \, dx \geq \psi(R^T \xi_{i,j},\delta,U)-\varsigma_j.
		\end{equation*}
		Thus the contradiction.
		
		\noindent \textbf{Step 2.} Combining \eqref{71} and the claim \eqref{72}, we obtain that for every $\delta \in (0,1/2)$ and $j \geq J$ it holds
		\begin{align*}
			\sum_{i=1}^{M_j} \mathcal{E}_{\varepsilon_j} \left( \mu_{\varepsilon_j},\beta_{\varepsilon_j};B_{\rho_{\varepsilon_j}}(x_{i,j}) \right) & \geq \frac{1}{|\log \varepsilon_j|^2} \sum_{i=1}^{M_j} \sum_{k=1}^{\tilde{k}_j} \left( \psi(R^T \xi_{i,j},\delta,U)-\sigma_j \right) \\
			& \geq \frac{1}{|\log \varepsilon_j|} \sum_{i=1}^{M_j} \left(s-\frac{|\log \rho_{\varepsilon_j}|}{|\log \varepsilon_j|} \right) \left( \frac{\psi(R^T \xi_{i,j},\delta,U)}{|\log \delta|}-\frac{\sigma_j}{|\log \delta|} \right).
		\end{align*}
		Using Proposition \ref{prop: self energy} and Remark \ref{rem: selfenergy}, sending $\delta \to 0$ we finally get for every $j \geq J$
		\begin{equation*}
			\liminf_{j \to +\infty} \sum_{i=1}^{M_j} \mathcal{E}_{\varepsilon_j} \left( \mu_{\varepsilon_j},\beta_{\varepsilon_j};B_{\rho_{\varepsilon_j}}(x_{i,j}) \right) \geq \frac{1}{|\log \varepsilon_j|} \sum_{i=1}^{M_j} \left( s-\frac{|\log \rho_{\varepsilon_j}|}{|\log \varepsilon_j|} \right) \hat{\psi}(R^T \xi_{i,j},U),
		\end{equation*}
		that is \eqref{67}.
	\end{proof}
	
	Now we are ready to prove the $\Gamma-\liminf$.
	\begin{proof}[Proof of Proposition \ref{prop: liminf}]
		Let $(\mu,\beta,RU) \in (H^{-1}(\Omega;\R^2) \cap \mathcal{M}(\Omega;\R^2)) \times L^2(\Omega;\R^{2 \times 2}) \times K$ with $\curl \beta=R^T \mu$ and $(\mu_\varepsilon,\beta_\varepsilon) \in \mathcal{M}(\Omega;\R^2) \times L^2(\Omega;\R^{2 \times 2})$ be as in the statement and assume that $\liminf_{\varepsilon \to 0} \mathcal{E}_\varepsilon(\mu_\varepsilon,\beta_\varepsilon)=\lim_{j \to +\infty} \mathcal{E}_{\varepsilon_j} (\mu_{\varepsilon_j},\beta_{\varepsilon_j})$. Suppose moreover that $\mathcal{E}_{\varepsilon_j}(\mu_{\varepsilon_j},\beta_{\varepsilon_j}) \leq C$ for every $j \geq 1$, otherwise there is nothing to prove. This implies in particular that $(\mu_{\varepsilon_j},\beta_{\varepsilon_j}) \in X_{\varepsilon_j} \times \mathcal{AS}_{\varepsilon_j}(\mu_{\varepsilon_j})$. 
		
		Let us set $\mu_j:=\mu_{\varepsilon_j}$ and $\beta_j:=\beta_{\varepsilon_j}$. We decompose the energy into a contribution far from the dislocations, in $\Omega_{\rho_{\varepsilon_j}}(\mu_j):=\Omega \setminus \cup_{i=1}^{M_j} B_{\rho_{\varepsilon_j}}(x_{i,j})$, and a contribution close to the dislocations:
		$$
		\mathcal{E}_{\varepsilon_j}(\mu_j,\beta_j)=\mathcal{E}_{\varepsilon_j}\left(\mu_j,\beta_j;\Omega_{\rho_{\varepsilon_j}}(\mu_j)\right)+\sum_{i=1}^{M_j} \mathcal{E}_{\varepsilon_j} \left( \mu_j,\beta_j;B_{\rho_{\varepsilon_j}}(x_{i,j}) \right).
		$$
		
		\noindent \textit{Lower bound far from the dislocations.}
		For the energy contribution far from the dislocations we will perform a second order expansion at scale $\varepsilon_j |\log \varepsilon_j|$ of $W$ at $U$. We have $W(U+F)=\frac{1}{2} \mathbb{C}_U F : F +\sigma(F)$, where $\sigma(F)/|F|^2 \to 0$ as $|F| \to 0$. Then, setting $\omega(t):=\sup_{|F| \leq t} |\sigma(F)|$, we have
		\begin{equation}\label{82}
		W(U+\varepsilon_j|\log \varepsilon_j| F) \geq \frac{1}{2} \varepsilon_j^2 |\log \varepsilon_j|^2 \mathbb{C}_U F:F-\omega(\varepsilon_j |\log \varepsilon_j| |F|),
		\end{equation}
		with $\omega(t)/t^2 \to 0$ as $t \to 0$. Set
		$$
		G_j:=\frac{R^T_{\varepsilon_j} \beta_j - U}{\varepsilon_j |\log \varepsilon_j|},
		$$
		recall that $\Omega_j:=\{ x \in \Omega \colon \ \dist(x,\partial \Omega) > \rho_{\varepsilon_j}/2 \}$ and define the characteristic function 
		\begin{equation*}
			\chi_j:=
			\begin{cases*}
				1 & if $x \in \Omega_j \cap \Omega_{\rho_{\varepsilon_j}}(\mu_j)$ and $|G_j| \leq \varepsilon_j^{-1/2}$, \\
				0 & otherwise.
			\end{cases*}
		\end{equation*}
		By the boundedness of $G_j$ in $L^{2}(\Omega_j;\R^{2 \times 2})$ and recalling the definitions of $\Omega_j$ and $\Omega_{\rho_{\varepsilon_j}}(\mu_j)$, we have that $\chi_j \to 1$ boundedly in measure on $\Omega$. Therefore, from the definition of convergence \ref{def: convergence}, we deduce that
		\begin{equation*}
			\tilde{G}_j :=\chi_j G_j \rightharpoonup \beta \qquad \mbox{in $L^2(\Omega;\R^{2 \times 2})$}.
		\end{equation*} 
		Using the frame indifference of $W$ and \eqref{82}, we infer
		\begin{align*}
			\mathcal{E}_{\varepsilon_j}\left(\mu_j,\beta_j;\Omega_{\rho_{\varepsilon_j}}(\mu_j)\right) & \geq \frac{1}{\varepsilon_j^2 |\log \varepsilon_j|^2} \int_\Omega \chi_j W(R^T_{\varepsilon_j} \beta_j) \, dx \\
			& = \frac{1}{\varepsilon_j^2 |\log \varepsilon_j|^2} \int_\Omega \chi_j W\left( U+\varepsilon_j |\log \varepsilon_j|G_j \right) \, dx \\
			& \geq \int_\Omega \left( \frac{1}{2} \mathbb{C}_U \tilde{G}_j : \tilde{G}_j-\chi_j \frac{\omega(\varepsilon_j|\log \varepsilon_j| |G_j|)}{\varepsilon_j^2 |\log \varepsilon_j|^2} \right) \, dx.
		\end{align*}
		Arguing as in \cite[Theorem 4.6]{MuScaZep} we conclude that
		\begin{equation}\label{83}
			\liminf_{j \to +\infty} \mathcal{E}_{\varepsilon_j} \left( \mu_j,\beta_j;\Omega_{\rho_{\varepsilon_j}}(\mu_j)\right) \geq \frac{1}{2} \int_\Omega \mathbb{C}_U \beta : \beta \, dx.
		\end{equation}
		
		\noindent \textit{Lower bound close to the dislocations.}
		Notice that the sequence $(\mu_j,\beta_j)$ satisfies the assumptions of Lemma \ref{lem: gammaliminf}. Hence, for every $s \in (0,1)$ and every $j$ large enough
		\begin{equation}\label{84}
			\sum_{i=1}^{M_j} \mathcal{E}_{\varepsilon_j} \left( \mu_{\varepsilon_j},\beta_{\varepsilon_j};B_{\rho_{\varepsilon_j}}(x_{i,j}) \right) \geq \frac{1}{|\log \varepsilon_j|} \sum_{i=1}^{M_j} \left( s-\frac{|\log \rho_{\varepsilon_j}|}{|\log \varepsilon_j|} \right) \hat{\psi}(R^T \xi_{i,j},U).
		\end{equation}
		By definition of $\mu_j$ and formula \eqref{selfenergy} follows 
		\begin{equation}\label{85}
			\frac{1}{|\log \varepsilon_j|} \sum_{i=1}^{M_j} \hat{\psi}(R^T \xi_{i,j},U)  \geq \int_\Omega \varphi \left( RU, \frac{d \tilde{\mu}_j}{d |\tilde{\mu}_j|} \right) \, d |\tilde{\mu}_j|,
		\end{equation}
		where $\tilde{\mu}_j:= \mu_j /(\varepsilon_j|\log \varepsilon_j|)$. Recall that by assumption $\tilde{\mu}_j \xrightharpoonup{*} \mu$ in $\mathcal{M}(\Omega;\R^2)$ as $j \to +\infty$. Notice that $\textnormal{Span}_\R \Sf=\R^2$, hence, the convex $1$-homogeneous function $\varphi$ is finite on $\R^2$ and therefore continuous. Then, Reshetnyak’s lower-semicontinuity Theorem, \eqref{84} and \eqref{85} give
		\begin{equation}\label{86}
			\sum_{i=1}^{M_j} \mathcal{E}_{\varepsilon_j} \left( \mu_{\varepsilon_j},\beta_{\varepsilon_j};B_{\rho_{\varepsilon_j}}(x_{i,j}) \right) \geq \int_\Omega s \varphi \left( RU, \frac{d \mu}{d |\mu|} \right) \, d |\mu|.
		\end{equation}
		Combining \eqref{83}, \eqref{86} and letting $s \to 1$ we conclude the lower bound.
	\end{proof}

	Next, we prove the existence of a recovery sequence for the functional $\mathcal{E}$.
	
	\begin{prop}[$\Gamma-\limsup$ inequality]\label{prop: limsup}
		Let $(\mu,\beta,RU) \in (H^{-1}(\Omega;\R^2) \cap \mathcal{M}(\Omega;\R^2)) \times L^2(\Omega;\R^{2 \times 2}) \times K$ with $\curl \beta=R^T \mu$. There exists a sequence $(\mu_j,\beta_j) \in X_{\varepsilon_j} \times \mathcal{AS}_{\varepsilon_j}(\mu_j)$ converging to $(\mu,\beta,RU)$ in the sense of Definition \ref{def: convergence} and such that: $\varepsilon_j \to 0$ as $j \to +\infty$ and
		\begin{equation}\label{87}
			\limsup_{j \to +\infty} \mathcal{E}_{\varepsilon_j}(\mu_j,\beta_j) \leq \mathcal{E}(\mu,\beta,RU).
		\end{equation}
	\end{prop}

	Before proving the upper bound, we need a density result. The construction is identical to the one in \cite[Step 3 of the proof of Theorem 12]{GaLeoPonDislocations}, we have only to check in addition that the divergence of the approximants is an $L^2$ function on $\Omega$. Hence, we only sketch the main steps of the proof.  In the statement below, a locally constant measure is an absolutely continuous measure with respect to the Lebesgue one whose density is constant on a regular partition of the domain.
	
	\begin{lem}\label{lem: density}
		For every $(\mu,\beta,RU) \in (H^{-1}(\Omega;\R^2) \cap \mathcal{M}(\Omega;\R^2)) \times L^2(\Omega;\R^{2 \times 2}) \times K$ with $\curl \beta=R^T \mu$, there exists a sequence $(\mu_m,\beta_m)$ such that $\beta_m \in L^\infty(\Omega;\R^{2 \times 2})$, $\divr \beta_m \in L^2(\Omega;\R^2)$, $\mu_m \in \mathcal{M}(\Omega;\R^2)$ is a locally constant measure on $\Omega$, $\curl \beta_m = R^T \mu_m$ and
		\begin{equation}\label{88}
			\beta_m \to \beta \in L^2(\Omega;\R^{2 \times 2}), \qquad \mu_m \xrightharpoonup{*} \mu \in \mathcal{M}(\Omega;\R^2) \qquad \mbox{and} \qquad |\mu_m|(\Omega) \to |\mu|(\Omega).
		\end{equation}
		Moreover, the measure $\mu_m$ can be chosen such that it is constant on the sets belonging to a partition $\{E_l\}_{l=1}^L$ of $\Omega$, where every $E_l$ is a simply connected open set with Lipschitz boundary.
	\end{lem}

	\begin{proof}
		Assume without loss of generality that $R=I$. By standard reflection and convolution arguments we can find sequences $(f_h) \in C^\infty(\overline{\Omega};\R^{2 \times 2})$, $(g_h) \in C^\infty(\overline{\Omega};\R^2)$ with $\curl f_h = g_h$ and such that
		\begin{equation}\label{89}
			f_h \to \beta \in L^2(\Omega;\R^{2 \times 2}), \qquad g_h \, dx \xrightharpoonup{*} \mu \in \mathcal{M}(\Omega;\R^2) \qquad \mbox{and} \qquad |g_h \, dx|(\Omega) \to |\mu|(\Omega).
		\end{equation}
		Further, we can approximate each $g_h$ with piecewise constant functions in $\Omega$. That is, we find piecewise constant functions $(g_{h,k})$ such that
		\begin{equation}\label{90}
			\Vert g_{h,k} - g_h \Vert_{L^\infty(\Omega;\R^2)} \to 0 \mbox{ as } k \to +\infty \qquad \mbox{and} \qquad \int_\Omega (g_{h,k}-g_h) \, dx =0.
		\end{equation}
		Let $B_R$ be a ball in $\R^2$ such that $\Omega \subset B_R$. Consider now $r_{h,k}$ the solution of the following problem
		\begin{equation}\label{91}
			\begin{cases*}
				\curl \ r_{h,k}=\chi_{\Omega}(g_{h,k}-g_h) & in $B_R$, \\
				\divr \ r_{h,k}=0 & in $B_R$, \\
				r_{h,k} \cdot t =0 & on $\partial B_R$.
			\end{cases*}
		\end{equation}
		Observe that $r_{h,k} J=\nabla z_{h,k}$ on $B_R$, where $J$ is the anticlockwise rotation and $z_{h,k} \in H^1(B_R;\R^2)$. Moreover, $z_{h,k}$ satisfies the following system (recall \eqref{90})
		\begin{equation*}
			\begin{cases*}
				\Delta z_{h,k}=\chi_{\Omega}(g_{h,k}-g_h) & in $B_R$, \\
				\nabla z_{h,k} \cdot \nu =0 & on $\partial B_R$.
			\end{cases*}
		\end{equation*}
		By standard elliptic estimates it holds $\Vert z_{h,k} \Vert_{W^{2,p}(B_R;\R^{2})} \leq C(p) \Vert g_{h,k} - g_h \Vert_{L^p(\Omega;\R^2)}$ for every $p \in (1,\infty)$. Hence, by Sobolev Embedding Theorem we have  
		\begin{equation}\label{92}
		\Vert r_{h,k} \Vert_{L^\infty(B_R;\R^{2 \times 2})} \leq C \Vert g_{h,k} - g_h \Vert_{L^\infty(\Omega;\R^2)}.
		\end{equation}
		Set $f_{h,k}:=f_h + r_{h,k}$. By \eqref{91} we have $\curl f_{h,k}=g_{h,k}$ and $\divr f_{h,k}= \divr f_h$. Moreover, by \eqref{90} and \eqref{92} we also have $f_{h,k} \to f_h$ in $L^2(\Omega;\R^{2 \times 2})$ as $k \to +\infty$. Hence, recalling \eqref{89} and using a diagonal argument, we can extract the desired sequence $(\mu_m,\beta_m)$ satisfying \eqref{88}, with $\mu_m = g_m \mathrm{d}x$. 
	\end{proof}

	We will also employ the following construction in order to deal with the penalisation in the $\Gamma-\limsup$.
	
	\begin{lem}\label{lem: interpolation}
		Let $\mathbb{C}$ be a linear symmetric operator on $\R^{2 \times 2}$ with the property that there exist $\lambda_1,\lambda_2 >0$ such that for every $\beta \in \R^{2 \times 2}$
		$$
		\lambda_1 |\beta_\textnormal{sym}|^2 \leq \mathbb{C} \beta \cdot \beta \leq \lambda_2 |\beta_\textnormal{sym}|^2.
		$$ 
		Let $\xi \in \R^2$ and $\beta_{\R^2} \colon \R^2 \setminus \{0\} \to \R^{2 \times 2 }$ be the smooth matrix field such that
		\begin{equation*}
			\begin{cases*}
				\curl \beta_{\R^2}(\xi)=\xi \delta_0 & in $\R^2$, \\
				\divr \mathbb{C} \beta_{\R^2}(\xi) = 0 & in $\R^2$;
			\end{cases*}
		\end{equation*}
		as defined in Proposition \eqref{prop: self energy}.
		Then, there exists a smooth matrix field $\zeta_\xi \colon \R^2 \setminus \{0\} \to \R^{2 \times 2}$ and a constant $C=C(\lambda_1,\lambda_2)>0$ such that:
		\begin{align}
			& \curl \, \zeta_\xi =0 \qquad \mbox{in $\R^2 \setminus \{0\}$}, \label{93} \\
			& \int_{\partial B_r} \zeta_\xi \cdot t \, d\mathcal{H}^1 = \xi \qquad \mbox{for every $r>0$}, \label{94} \\
			& \divr \, \zeta_\xi =0 \qquad \mbox{in $B_{1/2}$}, \label{95} \\
			& \zeta_\xi = \beta_{\R^2}(\xi) \qquad \mbox{in $\R^2 \setminus B_1$}, \label{96} \\
			& |\zeta_\xi(x)| \leq \frac{C}{|x|}|\xi| \qquad  \mbox{in $\R^2 \setminus \{0\}$}, \label{97} \\
			& |\nabla \zeta_\xi(x)| \leq \frac{C}{|x|^2}|\xi| \qquad \mbox{in $\R^2 \setminus \{0\}$} \label{98}.
		\end{align}
	\end{lem}

	\begin{proof}
		Observe that by elliptic regularity, the strain $\beta_{\R^2}(\xi)$ is smooth on $\R^2 \setminus \{0\}$, moreover, in polar coordinates it is of the form $\beta_{\R^2}(\xi)(r,\theta)=\frac{1}{r}\Gamma_\xi(\theta)$ where $\Gamma_\xi$ depends on the tensor $\mathbb{C}$ and it is linear in $\xi$ (we refer to \cite{BACON} for a detailed treatment of the subject). Hence, there exists $C=C(\lambda_1,\lambda_2)>0$ such that for every $\theta \in [0,2\pi)$
		$$
		|\Gamma_\xi(\theta)| \leq C|\xi| \qquad \mbox{and} \qquad |\partial_\theta \Gamma_\xi(\theta)| \leq C|\xi|.
		$$
		This in turn implies that 
		\begin{equation}\label{99}
			|\beta_{\R^2}(\xi)(r,\theta)| \leq \frac{C}{r}|\xi| \qquad \mbox{and} \qquad |\nabla \beta_{\R^2}(\xi)(r,\theta)| \leq \frac{C}{r^2}|\xi|.
		\end{equation}
		Let $\beta_0(\xi) \colon \R^2 \setminus \{0\} \to \R^{ 2 \times 2}$ be the distributional solution of the following system
		\begin{equation*}
			\begin{cases*}
				\curl \beta_{0}(\xi)=\xi \delta_0 & in $\R^2$, \\
				\divr \beta_{0}(\xi) = 0 & in $\R^2$.
			\end{cases*}
		\end{equation*}
		We have that for $\beta_0(\xi)$ the same consideration of $\beta_{\R^2}(\xi)$ hold.
		Observe that $\curl(\beta_0(\xi)-\beta_{\R^2}(\xi))=0$ on $\R^2$, hence, there exists a smooth function $u \colon \R^2 \setminus \{0\} \to \R^2$ such that $\nabla u = \beta_0(\xi)-\beta_{\R^2}(\xi)$. We can assume that $u(x_0)=0$ for $x_0 \in B_1 \setminus B_{1/2}$, therefore, using \eqref{99} we have
		\begin{equation}\label{100}
			\sup_{B_1 \setminus B_{1/2}} |u| \leq (1+\pi) \sup_{B_1 \setminus B_{1/2}} |\nabla u| \leq C.
		\end{equation}
		Let $f \colon \R^2 \to \R$ be a smooth cut off function such that $f \equiv 1$ on $B_{1/2}$, $f \equiv 0$ on $\R^2 \setminus B_1$ and $|\nabla f|+|\nabla^2 f| \leq C$ on $\R^2$.
		
		Set $\zeta_\xi:=\beta_{\R^2}(\xi) + \nabla(fu)$. By construction, \eqref{93}--\eqref{96} hold. To see \eqref{97}, observe that by \eqref{99} and \eqref{100} we have
		$$
		|\zeta_\xi(x)| \leq (1-f(x))|\beta_{\R^2}(\xi)(x)|+f(x)|\beta_0(\xi)(x)|+|u(x)||\nabla f(x)| \leq \frac{C}{|x|}|\xi|.
		$$
		Finally, we estimate
		$$
		|\nabla \zeta_\xi(x)| \leq (1-f(x))|\nabla \beta_{\R^2}(\xi)(x)|+f(x)|\nabla \beta_0(\xi)(x)|+2|\nabla f(x)||\nabla u(x)|+|u(x)||\nabla^2 f(x)| \leq \frac{C}{|x|^2}|\xi|
		$$
		and conclude \eqref{98}.
	\end{proof}

	\begin{proof}[Proof of Proposition \ref{prop: limsup}.]
		Without loss of generality, to simplify the notation we consider only the case $R=I$.
		We divide the proof in three steps.
		
		\noindent \textbf{Step 1.} We first show the inequality for $\beta \in L^\infty(\Omega;\R^{2 \times 2})$, $\divr \beta \in L^2(\Omega;\R^2)$ and $\mu = \xi \chi_E \, dx$, where $\xi \in \R^2$ and $\overline{E} \subset \Omega$ is a simply connected open set with Lipschitz boundary.
		
		Let $\hat{\psi}$ be as in \eqref{hatpsi}. By Remark \eqref{rem: selfenergy} there exist $M \in \N$, $\xi_1,\dots,\xi_M \in \Sf$ and $\lambda_k \geq 0$ for every $k=1,\dots,M$ such that
		\begin{equation}\label{101}
			\varphi(U,\xi)=\sum_{k=1}^M \lambda_k \hat{\psi}(\xi_k,U), \qquad \mbox{and} \qquad \xi=\sum_{k=1}^M \lambda_k \xi_k.
		\end{equation}
		Set
		$$
		\Lambda:= \sum_{k=1}^M \lambda_k, \qquad r_j:=\frac{1}{2\sqrt{\Lambda|\log \varepsilon_j|}}.
		$$
		Recalling \eqref{rho2} we have that $r_j \gg \rho_{\varepsilon_j}$. Using \cite[Lemma 14]{GaLeoPonDislocations}, there exists a sequence of admissible measures $\mu_j \in X_{\varepsilon_j}$ of the form
		\begin{equation*}
			\mu_j = \sum_{k=1}^M \varepsilon_j \xi_k \mu_j^k, \qquad \mbox{where} \qquad \mu_j^k=\sum_{i=1}^{M_j^k} \delta_{x_{i,k}^j},
		\end{equation*}
		with the property that
		$$
		B_{r_j}(x_{i,k}^j) \subset E, \qquad |x_{i,k}^j-x_{i',k'}^j| \geq 2r_j \ \ \mbox{if $(i,k) \neq (i',k')$}
		$$
		and
		\begin{align}
			& \frac{\mu_j^k}{|\log \varepsilon_j|} \xrightharpoonup{*} \lambda_k \chi_E \, dx \qquad \mbox{weakly in $\mathcal{M}(\Omega)$ as $j \to +\infty$ for every $k=1,\dots,M$,} \label{102} \\
			& \frac{\mu_j}{\varepsilon_j|\log \varepsilon_j|} \xrightharpoonup{*} \mu \qquad \mbox{weakly in $\mathcal{M}(\Omega;\R^2)$ as $j \to +\infty$}. \label{103}
		\end{align}
		It is useful to combine the two summations in the definition of $\mu_j$ into one as $\mu_j=\sum_{i=1}^{M_j} \varepsilon_j \xi_{i,j} \delta_{x_{i,j}}$.
		
		Let $(\gamma_j) \subset (0,1)$ be an increasing sequence such that $\gamma_j \to 1$ as $j \to +\infty$. Let $\mathbb{C}_U:=\partial^2 W/\partial F^2(U)$. By Lemma \ref{lem: interpolation}, for every $\xi_{i,j}$ we can find $\zeta_{\xi_{i,j}}$ satisfying \eqref{93}--\eqref{98}. Consider a smooth cut off function $f \colon \R^2 \to \R$ such that $f \equiv 1$ on $B_{1/2}$, $\supp f \subset B_1$ and $|\nabla f|+|\nabla^2 f| \leq C$ on $\R^2$. We define the field $\zeta_j $ as
		$$
		\zeta_j:= \sum_{i=1}^{M_j} \zeta_{i,j}, \qquad \mbox{where} \qquad \zeta_{i,j}:= \varepsilon_j^{1-\gamma_j} \zeta_{\xi_{i,j}}\left( \frac{x-x_{i,j}}{\varepsilon_j^{\gamma_j}} \right) f\left( \frac{x-x_{i,j}}{r_j} \right).
		$$
		For every $i=1,\dots,M_j$, by property \eqref{94} of $\zeta_{\xi_{i,j}}$ and a change of variables $x \mapsto x_{i,j}+\varepsilon_j^{\gamma_j}x$ we infer
		\begin{equation}\label{104}
			\int_{\partial B_{\varepsilon_j}(x_{i,j})} \zeta_{i,j} \cdot t \, d \mathcal{H}^1 =  \varepsilon_j^{\gamma_j} \int_{\partial B_{\varepsilon_j^{1-\gamma_j}}} \varepsilon^{1-\gamma_j}_j \zeta_{\xi_{i,j}} \cdot t \, d \mathcal{H}^1 = \varepsilon_j \xi_{i,j}.
		\end{equation}
		Moreover, by property \eqref{93} we have that $\curl \eta_{i,j}$ vanishes outside $\{x_{i,j}\} \cup \left( B_{r_j}(x_{i,j}) \setminus B_{r_j/2}(x_{i,j}) \right)$, hence,
		\begin{equation}\label{105}
			\int_{B_{r_j}(x_{i,j}) \setminus B_{\varepsilon_j}(x_{i,j})} \curl \zeta_{i,j} \, dx = - \int_{\partial B_{r_j/2}(x_{i,j})} \zeta_{i,j} \cdot t \, d \mathcal{H}^1 = -\varepsilon_j \xi_{i,j}.
		\end{equation}
		Define the measure $\nu_j:=\chi_{\Omega_{\varepsilon_j}(\mu_j)} \curl \zeta_j$ and $\tilde{\beta}_j:=\nabla w_j J$, where $J$ is the anticlockwise rotation of $\pi/2$ and $w_j \in H^1_0(B_R;\R^2)$ is the distributional solution of the following system
		\begin{equation}\label{106}
			\begin{cases*}
				\Delta w_j=\varepsilon_j|\log \varepsilon_j| \mu+ \nu_j & in $B_R$ \\
				w_j = 0 & on $\partial B_R$;
			\end{cases*}
		\end{equation}
		and $B_R$ is some ball containing $\Omega$. By definition we have $\curl \tilde{\beta}_j=-\varepsilon_j|\log \varepsilon_j| \mu- \nu_j$.
		
		We are now ready to define the recovery sequence as 
		\begin{equation}\label{107}
			\beta_j:=
			\begin{cases*}
				U+\zeta_j+\varepsilon_j|\log \varepsilon_j| \beta+\tilde{\beta}_j & on $\overline{\Omega_{\varepsilon_j}(\mu_j)}$, \\
				I & on $\cup_{i=1}^{M_j} B_{\varepsilon_j}(x_{i,j})$.
			\end{cases*}
		\end{equation}
		Notice that by definition $\curl \beta_j=0$ on $\Omega_{\varepsilon_j}(\mu_j)$ (as $\nu_j=\chi_{\Omega_{\varepsilon_j}(\mu_j)} \curl \zeta_j$). Moreover,
		$$
		\curl \left( \varepsilon_j|\log \varepsilon_j| \beta+\tilde{\beta}_j \right)=\varepsilon_j |\log \varepsilon_j|\mu -\varepsilon_j |\log \varepsilon_j|\mu-\nu_j=0 \qquad \mbox{on } \bigcup_{i=1}^{M_j} B_{\varepsilon_j}(x_{i,j}),
		$$
		hence, for every $i=1,\dots,M_j$
		$$
		\int_{\partial B_{\varepsilon_j}(x_{i,j})} \beta_j \cdot t \, d \mathcal{H}^1=\int_{\partial B_{\varepsilon_j}(x_{i,j})} \zeta_{i,j} \cdot t \, d \mathcal{H}^1=\varepsilon_j \xi_{i,j}.
		$$
		Finally, we have $\divr \beta_j = \divr \zeta_j + \varepsilon_j |\log \varepsilon_j| \beta$, recalling \eqref{95} and \eqref{98} together with the definition of $\zeta_j$, we conclude $\divr \beta_j \in L^2(\Omega_{\varepsilon_j}(\mu_j);\R^2)$. Thus, $\beta_j \in \mathcal{AS}_{\varepsilon_j}(\mu_j)$.
	
		We need to show that $(\mu_j,\beta_j)$ converges to $(\mu,\beta)$ in the sense of Definition \ref{def: convergence}. Recalling \eqref{103}, we only need to show that
		\begin{equation}\label{108}
			\frac{\beta_j - U}{\varepsilon_j|\log \varepsilon_j|} \rightharpoonup \beta \qquad \mbox{in $L^2(\Omega;\R^{2 \times 2 })$.}
		\end{equation}
		First, observe that since the $\xi_{i,j}$'s are uniformly bounded in norm and by \eqref{97}, we have
		\begin{align*}
			\int_{\Omega_\varepsilon(\mu_j)} |\zeta_j|^2 \, dx \leq C M_j \varepsilon_j^2 |\log \varepsilon_j|^2 \leq C \varepsilon_j^2 |\log \varepsilon_j|^2
		\end{align*}
		and
		$$
		\int_\Omega |\zeta_j| \, dx \leq C M_j (r_j-\varepsilon_j)\varepsilon_j \leq C\varepsilon_j|\log \varepsilon_j|r_j.
		$$
		Therefore, the sequence $\zeta_j \chi_{\Omega_{\varepsilon_j}(\mu_j)}/\varepsilon_j|\log \varepsilon_j|$ converges to zero in $L^1$ and it is bounded in $L^2$ on $\Omega$. This gives that
		\begin{equation}\label{108.5}
		\chi_{\Omega_{\varepsilon_j}(\mu_j)}\frac{\zeta_j }{\varepsilon_j|\log \varepsilon_j|} \rightharpoonup 0 \qquad \mbox{in $L^2(\Omega;\R^{2 \times 2 })$.}
		\end{equation}
		Next, we claim that
		\begin{equation}\label{109}
			\frac{1}{\varepsilon_j|\log \varepsilon_j|} \nu_j + \mu \xrightharpoonup{*} 0 \qquad \mbox{in $L^\infty(\Omega;\R^2)$}.
		\end{equation}
		To see this, recall that $\curl \zeta_{i,j}=0$ outside $\{x_{i,j}\} \cup \left( B_{r_j}(x_{i,j}) \setminus B_{r_j/2}(x_{i,j}) \right)$, hence, by \eqref{98}, we estimate
		\begin{equation}\label{110}
		|\curl \zeta_{i,j}(x)| \leq |\nabla \zeta_{i,j}(x)| \leq |\nabla f(x)||\zeta_{i,j}(x)|+|\nabla \zeta_{i,j}(x)||f(x)| \leq \frac{C}{r_j^2} \varepsilon_j \leq C \varepsilon_j |\log \varepsilon_j|.
		\end{equation}
		Thus, since the sequence is uniformly bounded in $L^\infty$, to prove \eqref{109} it is enough to show that for every Lipschitz function $g \in L^1(\Omega;\R^2)$ it holds
		\begin{equation*}
			\int_\Omega \left( \frac{1}{\varepsilon_j|\log \varepsilon_j|} \nu_j + \mu \right) g \, dx = \frac{1}{\varepsilon_j|\log \varepsilon_j|}  \int_\Omega \left( \nu_j + \mu_j \right) g \, dx+\int_\Omega \left( -\frac{1}{\varepsilon_j|\log \varepsilon_j|} \mu_j + \mu \right) g \, dx \to 0.
		\end{equation*}
		The second term converges to zero by \eqref{103}. For the first term, recalling \eqref{105} and \eqref{110} we have
		\begin{align*}
			\left| \int_\Omega \left( \nu_j + \mu_j \right) g \, dx \right| & = \left| \sum_{i=1}^{M_j} \int_{B_{r_j}(x_{i,j}) \setminus B_{\varepsilon_j}(x_{i,j})} \curl \zeta_{i,j}(x)g(x) \, dx+ \varepsilon_j \xi_{i,j} g(x_{i,j}) \right| \\
			& = \left| \sum_{i=1}^{M_j} \int_{B_{r_j}(x_{i,j}) \setminus B_{\varepsilon_j}(x_{i,j})} \curl \zeta_{i,j}(x)\left(g(x) -g(x_{i,j}) \right)\, dx \right| \\
			& \leq \vphantom{\int} C \sum_{i=1}^{M_j} \varepsilon_j |\xi_{i,j}| r_j \lip \, g \leq C r_j \lip \, g |\mu_j|(\Omega).
		\end{align*}
		The claim follows recalling that $r_j \to 0$ as $j \to +\infty$. By elliptic regularity, \eqref{107} and \eqref{109}, we infer that 
		$w_j/\varepsilon_j|\log \varepsilon_j| \rightharpoonup 0$ in $W^{2,p}(\Omega;\R^2)$ for every $p < \infty$. Hence, by Sobolev compact Embedding Theorem, we deduce
		\begin{equation}\label{111}
			\frac{\tilde{\beta}_j}{\varepsilon_j |\log \varepsilon_j|} \to 0 \qquad \mbox{uniformly on $\Omega$}.
		\end{equation}
		Together with \eqref{108.5} this gives \eqref{108}.
		
		We now prove the upper bound
		\begin{equation*}
			\limsup_{j \to +\infty} \mathcal{E}_{\varepsilon_j}(\mu_j,\beta_j) \leq \mathcal{E}(\mu,\beta,U).
		\end{equation*}
		Fix $s \in (0,1)$, we split the energy as follows
		\begin{align*}
		\mathcal{E}_{\varepsilon_j}(\mu_j,\beta_j)= \frac{1}{\varepsilon_j^2 |\log \varepsilon_j|^2} \int_{\Omega_{\varepsilon_j^s}(\mu_j)} W(\beta_j) \, dx & + \frac{1}{\varepsilon_j^2 |\log \varepsilon_j|^2} \sum_{i=1}^{M_j} \int_{B_{\varepsilon_j^s}(x_{i,j})} W(\beta_j) \, dx \\
		& \ + \frac{\eta_{\varepsilon_j}^2}{\varepsilon_j^2 |\log \varepsilon_j|^2} \int_{\Omega_{\varepsilon_j}(\mu_j)} |\divr \beta_j|^2 \, dx =: I^1_j+I^2_j+I^3_j.
		\end{align*}
		In order to deal with $I^1_j$ we perform a linearisation near $U$:
		\begin{align*}
			I^1_j & = \frac{1}{\varepsilon_j^2 |\log \varepsilon_j|^2} \int_{\Omega_{\varepsilon_j^s}(\mu_j)} W(U+\varepsilon_j|\log \varepsilon_j| \beta +\zeta_j+\tilde{\beta}_j) \, dx \\
			& = \frac{1}{\varepsilon_j^2 |\log \varepsilon_j|^2} \frac{1}{2} \int_{\Omega_{\varepsilon_j^s}(\mu_j)} \mathbb{C}_U (\varepsilon_j|\log \varepsilon_j| \beta +\zeta_j+\tilde{\beta}_j):(\varepsilon_j|\log \varepsilon_j| \beta +\zeta_j+\tilde{\beta}_j) \, dx \\
			& \hphantom{=} + \frac{1}{\varepsilon_j^2 |\log \varepsilon_j|^2} \int_{\Omega_{\varepsilon_j^s}(\mu_j)} \sigma (\varepsilon_j|\log \varepsilon_j| \beta +\zeta_j+\tilde{\beta}_j) \, dx, 
		\end{align*}
		where $\sigma(F)/|F|^2 \to 0$ as $|F| \to 0$.
		We claim that
		\begin{equation}\label{112}
			\limsup_{j \to +\infty} \frac{1}{\varepsilon_j^2 |\log \varepsilon_j|^2} \frac{1}{2} \int_{\Omega_{\varepsilon_j^s}(\mu_j)} \mathbb{C}_U (\varepsilon_j|\log \varepsilon_j| \beta +\zeta_j+\tilde{\beta}_j):(\varepsilon_j|\log \varepsilon_j| \beta +\zeta_j+\tilde{\beta}_j) \, dx \leq \mathcal{E}(\mu,\beta,U).
		\end{equation}
		First observe that thanks to \eqref{108.5} and \eqref{111} the mixed terms and the quadratic term involving $\tilde{\beta}_j$ vanish in the limit.
		Next, we trivially have
		\begin{equation}\label{113}
			\lim_{j \to +\infty} \frac{1}{\varepsilon_j^2 |\log \varepsilon_j|^2} \frac{1}{2} \int_{\Omega_{\varepsilon_j^s}(\mu_j)} \mathbb{C}_U\varepsilon_j|\log \varepsilon_j| \beta:\varepsilon_j|\log \varepsilon_j| \beta \, dx = \frac{1}{2} \int_\Omega \mathbb{C}_U \beta : \beta \, dx.
		\end{equation}
		Now, observe that the $L^2$-norm of $\zeta_j/\varepsilon_j |\log \varepsilon_j|$ is concentrated outside $\Omega_{\rho_{\varepsilon_j}}(\mu_j)$, indeed, recalling \eqref{97} and the definition of $\zeta_j$ we estimate for $j \to +\infty$
		$$
		\frac{1}{\varepsilon_j^2 |\log \varepsilon_j|^2} \int_{\Omega_{\rho_{\varepsilon_j}}(\mu_j)} |\zeta_j|^2 \leq \frac{1}{\varepsilon_j^2 |\log \varepsilon_j|^2} \sum_{i=1}^{M_j} \int_{\rho_{\varepsilon_j}}^{r_j} \varepsilon_j^{2-2\gamma_j} \frac{\varepsilon_j^{2\gamma_j}}{r^2} r  \, dr =\frac{1}{ |\log \varepsilon_j|^2} \sum_{i=1}^{M_j} (\log r_j-\log \rho_{\varepsilon_j}) \to 0.
		$$
		As $\gamma_j \to 1$ we have that $\varepsilon_j^{\gamma_j} < \varepsilon_j^s$ and $\rho_{\varepsilon_j} < r_j/2$ for $j$ large enough, in turn, recalling the definition of $\zeta_{i,j}$, \eqref{96} and the expression of $\beta_{\R^2}(\xi_{i,j})$ in polar coordinates, we obtain
		\begin{align*}
		\frac{1}{\varepsilon_j^2 |\log \varepsilon_j|^2} \frac{1}{2} \int_{\Omega_{\varepsilon_j^s}(\mu_j)} \mathbb{C}_U \zeta_j: \zeta_j \, dx & =  \frac{1}{\varepsilon_j^2 |\log \varepsilon_j|^2} \frac{1}{2} \int_{\Omega_{\varepsilon_j^s}(\mu_j) \setminus \Omega_{\rho_{\varepsilon_j}}(\mu_j)} \mathbb{C}_U \zeta_j: \zeta_j \, dx +o(1)\\
		& =\frac{1}{|\log \varepsilon_j|^2} \sum_{i=1}^{M_j} \frac{1}{2} \int_{B_{\rho_{\varepsilon_j}}(x_{i,j}) \setminus B_{\varepsilon_j^s}(x_{i,j})} \mathbb{C}_U \beta_{\R^2}(\xi_{i,j}) :  \beta_{\R^2}(\xi_{i,j}) \, dx + o(1).
		\end{align*}
		For $i=1,\dots,M_j$, by \eqref{psieps} we have
		$$
		\frac{1}{|\log \varepsilon_j|} \frac{1}{2} \int_{B_{\rho_{\varepsilon_j}}(x_{i,j}) \setminus B_{\varepsilon_j^s}(x_{i,j})} \mathbb{C}_U \beta_{\R^2}(\xi_{i,j}) :  \beta_{\R^2}(\xi_{i,j}) \, dx \leq \psi_{\varepsilon} (\xi_{i,j},U)(1 + o(1)),
		$$
		hence, for $j \to +\infty$ 
		\begin{equation}\label{114}
			\frac{1}{\varepsilon_j^2 |\log \varepsilon_j|^2} \frac{1}{2} \int_{\Omega_{\varepsilon_j^s}(\mu_j)} \mathbb{C}_U \zeta_j: \zeta_j \, dx \leq \frac{1}{|\log \varepsilon_j|} \sum_{i=1}^{M_j} \psi_{\varepsilon_j}(\xi_{i,j},U)+o(1).
		\end{equation}
		Recalling \eqref{101} and \eqref{102}
		\begin{align}\label{115}
			\begin{split}
			\lim_{j \to +\infty} \frac{1}{|\log \varepsilon_j|} \sum_{i=1}^{M_j} \psi_{\varepsilon_j}(\xi_{i,j},U) & = \lim_{j \to +\infty} \frac{1}{|\log \varepsilon_j|} \sum_{k=1}^{M} |\mu_j^k|(\Omega) \psi_{\varepsilon_j}(\xi_{i,j},U) \\
			& = \mathcal{L}^2(E) \sum_{k=1}^{M} \lambda_k \hat{\psi}(\xi_k,U)=\int_E \varphi(U,\xi) \, dx.
			\end{split}
		\end{align}
		Therefore, combining \eqref{114}, \eqref{115} and recalling that $\varphi$ is nonnegative, we get
		\begin{equation}\label{116}
			\lim_{j \to +\infty} \frac{1}{\varepsilon_j^2 |\log \varepsilon_j|^2} \frac{1}{2} \int_{\Omega_{\varepsilon_j^s}(\mu_j)} \mathbb{C}_U \zeta_j: \zeta_j \, dx \leq \int_E \varphi(U,\xi) \, dx \leq \int_\Omega \varphi(U,\xi) \, dx,
		\end{equation}
		thus, the claim \eqref{112} is proven. It remains to prove that the remainder term goes to zero in the limit. We start observing that by definition of $\zeta_j$ and \eqref{97}, for every $x \in \Omega_{\varepsilon_j^s}(\mu_j)$ it holds $|\zeta_j(x)| \leq C \varepsilon_j^{1-s}$. Hence, \eqref{111} and the fact that $\beta \in L^\infty(\Omega;\R^{2 \times 2})$ give
		\begin{equation}\label{117}
			|\varepsilon_j |\log \varepsilon_j| \beta + \zeta_j + \tilde{\beta}_j| \leq C (\varepsilon_j|\log \varepsilon_j|+\varepsilon_j^{1-s}).
		\end{equation}
		Setting $\omega(t):=\sup_{|F|\leq t} |\sigma(F)|$, we estimate
		\begin{align*}
			\int_{\Omega_{\varepsilon_j^s}(\mu_j)} \! \frac{\sigma\left(\varepsilon_j|\log \varepsilon_j| \beta +\zeta_j+\tilde{\beta}_j\right)}{\varepsilon_j^2 |\log \varepsilon_j|^2} \, dx 
			\leq \int_\Omega \chi_{\Omega_{\varepsilon_j^s}(\mu_j)} \frac{\omega\left(\left|\varepsilon_j|\log \varepsilon_j| \beta +\zeta_j+\tilde{\beta}_j\right|\right)}{\left|\varepsilon_j|\log \varepsilon_j| \beta +\zeta_j+\tilde{\beta}_j\right|^2} \, \frac{\left|\varepsilon_j|\log \varepsilon_j| \beta +\zeta_j+\tilde{\beta}_j\right|^2}{\varepsilon_j^2 |\log \varepsilon_j|^2} \, dx,
		\end{align*}
		and the limit as $j \to +\infty$ is zero, since the integrand in the expression above is the product of a sequence converging uniformly to zero (thanks to \eqref{117}) and a sequence bounded in $L^1$ (thanks to \eqref{108.5} and \eqref{111}). Thus, 
		\begin{equation}\label{118}
			\limsup_{j \to +\infty} I^1_j \leq \mathcal{E}(\mu,\beta,U).
		\end{equation}
	
		We now deal with $I^2_j$. Using the quadratic upper bound \eqref{W4} on $W$, we have
		$$
		I^2_j \leq \frac{1}{\varepsilon_j^2 |\log \varepsilon_j|^2} \sum_{i=1}^{M_j} \int_{B_{\varepsilon_j^s}(x_{i,j}) \setminus B_{\varepsilon_j}(x_{i,j})} C \left| \varepsilon_j |\log \varepsilon_j| \beta +\eta_j +\tilde{\beta}_j \right|^2 \, dx.
		$$
		Arguing as before, we have only to consider the quadratic terms involving $\beta$ and $\zeta_j$, as the others vanish in the limit. Observe that 
		$$
		\sum_{i=1}^{M_j} \int_{B_{\varepsilon_j^s}(x_{i,j}) \setminus B_{\varepsilon_j}(x_{i,j})} |\beta|^2 \, dx \leq \Vert \beta \Vert_{L^\infty(\Omega;\R^{2 \times 2})} M_j (\varepsilon_j^{2s}-\varepsilon_j) \to 0 \qquad \mbox{as $j \to +\infty$},
		$$
		and
		$$
		\frac{1}{\varepsilon_j^2 |\log \varepsilon_j|^2} \sum_{i=1}^{M_j} \int_{B_{\varepsilon_j^s}(x_{i,j}) \setminus B_{\varepsilon_j}(x_{i,j})} |\zeta_j|^2 \, dx \leq \frac{1}{|\log \varepsilon_j|^2} C M_j \int_{\varepsilon_j}^{\varepsilon_j^s} \frac{1}{r} \, dr \leq C(1-s).
		$$
		where the constant $C$ does not depends on $s$. Therefore,
		\begin{equation}\label{119}
			\limsup_{j \to +\infty} I^2_j \leq C(1-s).
		\end{equation}
		Finally, we deal with $I^3_j$. First, observe that by definition of $\tilde{\beta}_j$ we have $\divr \beta_j=\varepsilon_j |\log \varepsilon_j| \divr \beta+\divr \zeta_j$. Hence, recalling that $\eta_{\varepsilon_j}=\varepsilon_j|\log \varepsilon_j|/\rho_{\varepsilon_j}$,
		$$
		I_j^3 \leq \frac{4\varepsilon_j^2 |\log \varepsilon_j|^2}{\rho_{\varepsilon_j}^2} \int_\Omega |\divr \beta|^2 \, dx+\frac{4}{\rho_{\varepsilon_j}^2} \int_{\Omega_{\varepsilon_j}(\mu_j)} |\divr \zeta_j|^2 \, dx.
		$$
		As $\divr \beta \in L^2(\Omega;\R^2)$ the first term trivially vanishes in the limit. For the second term, recalling \eqref{95}, \eqref{97}, \eqref{98} and the definition of $\zeta_j$, we estimate
		\begin{align*}
			\frac{1}{\rho_{\varepsilon_j}^2} \int_{\Omega_{\varepsilon_j}(\mu_j)} |\divr \zeta_j|^2 \, dx & \leq \frac{C}{\rho_{\varepsilon_j}^2} \sum_{i=1}^{M_j} \left(\int_{B_{r_j}(x_{i,j}) \setminus B_{r_j/2}(x_{i,j})} \frac{\varepsilon_j^2}{r_j^4}  + \int_{B_{r_j/2}(x_{i,j}) \setminus B_{\varepsilon_j^{\gamma_j}/2}(x_{i,j})} |\divr \zeta_{i,j}|^2 \right) \\
			& \leq \frac{C \varepsilon_j^2 |\log \varepsilon_j|}{\rho_{\varepsilon_j}^2 r_j^4}+ \frac{C}{\rho_{\varepsilon_j}^2} M_j \varepsilon_j^2 \int_{\varepsilon_j^{\gamma_j}/2}^{r_j/2} \frac{1}{r^3} \, dr \\
			&  \leq \frac{C \varepsilon_j^2 |\log \varepsilon_j|}{\rho_{\varepsilon_j}^2} \left( \frac{4}{\varepsilon_j^{2\gamma_j}}-\frac{4}{r_j^2} \right) \leq C \varepsilon_j^{2-2\gamma_j} \frac{|\log \varepsilon_j|}{\rho_{\varepsilon_j}^2}.
		\end{align*}
		We now fix the choice for the sequence $\gamma_j$, namely
		$$
		\gamma_j := 1- \frac{|\log \rho_{\varepsilon_j}|}{|\log \varepsilon_j|} - \frac{1}{2\sqrt{|\log \varepsilon_j|}}.
		$$
		Notice that for this choice we have $\gamma_j \in (0,1)$ for $j $ large enough, $\gamma_j \to 1$ and $\varepsilon_j^{2 -2\gamma_j}=\rho^2_{\varepsilon_j} e^{-\sqrt{|\log \varepsilon_j|}}$, hence
		\begin{equation}\label{120}
			\limsup_{j \to +\infty} I_j^3 =0.
		\end{equation}
		Thus, combining \eqref{118}--\eqref{120}, we obtain for every $s \in (0,1)$ the final estimate
		$$
		\limsup_{j \to +\infty} \mathcal{E}_{\varepsilon_j} (\mu_j,\beta_j) \leq \mathcal{E}(\mu,\beta,U)+C(1-s),
		$$
		where the constant $C$ does not depend on $s$; hence, sending $s \to 1$ we conclude. 
		
		For the sake of the next step, we observe that \eqref{113}, \eqref{116} and the arguments above entails that, for every open set $E \subseteq \Omega$
		\begin{equation}\label{121}
			\limsup_{j \to +\infty} \mathcal{E}_{\varepsilon_j}(\mu_j,\beta_j;E) \leq \mathcal{E}(\mu,\beta,U;E).
		\end{equation}
		Moreover, for any sequence $(\alpha_j) \subset L^2(E;\R^{2 \times 2})$ such that for every $j \geq 1$ it holds $\curl \alpha_j =0$ and $\divr \alpha_j =0$ on $E$ and such that $\alpha_j/\varepsilon_j |\log \varepsilon_j| \to 0$ uniformly in $E$, the same arguments give that
		\begin{equation}\label{122}
			\limsup_{j \to +\infty} \mathcal{E}_{\varepsilon_j}(\mu_j,\beta_j+\alpha_j;E) \leq \mathcal{E}(\mu,\beta,U;E).
		\end{equation}
	
		\noindent \textbf{Step 2.} Let now $\beta \in L^\infty(\Omega;\R^{2 \times 2})$, $\divr \beta \in L^2(\Omega;\R^2)$ and $\mu$ of the type
		$$
		\mu=\sum_{l=1}^L \mu_l, \qquad \mbox{with} \qquad \mu_l= \xi_l \chi_{E_l} \, dx,
		$$
		where $\xi_l \in \R^2$ and $\{ E_l \}_{l=1}^L$ is a partition of $\Omega$, that is $\mathcal{L}^2 \left( \Omega \setminus \cup_{l=1}^L E_l \right)=0$ and $E_l$ is a simply connected open set with Lipschitz boundary for every $l=1,\dots,L$.
		
		Let $\mu_j^l$, $\zeta_j^l$ and $\tilde{\beta}_j^l$ be the sequences corresponding to $\mu_l$ constructed in \textbf{Step 1}. For every $l=1,\dots,L$ we have
		\begin{align}
			& \supp \, \zeta_j^l \subset E_l, \label{123} \\
			& \frac{\tilde{\beta}_j^l}{\varepsilon_j |\log \varepsilon_j|} \to 0 \qquad \mbox{uniformly in $\Omega$}, \label{124} \\
			& \chi_{\Omega_{\varepsilon_j}(\mu_j^l)} \frac{\zeta_j^l}{\varepsilon_j |\log \varepsilon_j|} \rightharpoonup 0 \qquad \mbox{weakly in $L^2(\Omega;\R^{2 \times 2})$}, \label{125} \\
			& \curl \tilde{\beta}_j^l = \divr \tilde{\beta}_j^l =0 \qquad \mbox{on $\Omega \setminus E_l$}. \label{126}
		\end{align}
		Set $\mu_j:= \sum_{l=1}^L \mu_j^l$, $\chi_{\Omega_{\varepsilon_j}(\mu_j)}:=\Pi_{l=1}^L \chi_{\Omega_{\varepsilon_j}(\mu_j^l)}$ and define
		$$
		\beta_j:= I + \chi_{\Omega_{\varepsilon_j}(\mu_j)} \left(U-I +\varepsilon_j |\log \varepsilon_j| \beta + \sum_{l=1}^L \left( \zeta_j^l + \tilde{\beta}_j^l \right) \right).
		$$
		Using \eqref{123} and \eqref{126} and arguing as in \textbf{Step 1}, we have that $\beta_j \in \mathcal{AS}_{\varepsilon_j}(\mu_j)$. Moreover, using \eqref{123}--\eqref{125} we infer that $(\mu_j,\beta_j) $ converges to $(\mu,\beta,U)$ in the sense of Definition \ref{def: convergence}.
		For every $l=1,\dots,L$, \eqref{121}, \eqref{122}, \eqref{125} and \eqref{126} gives that
		$$
		\limsup_{j \to +\infty} \mathcal{E}_{\varepsilon_j}(\mu_j,\beta_j;E_l) \leq \mathcal{E} (\mu,\beta,U;E_l).
		$$
		Finally, recalling that $\mathcal{L}^2(\Omega \setminus \cup_{l=1}^L E_l)=0$ and $|\mu|(\Omega \setminus \cup_{l=1}^L E_l)=0$, we conclude the upper bound
		$$
		\limsup_{j \to +\infty} \mathcal{E}_{\varepsilon_j}(\mu_j,\beta_j) \leq \mathcal{E} (\mu,\beta,U).
		$$
		
		\noindent \textbf{Step 3.} We finally consider a general $\mu \in (H^{-1}(\Omega;\R^2) \cap \mathcal{M}(\Omega;\R^2))$ and $\beta \in  L^2(\Omega;\R^{2 \times 2})$. 
		
		By approximation using Lemma \ref{lem: density} we can reduce to the case of locally constant measure, that is \textbf{Step 2} and then use a diagonal argument. Indeed, notice that \eqref{88} gives that 
		$$
		\lim_{m \to +\infty} \mathcal{E}(\mu_m,\beta_m,U)=\mathcal{E}(\mu,\beta).
		$$
		This concludes the proof of the upper bound.
	\end{proof}

	\appendix 
	\section{}
	
	We show that the constants appearing in Theorem \ref{thm: n=2 BB} and in the rigidity estimates: Theorem \ref{thm: FJM rigidity} and Theorem \ref{thm: CDM rigidity}, are uniform for the shrinked Lipschitz domains obtained through Lemma \ref{lem: shrinking lip}. 
	Here $\Omega \subset \R^2$ is a fixed bounded, simply connected and Lipschitz domain. We recall here the statement of Lemma \ref{lem: shrinking lip}.
	\begin{lem}
		For every $\delta >0$ small enough, there exists a bounded Lipschitz domain $\Omega^\delta \subset \R^2$ such that:
		\begin{itemize}
			\item[\textnormal{(i)}] the set $\Omega^\delta$ is simply connected and there exists $c>0$ not depending on $\delta$ such that 
			$$\{ x \in \Omega \colon \dist(x,\partial \Omega) > \delta \} \subset \Omega^\delta \Subset \Omega^\delta+B_{c\delta} \subset \Omega;$$
			\item[\textnormal{(ii)}] given a set $E \subset \Omega^\delta$ of finite perimeter, there exists $C(\Omega)>0$ such that 
			$$\min \left\{ \mathcal{L}^2(E) ,\mathcal{L}^2(\Omega \setminus E)\right\} \leq C(\Omega) P\left(E,\Omega^\delta\right)^2;$$
			\item[\textnormal{(iii)}] given $p \in (1,\infty)$, for every $f \in W^{1,p}_{\textnormal{loc}}(\Omega^\delta;\R^{2 \times 2})$ there exists a constant $C'(\Omega,p)>0$ such that
			$$
			\min_{F \in \R^{2 \times 2}} \int_{\Omega^\delta} |f-F|^p \, dx \leq C'(\Omega,p) \int_{\Omega^\delta} |\nabla f|^p \dist^p\left(x,\partial \Omega^\delta\right) \, dx.
			$$
		\end{itemize}
	\end{lem}

	\begin{rem}\label{App: rem}
		Recalling the proof of Lemma \ref{lem: shrinking lip}, for every $\delta>0$ small there exists a diffeomorphism $\Phi_\delta \in C^\infty(\R^2,\R^2)$ such that $\Phi_\delta(\Omega)=\Omega^\delta$ and 
		\begin{equation*}
			\sup_{x \in \Omega} \left(|D^k \Phi_\delta-D^k \textnormal{Id}|+|D^k \Phi_\delta^{-1}-D^k \textnormal{Id}| \right) \leq C \delta,
		\end{equation*}
		for every $k \in \N$ and for some $C=C(\Omega)>0$, where $D^k$ is the Fréchet derivative of order $k$.
	\end{rem}
	
	We start with the following.
	
	\begin{lem}\label{App: lem for BBV}
		The constant appearing in Theorem \ref{thm: n=2 BB} is uniform for every $\Omega^\delta$ when $\delta$ is small enough. That is, there exists $C(\Omega)>0$ such that for every $\delta>0$ small enough and every $f \in L^1(\Omega^\delta;\R^2)$ vector field with $\divr f \in H^{-2}(\Omega^\delta)$ we have
		\begin{equation}\label{a1}
		\Vert f \Vert_{H^{-1}(\Omega^\delta;\R^2)} \leq C(\Omega) \left( \Vert \divr f\Vert_{H^{-2}(\Omega^\delta)}+\Vert f \Vert_{L^1(\Omega^\delta;\R^2)} \right).
		\end{equation}
	\end{lem}

	We will use the following Theorem for Lipschitz sets contained in \cite[Theorem 3']{BBJAMS}.

	\begin{thm}\label{thm: app BB}
		Let $E$ be a bounded Lipschitz domain. For every $f \in L^2(E)$ with $\int_E f \, dx =0$, there exists $Y \in L^\infty(E;\R^2) \cap H^1_0(E;\R^2)$ such that $\divr Y = f$ on $E$ and
		$$
		\Vert Y \Vert_{L^\infty(E;\R^2)} + \Vert Y \Vert_{H^1(E;\R^2)} \leq C(E) \Vert f \Vert_{L^2(E)}.
		$$
		The constant $C(E)>0$ depends only on the domain $E$ and not on $f$.
	\end{thm}

	\begin{proof}[Proof of Lemma \ref{App: lem for BBV}]
		We claim that for every $\varphi \in H^1_0(\Omega^\delta;\R^2)$ there exist $\psi \in L^\infty(\Omega^\delta;\R^2) \cap H^1_0(\Omega^\delta;\R^2)$ and $\eta \in H^2_0(\Omega^\delta)$ such that $\varphi=\psi+\nabla \eta$ on $\Omega^\delta$ and
		\begin{equation}\label{a2}
			\Vert \psi \Vert_{L^\infty(\Omega^\delta;\R^2)}+\Vert \psi \Vert_{H^1(\Omega^\delta;\R^2)}+\Vert \eta \Vert_{H^2(\Omega^\delta)} \leq C(\Omega) \Vert \varphi \Vert_{H^1(\Omega^\delta;\R^2)}.
		\end{equation}
		Notice that once we have \eqref{a2} given $f \in L^1(\Omega^\delta;\R^2)$ with $\divr f \in H^{-2}(\Omega^\delta)$, for every $\varphi \in L^\infty(\Omega^\delta;\R^2) \cap H^1_0(\Omega^\delta;\R^2)$ we estimate
		\begin{align*}
			\int_{\Omega^\delta} f \varphi \, dx & = \int_{\Omega^\delta} f (\psi +\nabla \eta) \, dx \leq \Vert f \Vert_{L^1(\Omega^\delta;\R^2)} \Vert \psi \Vert_{L^\infty(\Omega^\delta;\R^2)}+\Vert \divr f \Vert_{H^{-2}(\Omega^\delta)} \Vert \eta \Vert_{H^2(\Omega^\delta)} \\
			& \leq C(\Omega) \left( \Vert \divr f\Vert_{H^{-2}(\Omega^\delta)}+\Vert f \Vert_{L^1(\Omega^\delta;\R^2)} \right) \Vert \varphi \Vert_{H^1(\Omega^\delta;\R^2)},
		\end{align*}
		which gives \eqref{a1}.
		
		We now prove the claim. First, let $\varphi \in H^1_0(\Omega;\R^2)$. Since $\varphi$ has zero trace on $\partial \Omega$, we have that $\int_\Omega \curl \varphi =0$. Hence, we can apply Theorem \ref{thm: app BB} finding $Y \in L^\infty(\Omega;\R^2) \cap H^1_0(\Omega;\R^2)$ such that $\divr Y = \curl \varphi$ on $\Omega$ and 
		\begin{equation}\label{a3}
			\Vert Y \Vert_{L^\infty(\Omega;\R^2)} + \Vert Y \Vert_{H^1(\Omega;\R^2)} \leq C(\Omega) \Vert \curl \varphi \Vert_{L^2(\Omega)} \leq C(\Omega) \Vert \varphi \Vert_{H^1(\Omega;\R^2)}.
		\end{equation}
		By setting $\psi:=JY$, where $J$ is the anticlockwise rotation of $\pi/2$, we have that $\curl \psi = \divr Y = \curl \varphi$ and for $\psi$ holds the same estimate in \eqref{a3}. Hence, since $\Omega$ is simply connected we can set $\nabla \eta:=\varphi-\psi$ and $\eta \in H^2_0(\Omega)$. Moreover, by \eqref{a3}
		$$
		\Vert \nabla^2 \eta \Vert_{L^2(\Omega;\R^{2 \times 2})} \leq \Vert \nabla \varphi \Vert_{L^2(\Omega;\R^{2 \times 2})}+\Vert \nabla \psi \Vert_{L^2(\Omega;\R^{2 \times 2})} \leq C(\Omega) \Vert \varphi \Vert_{H^1(\Omega;\R^2)},
		$$
		which gives, as $\eta \in H^2_0(\Omega;\R^2)$, $\Vert \eta \Vert_{H^2(\Omega)} \leq C(\Omega) \Vert \varphi \Vert_{H^1(\Omega;\R^2)}$.
		
		Let now $\varphi \in H^1_0(\Omega^\delta;\R^2)$. Define on $\Omega$ the function $\tilde{\varphi} := (D\Phi_\delta)^T ( \varphi \circ \Phi_\delta )$, notice that by Remark \ref{App: rem} we have $\tilde{\varphi} \in H^1_0(\Omega;\R^2)$ and for $\delta>0$ small enough 
		$\Vert \tilde{\varphi} \Vert_{H^1(\Omega;\R^2)} \leq 4 \Vert \varphi \Vert_{H^1(\Omega^\delta;\R^2)}$. In view of the previous considerations, there exist $\tilde{\psi} \in L^\infty(\Omega;\R^2) \cap H^1_0(\Omega;\R^2)$ and $\tilde{\eta} \in H^2_0(\Omega)$ such that $\tilde{\varphi}=\tilde{\psi}+\nabla \tilde{\eta}$ on $\Omega$ and
		$$
		\Vert \tilde{\psi} \Vert_{L^\infty(\Omega;\R^2)}+\Vert \tilde{\psi} \Vert_{H^1(\Omega;\R^2)}+\Vert \tilde{\eta} \Vert_{H^2(\Omega)} \leq C(\Omega) \Vert \tilde{\varphi} \Vert_{H^1(\Omega;\R^2)}.
		$$
		Set $\psi:= (D\Phi_\delta^{-1})^T (\tilde{\psi} \circ \Phi_\delta^{-1})$ and $\eta := \tilde{\eta} \circ \Phi_\delta^{-1}$ on $\Omega^\delta$. By definition and Remark \ref{App: rem} we have $\psi \in L^\infty(\Omega^\delta;\R^2) \cap H^1_0(\Omega^\delta;\R^2)$, $\eta \in H^2(\Omega^\delta)$ and $\varphi=\psi+\nabla \eta$. Moreover, using again Remark \ref{App: rem}, for every $\delta>0$ small enough we estimate
		$$
		\Vert \psi \Vert_{L^\infty(\Omega^\delta;\R^2)} \leq 2 \Vert \tilde{\psi} \Vert_{L^\infty(\Omega;\R^2)}, \qquad \Vert \psi \Vert_{H^1(\Omega^\delta;\R^2)} \leq 4 \Vert \tilde{\psi} \Vert_{H^1(\Omega;\R^2)} \qquad \mbox{and} \qquad \Vert \eta \Vert_{H^2(\Omega^\delta)} \leq 4 \Vert \tilde{\eta} \Vert_{H^2(\Omega)}.
		$$
		Thus, the claim is proven.
	\end{proof}

	We now deal with the constant appearing in Theorem \ref{thm: FJM rigidity}.
	
	\begin{lem}\label{App: FJM}
		The constant appearing in Theorem \ref{thm: FJM rigidity} is uniform for every $\Omega^\delta$ when $\delta$ is small enough. That is, given any $p \in (1,\infty)$, there exists a constant $C(\Omega,p)>0$ with the following property: for every $\delta>0$ small enough and every $u \in W^{1,p}(\Omega^\delta;\R^2)$ there exists an associated rotation $R \in SO(2)$ such that
		\begin{equation}\label{a4}
			\Vert \nabla u -R \Vert_{L^p(\Omega^\delta;\R^{2 \times 2})} \leq C(\Omega) \Vert \dist(\nabla u,SO(2)) \Vert_{L^p(\Omega^\delta)}.
		\end{equation}
	\end{lem}

	\begin{proof}
		For the proof of this result we refer to \cite[Section 5]{ContiZwicknagl}, where the authors have proven that ultimately the constant appearing in \eqref{a4} depends on the constant of the weighted Poincaré inequality. Hence, recalling property (iii) for the sets $\Omega^\delta$ we conclude.
	\end{proof}

	\begin{rem}\label{App: Korn}
		We point out that the result contained in \cite[Section 5]{ContiZwicknagl} holds for a much more general class of Lipschitz domains which encompass the class of Lipschitz domains having smooth diffeomorphisms arbitrary close to the identity mapping one into the other. Moreover, in \cite{ContiZwicknagl} the authors prove also a Korn's inequality (see Theorem 5.10) with an uniform constant. In our setting we can rephrase it as follows: 
		
		\noindent Given $p \in (1,\infty)$, for every $\delta>0$ small enough and every $u \in W^{1,p}(\Omega^\delta;\R^2)$ there exists an associated matrix $S \in \R^{2 \times 2}_{\textnormal{skew}}$ such that
		\begin{equation}\label{a5}
			\Vert \nabla u - S \Vert_{L^p(\Omega^\delta;\R^{2 \times 2})} \leq C(\Omega) \Vert E u \Vert_{L^p(\Omega^\delta;\R^{2 \times 2})}.
		\end{equation}
	\end{rem}
		
	Finally we deal with the constant in Theorem \ref{thm: CDM rigidity} (we do it only for the case $p=2$ since it is the only one we need).
	
	\begin{prop}\label{App: CDM}
		The constant appearing in Theorem \ref{thm: CDM rigidity} is uniform for every $\Omega^\delta$ when $\delta$ is small enough. That is, there exists a constant $C(\Omega)>0$ with the following property: for every $\delta>0$ small enough and every $u \in W^{1,1}(\Omega^\delta;\R^2)$ with $\nabla u \in L^{2,\infty}(\Omega^\delta;\R^{ 2 \times 2})$, there exists an associated rotation $R \in SO(2)$ such that
		\begin{equation}\label{a6}
			\Vert \nabla u -R \Vert_{L^{2,\infty}(\Omega^\delta;\R^{2 \times 2})} \leq C(\Omega) \Vert \dist(\nabla u,SO(2)) \Vert_{L^{2,\infty}(\Omega^\delta)}.
		\end{equation}
	\end{prop}

	In order to prove Proposition \ref{App: CDM} we need the following two lemmas. The first one deals with the uniformity of the constant appearing in the Korn inequality with mixed growth (see \cite[Theorem 2.1]{CDM}). The second one is a Lipschitz truncation lemma.
	
	\begin{lem}\label{App CDM korn}
		Let $1<p<q<\infty$. there exists $\bar{\delta}=\bar{\delta}(\Omega)>0$ such that for every $\delta \in (0,\bar{\delta})$ the following holds. Let $u \in W^{1,1}(\Omega^\delta;\R^2)$ such that
		\begin{equation}\label{a7}
			E u =\frac{1}{2}\left|\nabla u + \nabla u^T\right| = f + g \qquad \mathcal{L}^2-\mbox{a.e. in $\Omega^\delta$},
		\end{equation} 
		where $f \in L^p(\Omega^\delta;\R^{2 \times 2})$ and $g \in L^q(\Omega^\delta;\R^{2 \times 2})$. Then, there exists a constant $C=C(\Omega,p,q)>0$, matrix fields $F \in L^p(\Omega^\delta;\R^{2 \times 2 })$, $G \in L^q(\Omega^\delta;\R^{2 \times 2 })$ and a skew-symmetric matrix $S \in \R^{2 \times 2}_{\textnormal{skew}}$ such that
		\begin{equation}\label{a8}
			\nabla u = S + F + G \qquad \mathcal{L}^2-\mbox{a.e. in $\Omega^\delta$},
		\end{equation}
		and
		\begin{equation}\label{a9}
			\Vert F \Vert_{L^p(\Omega^\delta;\R^{2 \times 2})} \leq C \Vert f \Vert_{L^p(\Omega^\delta)}, \qquad \Vert G \Vert_{L^q(\Omega^\delta;\R^{2 \times 2})} \leq C \Vert g \Vert_{L^q(\Omega^\delta)}.
		\end{equation}
	\end{lem}
	
	\begin{proof}
		To prove that the constant does not depend on $\delta$, we need to go through the proof of \cite[Theorem 2.1]{CDM} and check the dependencies of the constants appearing there. We will only sketch the main points.
		
		Let $\delta_1 =\delta_1(\Omega)>0$ such that for every $\delta \in (0,\delta_1)$ the considerations in Remark \ref{App: Korn} hold and $1/2 \leq \Vert \det D\Phi_\delta \Vert_{L^\infty(\Omega)} \leq 2$. Notice that if $\Vert g \Vert_{L^q} \leq \Vert f \Vert_{L^p}$ then, by H\"older inequality we have 
		$$
		\Vert E u \Vert_{L^p} \leq \Vert f \Vert_{L^p}+\Vert g \Vert_{L^p} \leq \Vert f \Vert_{L^p} + \mathcal{L}^2(\Omega^\delta)^{\frac{q-p}{pq}} \Vert g \Vert_{L^q} \leq \left(1+2^{\frac{q-p}{pq}}\mathcal{L}^2(\Omega^\delta)^{\frac{q-p}{pq}}\right) \Vert f \Vert_{L^p}.
		$$
		Hence, using \eqref{a5} with $p$ we conclude the assertion with $F=\nabla u-S$ and $G=0$. For every $\delta \in (0,\delta_1)$ we may thus assume that $\Vert f \Vert_{L^p} \leq \Vert g \Vert_{L^q}$.
		
		\textit{Construction of a cover for the domains.} We now construct a finite cover with balls of the domain $\Omega$ and then prove that we can find a corresponding cover for the $\Omega^\delta$'s with the relevant quantities independent of $\delta$, at least when the latter is small enough. For every $x \in \Omega$ we take $r_x>0$ such that $B(x,r_x) \subset \Omega$. For every $x \in \partial \Omega$ we take $r_x$ with the following properties. There exists an orthonormal coordinate system $v_1,v_2$ in $\R^2$ and a Lipschitz function $\phi_x \colon \R \to \R$ such that $\phi_x(0)=0$ and
		$$
		B(x,r_x) \cap \partial \Omega = \left\{ x+\sum_{i=1}^2 \lambda_i v_i \colon \, (\lambda_1,\lambda_2) \in B(0,r_x), \ \lambda_2=\phi_x(\lambda_1) \right\},
		$$
		$$
		B(x,r_x) \cap \Omega = \left\{ x+\sum_{i=1}^2 \lambda_i v_i \colon \, (\lambda_1,\lambda_2) \in B(0,r_x), \ \lambda_2 < \phi_x(\lambda_1) \right\}.
		$$
		This is possible since $\Omega$ has Lipschitz boundary. Let $L>0$ be a uniform Lipschitz constant for all $\phi_x$ with $x \in \partial \Omega$ (this exists since $\partial \Omega$ is compact). Define $\gamma:=1/(2\sqrt{1+L^2})$. By construction $\{ B(x,\gamma r_x/2) \}_{x \in \overline{\Omega}}$ is an open cover of $\overline{\Omega}$. Since $\overline{\Omega}$ is compact we may choose a finite subcover $\{ B(x_j,\gamma r_j/2) \}_{j=0,\dots,M}$ of $\overline{\Omega}$. Moreover, the finitely many balls $B_j:=B(x_j,\gamma r_j/2)$ satisfy
		\begin{equation}\label{a10}
			\alpha:= \min \left\{ \mathcal{L}^2(B_l \cap B_k \cap \Omega) \colon \ l, \ k \in \{0,\dots,M\}, \ B_l \cap B_k \cap \Omega \neq \emptyset \right\} >0.
		\end{equation}
		Let $j \in \{0,\dots,M\}$ such that $x_j \in \Omega$, observe that for every $\delta>0$ small we have $\Phi_\delta(B_j) \subset \Omega^\delta$ by definition. Moreover, in virtue of Remark \eqref{App: rem}, we infer that $\Phi_\delta(B_j) \to B_j$ in $\R^2$ in the Hausdorff sense. Hence, there exists $r_j^\delta < r_j$ such that $B(\Phi_\delta(x_j),\gamma r_j^\delta/2) \subset \Phi_\delta(B_j)$ and $r_j^\delta \to r_j$ as $\delta \to 0$. Let now $j \in \{0,\dots,M\}$ such that $x \in \partial \Omega$. Using again Remark \ref{App: rem} we have that there exists $\delta_2=\delta_2(\Omega)>0$ such that for every $\delta \in (0,\delta_2)$ we can find an appropriate orthonormal coordinate system $v_1^\delta,v_2^\delta$ of $\R^2$ and a Lipschitz function $\phi_j^\delta \colon \R \to \R$ such that $\phi_j^\delta(0)=0$ and
		$$
		\Phi_\delta(B_j) \cap \partial \Omega^\delta = \left\{ \Phi_\delta(x_j)+\sum_{i=1}^2 \lambda_i v_i^\delta \in \Phi_\delta(B_j) \colon \, \lambda_2=\phi_j^\delta(\lambda_1) \right\},
		$$
		$$
		\Phi_\delta(B_j) \cap \Omega^\delta = \left\{ \Phi_\delta(x_j)+\sum_{i=1}^2 \lambda_i v_i^\delta \in \Phi_\delta(B_j) \colon \, \lambda_2 < \phi_j^\delta(\lambda_1) \right\}.
		$$
		Moreover, $\lip \phi_j^\delta \leq L_\delta$, where $L_\delta \searrow L$ as $\delta \to 0$. Arguing as before, we can find $r_j^\delta \leq r_j$ such that $B(\Phi_\delta(x_j),\gamma r_j^\delta/2) \subset \Phi_\delta(B_j)$. 
		
		Set $\gamma_\delta:= 1/(2\sqrt{1+L_\delta^2}) \leq \gamma$. We fix $\delta_3=\delta_3(\Omega) \leq \delta_2$ such that for every $\delta \in (0,\delta_3)$ and $j=0,\dots,M$ the following inequalities hold
		\begin{equation}\label{a11}
			\frac{r_j}{2} \leq r_j^\delta \leq r_j, \qquad L \leq L_\delta \leq 2L, \qquad \frac{\gamma}{2} \leq \gamma_\delta \leq \gamma.
		\end{equation}
		Set $x_j^\delta=\Phi_\delta(x_j)$ and $B_j^\delta:=B(x_j^\delta,\gamma_\delta r_j^\delta/2)$. We have that the family $\{ B_j^\delta \}_{j=0,\dots,M}$ consists of balls intersecting $\Omega^{\delta}$ and $B_j^\delta \to B_j$ in the Hausdorff sense, for every $j=0,\dots,M$. Therefore, there exists $\delta_4=\delta_4(\Omega) \leq \delta_3$ such that for every $\delta \in (0,\delta_4)$ the family $\{ B_j^\delta \}_{j=0,\dots,M}$ is a cover of $\Omega^\delta$. This can be seen by considering the family $\{\Phi_\delta^{-1}(B_j^\delta)\}_{j=0,\dots,M}$, and observing that for $\delta$ small enough it is a covering for $\Omega$. Finally, by continuity we infer the existence of $\delta_5 = \delta_5(\Omega) \leq \delta_4$ such that the quantity $\alpha_\delta$ defined as 
		$$
		\alpha_\delta:= \min \left\{ \mathcal{L}^2(B_l^\delta \cap B_k^\delta \cap \Omega^\delta) \colon \ l, \ k \in \{0,\dots,M\}, \ B_l^\delta \cap B_k^\delta \cap \Omega^\delta \neq \emptyset \right\} >0,
		$$
		satisfies the bounds $\alpha/2 \leq \alpha_\delta \leq \alpha$ for every $\delta \in (0,\delta_5)$.
		
		\textit{Interior estimate.} We fix $\delta \in (0,\delta_5)$. Let $x_j \in \Omega$. Arguing exactly as in the proof of \cite[Theorem 2.1]{CDM} and in view of \eqref{a11} we get that there exists $S_j \in \R^{2 \times 2}_{\textnormal{skew}}$ such that
		$$
		\nabla u = F+G+S_j \qquad \mathcal{L}^2-\mbox{a.e in $B(x_j^\delta,r_j^\delta/2)$},
		$$
		with
		\begin{align*}
		&\Vert F \Vert_{L^p(B(x_j^\delta,r_j^\delta/2))} \leq C(p,q,r_j)\Vert f \Vert_{L^p(B(x_j^\delta,r_j^\delta))}, \\
		&\Vert G \Vert_{L^q(B(x_j^\delta,r_j^\delta/2))} \leq C(p,q,r_j)\Vert g \Vert_{L^q(B(x_j^\delta,r_j^\delta))}+C(p,q,r_j)\Vert f \Vert_{L^p(B(x_j^\delta,r_j^\delta))}.
		\end{align*}
		
		\textit{Local estimate at the boundary.} Now fix $x^\delta_j \in \partial \Omega^\delta$. By Nitsche extension (see \cite[Theorem 5.1]{CDM} and \cite{NitscheKorn}, there exist $\tilde{u}$, $\tilde{f}$ and $\tilde{g}$ extension of $u$, $f$ and $g$, respectively, on $\Omega^\delta \cup B(x_j^\delta,\gamma_\delta r_j^\delta)$ with $E\tilde{u}=\tilde{f}+\tilde{g}$ $\mathcal{L}^2$-a.e. on $B(x_j^\delta,\gamma_\delta r_j^\delta)$. Moreover, these functions satisfy the following estimates
		\begin{align}\label{a12}
			\begin{split}
			& \Vert \tilde{f} \Vert_{L^p(B(x_j^\delta,\gamma_\delta r_j^\delta))} \leq C(p,q,r_j^\delta,L_\delta) \Vert f \Vert_{L^p(B(x_j^\delta, r_j^\delta)\cap \Omega^\delta)},\\
			& \Vert \tilde{g} \Vert_{L^q(B(x_j^\delta,\gamma_\delta r_j^\delta))}  \leq  C(p,q,r_j^\delta,L_\delta) \Vert g \Vert_{L^q(B(x_j^\delta, r_j^\delta)\cap \Omega^\delta)}.
			\end{split}
		\end{align}
		Hence, in view of \eqref{a11}, it is clear that the constant appearing in \eqref{a12} depends only on $p$, $q$, $r_j$ and $L$. As before, we infer the existence of $S_j$, $\tilde{F}$ and $\tilde{G}$ such that
		$$
		\nabla \tilde{u} = \tilde{F}+\tilde{G}+S_j \qquad \mathcal{L}^2-\mbox{a.e in $B(x_j^\delta,\gamma_\delta r_j^\delta/2)$},
		$$
		with
		\begin{align*}
			&\Vert \tilde{F} \Vert_{L^p(B(x_j^\delta,\gamma_\delta r_j^\delta/2))} \leq C(p,q,r_j,L)\Vert \tilde{f} \Vert_{L^p(B(x_j^\delta,\gamma_\delta r_j^\delta))} \leq C(p,q,r_j,L)\Vert f \Vert_{L^p(B(x_j^\delta,r_j^\delta) \cap \Omega^\delta)}, \\
			&\Vert \tilde{G} \Vert_{L^q(B(x_j^\delta,\gamma_\delta r_j^\delta/2))} \leq C(p,q,r_j,L)\Vert \tilde{g} \Vert_{L^q(B(x_j^\delta,\gamma_\delta r_j^\delta))}+C(p,q,r_j,L)\Vert f \Vert_{L^p(B(x_j^\delta,\gamma_\delta r_j^\delta))} \\
			& \hphantom{\Vert \tilde{G} \Vert_{L^q(B(x_j^\delta,\gamma_\delta r_j^\delta/2))}} \leq C(p,q,r_j,L)\Vert g \Vert_{L^q(B(x_j^\delta,r_j^\delta) \cap \Omega^\delta)}+C(p,q,r_j,L)\Vert f \Vert_{L^p(B(x_j^\delta,r_j^\delta) \cap \Omega^\delta)}.
		\end{align*}
		\textit{Global estimate.} Arguing exactly as in the last step of the proof of \cite[Theorem 2.1]{CDM} and recalling that for every $\delta \in (0,\delta_5)$ we have $\alpha_\delta \geq \alpha/2$, we infer that the following decomposition with the relative estimates hold: $\nabla u = S_0+F+G$ and
		$$
		\Vert F \Vert_{L^p(\Omega^\delta)} \leq C(p,q,r_j,L,\alpha) \Vert f \Vert_{L^p(\Omega^\delta)}, \qquad \Vert G \Vert_{L^q(\Omega^\delta)} \leq C(p,q,r_j,L,\alpha) \left(\Vert g \Vert_{L^q(\Omega^\delta)}+\Vert f \Vert_{L^p(\Omega^\delta)}\right).
		$$
		
		Finally, taking $\overline{\delta}=\min \{ \delta_1,\delta_5 \}$ we infer that for every $\delta \in (0,\overline{\delta})$ the assertion holds with an uniform constant.
	\end{proof}

	\begin{lem}\label{App: lip trunc}
		There exists a constant $c=c(\Omega)>0$ such that for every $\delta>0$ small enough, every $u \in W^{1,1}(\Omega^\delta;\R^2)$ and all $\lambda>0$, there exists a measurable set $X^\delta \subset \Omega^\delta$ such that:
		\begin{itemize}
			\item[ ] $u$ is $c \lambda- Lipschitz $ on $X$,
			\item[ ]  $\mathcal{L}^2\left(\Omega^\delta \setminus X^\delta\right) \leq \displaystyle \frac{c}{\lambda} \displaystyle \int_{\{|\nabla u|>\lambda\}} |\nabla u| \, dx$.
		\end{itemize}
	\end{lem}

	\begin{proof}
		This is a standard result in the case of a fixed Lipschitz domain. Hence, the result holds for every $u \in W^{1,1}(\Omega;\R^2)$ and all $\lambda > 0$ with a constant $c=c(\Omega)>0$. As there is a one to one correspondence between the functions in $W^{1,1}(\Omega;\R^2)$ and the functions in $W^{1,1}(\Omega^\delta;\R^2)$, we prove the result for $u \circ \Phi_\delta$. Fix $\lambda>0$, we take a set $X \subset \Omega$ such that the function $u$ is $c \, 2\lambda$-Lipschitz on $X$ and 
		$$
		\mathcal{L}^2(\Omega \setminus X) \leq \frac{c}{2\lambda} \int_{\{|\nabla u|>2\lambda\}} |\nabla u| \, dx.
		$$
		Thus, in view of Remark \ref{App: rem}, for $\delta>0$ small enough we have that $u \circ \Phi_\delta$ is $4c\lambda$-Lipschitz in $X^\delta:=\Phi_\delta(X) \subset \Omega^\delta$. Moreover,
		$$
		\mathcal{L}^2\left(\Omega^\delta \setminus X^\delta\right) \leq 2 \mathcal{L}^2(\Omega \setminus X) \leq \frac{c}{\lambda}  \int_{\{|\nabla u|>2\lambda\}} |\nabla u| \, dx \leq \frac{4c}{\lambda}  \int_{\{|\nabla (u \circ \Phi_\delta)|>\lambda\}} |\nabla (u \circ \Phi_\delta)| \, dx.
		$$
		Therefore the assertion holds with the constant $4c$.
	\end{proof}

	We finally prove Proposition \ref{App: CDM}.
	
	\begin{proof}[Proof of Proposition \ref{App: CDM}]
		We are not going to prove again the rigidity estimate with mixed growth checking the dependencies of the constants involved. We just point out that the constant appearing in the rigidity estimate with mixed growth, \cite[Theorem 1.1]{CDM}, depends only on the constant of Friesecke-James-M\"uller rigidity estimate for $L^p$ and $L^q$, on the constant of the Korn's inequality with mixed growth and finally on the constant appearing in the Lipschitz truncation Lemma (see \cite[Lemma 3.1, Theorem 1.1]{CDM}). Hence, in view of Lemmas \ref{App: FJM}, \ref{App CDM korn} and \ref{App: lip trunc} we infer that for $\delta>0$ small enough, the constant is uniform for every $\Omega^\delta$. In turn, the rigidity estimate in weak Lebesgue spaces (and in Lorentz spaces in general) is a straightforward consequence of \cite[Theorem 1.1]{CDM} (cf. \cite[Corollary 4.1]{CDM}) using the interpolation property of Lorentz spaces. Therefore, the constant appearing in Theorem \ref{thm: CDM rigidity} is the same constant of the rigidity with mixed growth conditions for some fixed $1<p_1<p_2<\infty$. Thus, we conclude.
	\end{proof}

		\section*{Acknowledgments}
This work was supported by the Austrian Science Fund through the projects ESP-61 and P35359-N, by the Italian Ministry of Education and Research through the PRIN 2022 project ``Variational Analysis of Complex Systems in Material Science, Physics and Biology'' No. 2022HKBF5C, by the project Starplus 2020 Unina Linea 1 "New challenges in the variational modeling of continuum mechanics" by the University of Naples Federico II and Compagnia di San Paolo, and by the project FRA ``Regularity and Singularity in Analysis, PDEs, and Applied Sciences'' by the University of Naples Federico II. Finally, F.S.\ is member of the Gruppo Nazionale per l'Analisi Matematica, la Probabilit\'a e le loro Applicazioni (GNAMPA-INdAM).

	\bibliographystyle{siam}
	\bibliography{bibliography_NEW}

\end{document}